\documentclass[10pt]{article}
\usepackage{mathenv,amsfonts, amsmath, amssymb,  graphics, epsfig,color}
\graphicspath{{./}}
\newtheorem{theorem}{Theorem}
\newtheorem{lemma}{Lemma}
\newtheorem{prop}{Proposition}
\newtheorem{cor}{Corollary}
\newtheorem{rem}{Remark}
\newtheorem{definition}{Definition}
\newenvironment{proof}[1][]{{\medskip\noindent\it Proof #1: }}{ \hfill$\square$\\ \\}
\numberwithin{equation}{section}
\DeclareMathOperator{\dv}{div}
\begin{document}
\author{Alexandre
Munnier\footnote{Universit\'e de Lorraine and CNRS, Institut \'Elie Cartan de Lorraine, UMR 7502, Vand\oe uvre-l\`es-Nancy, F-54506, France. \texttt{alexandre.munnier@univ-lorraine.fr}}\and Karim Ramdani\footnote{Inria, Villers-l\`es-Nancy, F-54600, France. \texttt{karim.ramdani@inria.fr}}}
\title{Asymptotic analysis of a Neumann problem in a domain with cusp. \\ Application to the collision problem of rigid bodies in a perfect fluid. }
%\date{}
%\tableofcontents
\maketitle
%===================================
\begin{abstract}
We study a two dimensional collision problem for a rigid solid immersed in a cavity filled with a perfect fluid. We are led to investigate the asymptotic behavior of the Dirichlet energy associated to the solution of a Laplace Neumann problem as the distance $\varepsilon>0$ between the solid and the cavity's bottom tends to zero. Denoting by $\alpha>0$ the tangency exponent at the contact point, we prove that the solid always reaches the cavity in finite time, but with a non zero velocity for $\alpha <2$ (real shock case), and with null velocity for $\alpha \geqslant 2$ (smooth landing case). Our proof is based on a suitable change of variables sending to infinity the cusp singularity at the contact. More precisely, for every $\varepsilon\geqslant 0$, we transform the Laplace Neumann problem into a generalized Neumann problem set on a domain containing a horizontal strip $]0,\ell_\varepsilon[\times ]0,1[$, where $\ell_\varepsilon\to +\infty$. 
\end{abstract}
\section{Introduction: motivation, problem setting and statement of the main results}
%===================================
In this paper, we are interested in the asymptotic analysis of a singular perturbed problem for the Laplace Neumann equation in a domain $\Omega_{\varepsilon}$ depending on a small parameter $\varepsilon$. The typical configuration we have in mind is the one depicted on Figure~\ref{fig1}, namely the situation where $\Omega_{\varepsilon}$ is the domain located between two smooth surfaces situated at a distance $\varepsilon$ and touching each other at one single cuspid point (the origin) in the limit case $\varepsilon=0$. 

%========================================
\begin{figure}[h]
\centering
\input{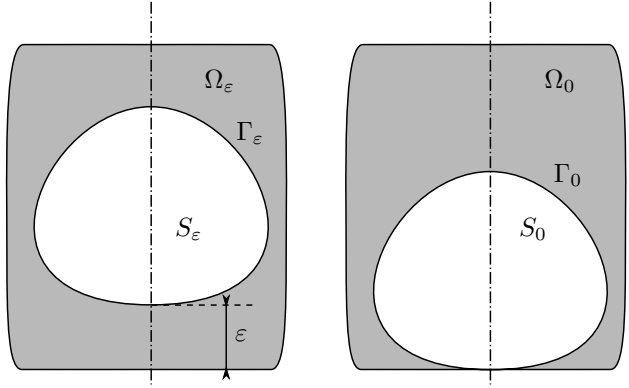}
\caption{\label{fig1} The symmetric domains $\Omega_\varepsilon$  for $\varepsilon>0$ and the singular limit domain $\Omega_0$. }
\end{figure}
%========================================

More precisely, we denote by $\mathsf{C}$ a smooth, open, bounded and connected set in $\mathbb R^2$ and we assume that $\mathsf{C}$ is symmetric with respect to the ordinate axis, that the origin belongs to the boundary of $\mathsf{C}$ and that, near the origin, the boundary $\partial\mathsf{C}$ 
is locally a straight line, the domain $\mathsf{C}$ being locally situated above $\partial\mathsf{C}$. We denote by $S_0$ a compact, connected set, symmetric with respect to the ordinate axis as well and such that, for 
some $\varepsilon^\ast>0$, the inclusion $ S_\varepsilon:= S_0+\varepsilon e_2\subset \mathsf{C}$ holds for every $0< \varepsilon\leqslant \varepsilon^\ast$ (throughout the paper, $\{e_1,e_2\}$ stands for an orthonormal basis). With these settings, we have $\Omega_\varepsilon:=\mathsf{C}\setminus  S_\varepsilon$ for every $0\leqslant \varepsilon\leqslant \varepsilon^\ast$ and we assume that there 
exists $\delta^\ast>0$ such that $\Omega_\varepsilon$ is locally, near the origin described by:
$$\{\xi:=(\xi_1,\xi_2)\in\mathbb R^2\,:\,|\xi_1|<\delta^\ast,\,0<\xi_2<H_\varepsilon(\xi_1)\},$$
where 
$$H_\varepsilon(\xi_1):=\kappa |\xi_1|^{1+\alpha}+\varepsilon,$$ 
the constants $\kappa>0$ and $\alpha>0$ (called the tangency exponent) being given.

The Neumann problem we shall consider is the following one:
\begin{subequations}
\label{Sys:intro}
\begin{alignat}{3}
-\Delta  U_\varepsilon&= F_\varepsilon&\quad&\text{in }\Omega_\varepsilon\\
\partial_n  U_\varepsilon&= G_\varepsilon&&\text{on }\Gamma_\varepsilon\\
\partial_n  U_\varepsilon&=0&&\text{on }\partial\mathsf{C},
\end{alignat}
\end{subequations}
where $\Gamma_\varepsilon:=\partial  S_\varepsilon$ denotes the boundary of the inclusion, $n$ is the unit normal to $\partial\Omega_\varepsilon$ directed toward the exterior of $\Omega_\varepsilon$ and $ F_\varepsilon$  and $ G_\varepsilon$ 
are given functions respectively defined on $\Omega_\varepsilon$ and $\Gamma_\varepsilon$ and satisfying the compatibility condition:
\begin{equation}
\label{compt_encore}
\int_{\Omega_\varepsilon} F_\varepsilon\,{\rm d}\xi+\int_{\Gamma_\varepsilon} G_\varepsilon\,{\rm d}s=0.
\end{equation}

Our main objective in this paper is to study the asymptotic behavior of $U_{\varepsilon}$ as $\varepsilon\to 0^+$, and more specifically, the behavior of the associated Dirichlet energy 
$\int_{\Omega_\varepsilon}|\nabla U_\varepsilon|^2\,{\rm d}\xi$ as $\varepsilon\to 0^+$. For simplicity, only volume data $ F_\varepsilon$ and boundary data $ G_\varepsilon$ symmetric with respect to the ordinate axis will be considered. 

Let us now describe the physical problem motivating such an asymptotic analysis. 
%================================================
\subsection{Underlying  fluid-structure contact problem}
We are interested in investigating the possibility of a collision between a neutrally buoyant rigid solid with the bottom of the bounded cavity where it is immersed. In addition to the solid, the cavity is supposed to be filled with a perfect fluid. 

Sticking to the notation of the previous  section, we denote by $\mathsf{C}$ the cavity ($\mathsf{C}$ has the same properties as in the previous section), and for every time $t>0$, by $S_t$ the domain occupied by the solid and by $\Omega_t$ the fluid domain. 

To simplify, we shall assume furthermore that at the initial time:
\begin{enumerate}
\item $ S_{t=0}= S_{\varepsilon^\ast}$ ($S_{\varepsilon^\ast}$ has the same properties as above, in particular regarding topology and symmetry).
\item \label{irrota}The flow is irrotational, which entails, according to Helmholtz's third theorem, that it will remain irrotational for every time;
\item The velocity of the solid is vertical.
\end{enumerate}
With these settings, for symmetry reason, the motion of the solid will take place along the ordinate axis only and $S_t$ and $\Omega_t$ will remain symmetric with respect to  this axis at every moment. 

In particular, the lowest point of $S_t$ (which we assume, for the time being, to be unique) has coordinates $(0,\varepsilon(t))$ and the velocity of the solid is therefore $(0,\varepsilon'(t))$ (here and subsequently, the prime denotes the time derivative).

Classically in ideal fluid theory, according to Hypothesis~\ref{irrota} above, we introduce at every time $t\geqslant 0$ the Kirchhoff potential $\varphi(t,\cdot)$ related to the vertical motion of the solid. This function solves a 
Laplace equation in $\Omega_t$ with Neumann boundary conditions, namely:
\begin{subequations}
\label{Sys:potentiel}
\begin{alignat}{3}
-\Delta \varphi(t,\cdot)&=0&\quad&\text{in }\Omega_t\\
\partial_n\varphi(t,\cdot)&=n_2&&\text{on }\Gamma_t\\
\partial_n\varphi(t,\cdot)&=0&&\text{on }\partial \mathsf{C},
\end{alignat}
\end{subequations}
where $\Gamma_t:=\partial  S_t$ and $n=(n_1,n_2)$ stands for the unit normal to $\partial\Omega_t$ directed towards the outside of the fluid. The Eulerian velocity of the fluid reads:
$$u(t,\cdot)=\varepsilon'(t)\nabla\varphi(t,\cdot)\quad\text{in }\Omega_t\qquad(t\geqslant 0).$$

Notice at this point that the domains, and thus also the potential function, depend on $t$ only through $\varepsilon(t)$. Consequently, from now on, we shall return to the notation of the previous section and we will denote by $\Omega_\varepsilon$, $ S_\varepsilon$, $ \Gamma_\varepsilon$ and 
$\varphi(\varepsilon,\cdot)$ respectively $\Omega_t$, $ S_t$, $\Gamma_t$ and $\varphi(t,\cdot)$.

The dynamics governing the motion of the solid can  now be   derived  easily from the conservation of energy of the frictionless fluid/solid system. We denote by $m_s$ the mass of the solid and by $\varrho_f$ 
the density of the fluid. Recall that the solid is assumed to be neutrally buoyant, so the 
total energy of the system reduces to the kinetic energy which reads merely
$$E(\varepsilon,\varepsilon'):=\frac{1}{2} (m_s+m_f(\varepsilon))|\varepsilon'|^2,$$
where 
\begin{equation}
\label{added_mass}
m_f(\varepsilon):=\varrho_f\int_{\Omega_\varepsilon}|\nabla\varphi(\varepsilon,\xi)|^2\,{\rm d}\xi,
\end{equation}
is the so-called added mass of the solid. Denoting by $\varepsilon_0'<0$ the initial value of  $\varepsilon'(t)$ (the initial velocity being $(0,\varepsilon'_0$)), the identity $E(\varepsilon(t),\varepsilon'(t))=E(\varepsilon^\ast,\varepsilon_0')$ for every $t>0$ leads to the following first order autonomous Cauchy problem for $\varepsilon$:
\begin{subequations}
\label{cauchy:pb}
\begin{alignat}{3}
\varepsilon'(t)&=\varepsilon_0'\sqrt{\frac{m_s+m_f(\varepsilon^\ast)}{m_s+m_f(\varepsilon(t))}},&\quad& t>0\\
\varepsilon(t)|_{t=0}&=\varepsilon^\ast>0.
\end{alignat}
\end{subequations}
It is proved in \cite{Chambrion:2012aa} in a more general context that the function 
$$\varepsilon\in ]0,\varepsilon^*]\mapsto m_f(\varepsilon)\in\mathbb R^+$$ 
is analytic, so there is no regularity issue as long as $\varepsilon(t)>0$. Actually, classical results for ODE ensure that the solution exists as long as $\varepsilon(t)>0$ (i.e. as long as the solid does not touch the boundary of the cavity). Considering the Cauchy problem \eqref{cauchy:pb}, it is clear that the asymptotic behavior 
of the solid when getting closer to the cavity's bottom relies on the asymptotic behavior of $m_f(\varepsilon)$ as $\varepsilon\to 0^+$. The following cases can occur:
\begin{enumerate}
\item The added mass $m_f(\varepsilon)$ is uniformly bounded for every $\varepsilon\geqslant 0$. It entails that $\varepsilon'$ is bounded from above by a negative constant and hence the solid 
will collide with the cavity's boundary in finite time with nonzero velocity (real shock case);
\item The added mass $m_f(\varepsilon)$ goes to $+\infty$ as $\varepsilon$ goes to 0. Depending on the strength of the blow up, two sub-cases are to be considered:
\begin{enumerate}
\item The solid reaches the boundary of the cavity in finite time with zero velocity  (``smooth landing'' case, no shock);
\item The solution to the Cauchy problem \eqref{cauchy:pb} exists for every time $t\geqslant 0$. In this case $\varepsilon(t)\to 0$ as $t\to+\infty$ (infinite time touchdown case);
\end{enumerate}
\end{enumerate}
The study of collisions between rigid solids was first addressed, to our knowledge, in \cite{Vazquez:2006aa} where the authors prove the lack of collision for a 1D model in which 
the fluid motion is governed by Burgers' equations and the solids are reduced to material points. This result has been generalized, but still for viscous fluid driven by the Navier-Stokes equations, in 2D and 3D in \cite{Hillairet:2007aa} and \cite{Hillairet:2009aa}. These studies assert that ``frontal collisions''  can not occur in a viscous fluid, contrarily to what happens in a perfect fluid. Indeed, in \cite{Houot:2008aa} the authors 
prove for a 2D model that a ball immersed in a perfect fluid can hit a wall with non zero velocity in finite time. 

In the present paper, we aim to extend this result to more general two dimensional configurations.
%========================================================================
\subsection{Back to the model problem: a singularly perturbed boundary value problem}

As already mentioned above, we will restrict our analysis to symmetric configurations (geometry, sources). For the sake of simplicity, we will use the same notation to denote the full domains $\mathsf{C}$, $S_\varepsilon$ and $ \Omega_\varepsilon$ and their intersections with the half-plane $\{\xi_1<0\}$. In addition to $\Gamma_\varepsilon:=\partial S_\varepsilon$ and $\partial {\mathsf{C}}$, the boundary $\partial\Omega_\varepsilon$ is hence from now on composed of $\Gamma_\varepsilon^b:=\{(0,\xi_2)\,:\,0<\xi_2<\varepsilon\}$ and 
$\Gamma_\varepsilon^t:=\partial\Omega_\varepsilon\setminus(\partial {\mathsf{C}}\cup \Gamma_\varepsilon\cup \Gamma_\varepsilon^b)$ (see Figure~\ref{fig1_2}). 
%========================================
\begin{figure}[h]
\centering
\input{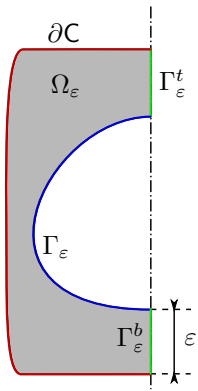}
\caption{\label{fig1_2} The new domain $\Omega_\varepsilon$ ($\varepsilon\geqslant 0$). }
\end{figure}
%========================================
The analysis of Problem~\ref{Sys:potentiel} (in the symmetric case considered here)  leads to solving the following problem set in the half cavity:
\begin{subequations}
\label{Sys:intro_2}
\begin{alignat}{3}
-\Delta  U_\varepsilon&=0&\quad&\text{in }\Omega_\varepsilon\\
\partial_n U_\varepsilon&=n\cdot e_2&&\text{on }\Gamma_\varepsilon\\
\partial_n U_\varepsilon&=0&&\text{on }\partial\Omega_\varepsilon\setminus \Gamma_\varepsilon,
\end{alignat}
\end{subequations}
which is nothing but a particular case of System~\ref{Sys:intro}, specifying $F_\varepsilon=0$ and $G_\varepsilon=n\cdot e_2$.

As already mentioned concerning the general system~\ref{Sys:intro}, the main objective of this paper is to study the convergence of $U_\varepsilon$ solution to System \eqref{Sys:intro_2}  and obtain the first order term of the asymptotics of the Dirichlet energy associated to $U_{\varepsilon}$, namely the quantity
\begin{equation}
\label{eq:Eepsilon}
E_{\varepsilon}:=\int_{\Omega_\varepsilon}|\nabla U_{\varepsilon}(\xi)|^2\,{\rm d}\xi.
\end{equation}
Notice that, up to a multiplicative constant, this quantity coincides with the added mass defined in \eqref{added_mass}.

Deriving the asymptotics of the Dirichlet energy \eqref{eq:Eepsilon} requires to solve two main difficulties: 
\begin{enumerate}
\item The solution $U_{\varepsilon}$ for $\varepsilon>0$ and the solution $U_{0}$ for $\varepsilon=0$ (if it exists) are not defined on the same domains (respectively $\Omega_{\varepsilon}$ and $\Omega_{0}$) and thus, they can not be ``compared'' in a simple way.
\item The domain $\Omega_{0}$ is strongly singular due to the presence of a cusp at the contact point. 

\end{enumerate}
Let us now formally explain our main ideas to overcome these two difficulties. The key ingredient we use is a suitable change of coordinates $x=\Psi_\varepsilon(\xi)$ defined for every $\varepsilon\geqslant 0$ (i.e. including the limit case) such that, 
denoting $\omega_{\varepsilon}:= \Psi_\varepsilon(\Omega_{\varepsilon})$, we have:
$$
\omega_\varepsilon=D\cup R_{\varepsilon}\text{ with }D\cap R_{\varepsilon}=\varnothing
$$
and where (see Figure~\ref{fig:chg_of_vari})
\begin{itemize}
\item $D$ is a fixed domain (i.e. independent of $\varepsilon\geqslant 0$);
\item $R_{\varepsilon}$ stands for the rectangle $]0,\ell_{\varepsilon}[\times ]0,1[$, where $\ell_{\varepsilon}\nearrow \ell_{0}:=+\infty$ as $\varepsilon$ goes to 0.
%\item $R$ denotes the half strip $]\delta,+\infty[\times ]0,1[$.
\end{itemize}
Denoting by $n$ the unit outer normal to $\partial\omega_\varepsilon$  and setting $\tau$ the tangent vector to $\partial\omega_\varepsilon$ such that $\tau^\perp=n$ and 
$$
u_{\varepsilon}:= U_{\varepsilon}(\Psi_{\varepsilon}^{-1}),  \qquad f_\varepsilon:=F_\varepsilon(\Psi_\varepsilon^{-1})|\det D\Psi_\varepsilon^{-1}|,\qquad
g_{\varepsilon}:=G_\varepsilon(\Psi_{\varepsilon}^{-1})|D\Psi_{\varepsilon}^{-1}\tau|, \qquad \gamma_{\varepsilon}:= \Psi_{\varepsilon}(\Gamma_{\varepsilon}),
$$
we will show that the general problem \eqref{Sys:intro} is transformed into a new boundary value problem set in $\omega_\varepsilon$:
\begin{subequations}
\label{eq:uepsilon}
\begin{alignat}{3}
-\dv(\mathbb A_{\varepsilon}\nabla u_\varepsilon)&=f_\varepsilon&\quad&\text{in }\omega_{\varepsilon}\\
\mathbb A_\varepsilon \nabla u_{\varepsilon}\cdot n&=g_{\varepsilon}&&\text{on }\gamma_{\varepsilon}\\
\mathbb A_\varepsilon \nabla u_{\varepsilon}\cdot n&=0&&\text{on }\partial \omega_{\varepsilon}\setminus\gamma_{\varepsilon}.
\end{alignat}
\end{subequations}
Notice that the compatibility condition \eqref{compt_encore} for the functions $F_\varepsilon$ and $G_\varepsilon$ yields:
$$
\int_{\omega_\varepsilon}f_\varepsilon\,{\rm d}x+\int_{\gamma_\varepsilon} g_\varepsilon\,{\rm d}s=0.$$
Here, $\mathbb A_{\varepsilon}$ denotes the 2 by 2 matrix with continuous coefficients defined by 
\begin{equation}
\label{eq:defAepsilon}
\mathbb A_\varepsilon:=\Big[(D\Psi_\varepsilon)\circ\Psi_\varepsilon^{-1}\Big]\Big[(D\Psi_\varepsilon)\circ\Psi_\varepsilon^{-1}\Big]^{\mathbf{T}}|\det(D\Psi_\varepsilon^{-1})|.
\end{equation}
The Dirichlet energy $E_{\varepsilon}$ defined by \eqref{eq:Eepsilon} takes the form
$$
E_{\varepsilon} =  \int_{\omega_\varepsilon}\mathbb A_{\varepsilon}\nabla u_{\varepsilon}\cdot \nabla u_{\varepsilon} \,{\rm d}x.
$$

%It is now time to define precisely what we understand by ``solution'' to System~\ref{eq:uepsilon}. To this purpose, we introduce:
%\begin{equation}
%V_\varepsilon:=\{v\in L^2(\omega_\varepsilon)\,:\,\nabla v\in [L^2(\omega_\varepsilon)]^2\}.
%\end{equation}
%This function space is nothing but the classical Sobolev space $H^1(\omega_\varepsilon)$ when $\varepsilon>0$. It should be equipped with a weighted scalar product when $\varepsilon=0$ for a Poincar\'e-Wirtinger 
%inequality to hold true. These questions are addressed in Section~\ref{SEC:function}.
%
%\begin{definition}[Finite energy solution]
%\label{def:solution}
%A function $u_\varepsilon\in V_\varepsilon$ ($\varepsilon\geqslant 0$) is termed a solution to System~\ref{eq:uepsilon} when:
%\begin{equation}
%\int_{\omega_\varepsilon}\mathbb A_\varepsilon\nabla u_\varepsilon\cdot \nabla v\,{\rm d}\xi=\int_{\Omega_\varepsilon} f_\varepsilon v\,{\rm d}\xi+\int_{\Gamma_\varepsilon} g_\varepsilon v\,{\rm d}\sigma,\qquad\forall\,v\in V_\varepsilon.
%\end{equation}
%\end{definition}
%
%Since, for every $\varepsilon>0$, the domain $\omega_\varepsilon$ is bounded and non-singular, the following result can be found in any PDE textbook: 
%\begin{lemma}
%For $\varepsilon>0$, $f_\varepsilon\in L^2(\omega_\varepsilon)$ and $g_\varepsilon\in L^2(\Gamma_\varepsilon)$ satisfying the compatibility condition:
%$$\int_{\omega_\varepsilon}f_\varepsilon\,{\rm d}\xi+\int_{\gamma_\varepsilon}g_\varepsilon \,{\rm d}\sigma=0,$$
%there exists solutions to problem \eqref{eq:defAepsilon}, any two of them differing only up to an additive constant. 
%\end{lemma}

Regarding the boundary value problem \eqref{eq:uepsilon}, we note the following 
\begin{enumerate}
\item In the new system of coordinates $x=(x_{1},x_{2})$, comparing the solution $u_{\varepsilon}$ and the solution $u_0$ is now possible since $(\omega_{\varepsilon})_{\varepsilon\geqslant 0}$ is an increasing sequence of domains, all of them
included in the (unbounded) domain $\omega_0$. In the new system of coordinates, the cusp singularity is sent to infinity.
\item The operator involved is not anymore the Laplace operator but the second order  operator $-\dv(\mathbb A_\varepsilon\nabla\cdot)$
(which depends on $\varepsilon\geqslant 0$). However, as we will see later, this operator is uniformly elliptic with respect to $\varepsilon\geqslant 0$ and $x\in\omega_\varepsilon$. 
\end{enumerate}

%In that case, the Definition~\ref{def:solution} should be completed as follows. 
%\begin{definition}[Infinite energy solution]
%\label{def:infinite}
%A mesurable function $u_0$ set in $\omega_0$ is termed a solution to System~\ref{eq:uepsilon} if there exists a function $u_0^s\in C^2(\overline{\omega}_0)$, $u_0^s\notin V_0$ such that System~\ref{eq:uepsilon} with the modified volume term 
%$\hat f_0:=f_0-\,{\rm div}(\mathbb A_0\nabla u_0^s)$ and the modified boundary data $\hat r_0:=g_0+\mathbb A_0\nabla u_0^s\cdot n$ admits a solution $u_0^r$ in the sense of Definition~\ref{def:solution}. In that case, we can write that:
%$$u_0=u_0^s+u_0^r,$$
%and we call $u_0^s$ the singular part of the solution while $u_0^r$ is the regular part.
%\end{definition}

Motivated by the physical problem \eqref{Sys:potentiel}, we will mainly focus on system \eqref{eq:uepsilon}, obtained from System~\ref{Sys:intro_2} after applying the change of variables, i.e. for a volume source term $f_\varepsilon=0$ and boundary data $g_\varepsilon=D\Psi^{-1}_\varepsilon\tau\cdot e_1$.
Recalling that the cusp is locally described by the equation $\xi_{2} = H_0(\xi_{1})= \kappa\, |\xi_{1}|^{1+\alpha}$ (with $\kappa, \alpha>0$), our main result can  be stated as follows:
%===================================
\begin{theorem}
\label{theo:main_theorem}
For every $\varepsilon>0$, let $u_\varepsilon$ be a solution to
\begin{subequations}
\label{eq:uepsilon_2}
\begin{alignat}{3}
-\dv(\mathbb A_{\varepsilon}\nabla u_\varepsilon)&=0&\quad&\text{in }\omega_{\varepsilon}\\
\mathbb A_\varepsilon \nabla u_{\varepsilon}\cdot n&=D\Psi^{-1}_\varepsilon\tau\cdot e_1&&\text{on }\gamma_{\varepsilon}\label{eq:ug}\\
\mathbb A_\varepsilon \nabla u_{\varepsilon}\cdot n&=0&&\text{on }\partial \omega_{\varepsilon}\setminus\gamma_{\varepsilon}.
\end{alignat}
\end{subequations}
Then the 
following alternative holds true:
\begin{enumerate}
\item \label{theo:case_1}For $\alpha<2$: System~\ref{eq:uepsilon_2} with $\varepsilon=0$ admits a finite energy solution $u_0$. Moreover, 
$$\|\nabla u_{\varepsilon}-\nabla u_0\|_{L^{2}(\omega_{\varepsilon})}\to 0\text{ and }E_{\varepsilon}\to E_{0}<\infty\text{ as }\varepsilon\to 0^+.$$ 

\item  \label{theo:case_2}For $\alpha\geqslant 2$: System~\ref{eq:uepsilon_2} with $\varepsilon=0$ has no finite energy solution and two kinds of blow up are possible as $\varepsilon \to 0^+$ for the Dirichlet energy $E_\varepsilon$:
\begin{equation}
\label{equiv_energy}
E_{\varepsilon}\underset{\varepsilon=0}{\sim}
\begin{cases} 
\frac{1}{3}\kappa^{-1}|\ln(\varepsilon)|&\mbox{if } \alpha=2 \\
\displaystyle\frac{1}{3}\varepsilon^{\frac{3}{1+\alpha}-1}\kappa^{-\frac{3}{1+\alpha}}\frac{3\pi/(1+\alpha)}{\sin(3\pi/(1+\alpha))} & \mbox{if } \alpha>2. 
\end{cases}
\end{equation}
\end{enumerate}
\end{theorem}
Let us emphasize that the behavior of the Dirichlet energy only depends on the nature of the cusp (i.e. the constants $\kappa$ and $\alpha$) and not on other geometric features of the fluid domain.

This result follows immediately from the gathering of Theorem~\ref{THEO:nice_case} and Theorem~\ref{THEO:bad_case} below.

Regarding the collision problem, we claim:
 \begin{cor}
\label{cor_intro}
In case~\ref{theo:case_1} of the Theorem ($\alpha<2$), the solid will collide with the cavity's boundary in finite time with non zero velocity (real shock case).
In case~\ref{theo:case_2} ($\alpha\geqslant 2$), the solid reaches the cavity's bottom in finite time but with null velocity (smooth landing case).
\end{cor}
This corollary is a restatement of Corollary~\ref{cor_shock} for the case $\alpha<2$. The case $\alpha\geqslant 2$ results from the following lemma (whose proof is postponed to Appendix~\ref{SEC:technic}) and the estimates \eqref{equiv_energy}.
\begin{lemma}
\label{lem_intro}
Assume that $E_\varepsilon\to +\infty$ as $\varepsilon\to 0^+$ and that there exists $\beta<2$ such that $E_\varepsilon=O(\varepsilon^{-\beta})$, then 
the solid reaches the cavity's bottom in finite time but with null velocity (smooth landing case).
\end{lemma}

%===================================
Several references can be found in the literature regarding the asymptotics of the Dirichlet problem near a tangency point of smooth components of the boundary, see for instance the papers of Maz'ya, Nazarov and Plamenevskij \cite{MazNazPla82,MazNazPla84,MazNazPla86} and their book \cite[Chapter 14]{MazNazPla00b}. The Neumann problem has been investigated more recently. In particular the singular behavior of the limit problem is derived in \cite{Naz94,NazSokTas09} while the full asymptotics with respect to the small parameter $\varepsilon$  is studied in \cite{CarNazSok07,CarNazSok09}. In particular, Cardone, Nazarov and Sokolowski provide in \cite{CarNazSok09} the first order asymptotics for the Neumann problem with thin ligaments in arbitrary dimension. However these papers deal with the case where the tangency exponent $\alpha$ of the cusp is an even integer $2m$.   

In this paper, we propose a new method to obtain the first order approximation of the solution for the two dimensional Neumann Laplacian problem for arbitrary tangency exponent $\alpha>0$. Our method relies on the use of a suitable change of variables leading to the study of Neumann problems set on the domains $\omega_{\varepsilon}$ and on a precise description of the asymptotic behavior at infinity of the solutions $u_{\varepsilon}$ of these problems when $\varepsilon$ tends to 0. Let us point out the main advantages of our approach :
\begin{enumerate}
\item Since $\omega_{\varepsilon}$ defines an increasing sequence, the solutions $u_{\varepsilon}$ and the (potential) limit solution $u_{0}$ can be easily compared on the domain $\omega_{\varepsilon}$ in which they are both defined.
\item The cases $\varepsilon>0$ and the limit case $\varepsilon=0$ can be handled exactly in the same way in the new geometry $\omega_{\varepsilon}$. Indeed, using  an appropriate weighted Sobolev space (with a decaying weight as $x_1\to \infty$), we will deal with uniformly elliptic Neumann problems on the domains $\omega_{\varepsilon}$ for all $\varepsilon\geqslant 0$. 
\item In the new variables, the strength of the cusp (i.e. the parameter $\alpha$) appears only in the operators and the boundary data of the Neumann problems and does not appear in the weight used in the functional spaces. Let us also emphasize that this makes our approach valid for arbitrary $\alpha>0$, integer or not. 
\item  Finally, let us point out that our method can be, in principle, generalized to higher dimensions. The limit problem initially set on $\Omega_{\varepsilon}\subset\mathbb R^d$ is then transformed into a problem set on $\mathbb R_{+}^{d-1}\times(0,1)$.
\end{enumerate}

\subsection{Outline}
The paper is organized as follows. In Section \ref{sect:preliminary} we collect some preliminary but elementary remarks on the asymptotic behavior of the Dirichlet energy as $\varepsilon$ tends to 0. The change of variables near the cusp and its main properties are given in Section \ref{sect:Psi}. For the sake of clarity, its full construction (near and far from the cusp) is described in Appendix~\ref{App:chg_of_vari}. The rest of the paper deals with the analysis of the general boundary problems \eqref{eq:uepsilon} set in $\omega_{\varepsilon}$ and obtained after applying the change of variables to System~\ref{Sys:intro}. In Section \ref{SEC:function}, we describe the functional framework used to study the asymptotic behavior of the solutions $u_{\varepsilon}$ of these problems as $\varepsilon$ goes to 0. We introduce appropriate weighted Sobolev spaces and we prove some useful lemmas (a trace theorem and a Poincar\'e-Wirtinger inequality involving constants which are uniform with respect to $\varepsilon$). In Section \ref{sect:abstract}, we provide a well-posedness result for the problem \eqref{eq:uepsilon} with $\varepsilon=0$, set in the unbounded domain $\omega_{0}$ and a convergence result of $u_{\varepsilon}$ (towards $u_{0}$) in the energy space for well prepared data (i.e. data having a suitable decay rate at infinity). These results are applied in Section \ref{sect:application} to investigate the asymptotic behavior of the Dirichlet energy $E_{\varepsilon}$ for the 
particular system~\ref{eq:uepsilon_2}. We show that for $\alpha<2$ (recall that $\alpha$ is the coefficient describing the strength of the cusp) and $\varepsilon=0$, the Neumann datum in \eqref{eq:ug} is well prepared. This leads to the well-posedness of the limit problem and to a finite limit energy $E_{0}$. On the contrary, for $\alpha\geqslant 2$, the boundary data in \eqref{eq:ug}  does not have the decay rate required to apply the results of Section~\ref{sect:abstract}. In this case, we prove the existence of a singular (non decaying) solution for the  problem \eqref{eq:uepsilon_2} when $\varepsilon=0$ and the blow up of the Dirichlet energy $E_{\varepsilon}$ as $\varepsilon$ tends to $0^+$. In view of the collision issue, the first term of the asymptotics of $E_{\varepsilon}$ is also given. Finally, in Section~\ref{subsect:flat}, we show through some examples how the method can be adapted to deal 
with more general configurations.
%===================================
\section{Some preliminary remarks on the asymptotic behavior}
\label{sect:preliminary}
%===================================

In order to get a first intuition about the behavior of the solution $U_{\varepsilon}$ of \eqref{Sys:intro_2} as $\varepsilon$ goes to 0, we collect here some general remarks about the problem and some comparison results obtained thanks to elementary considerations.

First of all, we recall a result proved by Nazarov {\it et al} in \cite[Section 5]{NazSokTas09}, providing a non existence result of finite energy solutions for Problem~\ref{Sys:intro_2}, when $\varepsilon=0$. The proof being short, it is given for the sake of completeness.
\begin{prop} 
\label{naza}
The variational formulation 
\begin{equation}
\label{eq:FVOmega0}
\int_{\Omega_0}\nabla U\cdot \nabla V\,{\rm d}\xi=\int_{\Gamma_0} G_0\,V\,{\rm d}\sigma, \qquad \forall\, V\in H^1(\Omega_0),
\end{equation}
where $G_0:=n\cdot e_2$ has no solution $U\in H^1(\Omega_0)$ if $\alpha\geqslant 2$.
\end{prop}
\begin{proof}
Using a contradiction argument, let us assume that there exists $U\in H^1(\Omega_0)$ satisfying the variational  formulation \eqref{eq:FVOmega0}. Given a neighborhood of the cusp $\mathcal V_0\subset \Omega_{0}$ and a function $\chi\in C^\infty_0(\mathbb R)$ with support in $]1/2,1[$ such that $\int_{\mathbb R}\chi>0$, define the sequence of test functions
$$
V_{k}(\xi) =
\begin{cases}
0&\text{if }\xi\in \Omega_0\setminus\overline{\mathcal V_0}\\
2^{\frac{k\alpha}{2}} \chi(2^k\xi_1)& \text{if }\xi\in \mathcal V_0.
\end{cases}
$$
It can be easily checked that the sequence $(V_{k})_k$ is bounded in $H^{1}(\Omega_0)$ and using the dominated convergence theorem that
$$
\lim_{k\to \infty}\int_{\Omega_0}\nabla U\cdot \nabla V_{k}\,{\rm d}\xi=0.
$$ 
On the other hand, for the right-hand side of \eqref{eq:FVOmega0}, we note that

\begin{align*}
\displaystyle \int_{\Gamma_0} G_0\,V_{k}\,{\rm d}\sigma &=\displaystyle 2^{\frac{k\alpha}{2}} \int_{\mathbb R} \chi(2^k \xi_1)\,{\rm d}\xi_1\\
&= \displaystyle 2^{k\left(\frac{\alpha}{2}-1\right)} \int_{\mathbb R} \chi(s)\,{\rm d}s,\end{align*}
which tends to $+\infty$ if $\alpha>2$ and to a non zero finite limit if $\alpha=2$, leading to a contradiction.
\end{proof}

The above result suggests that the limit energy $E_{0}$ is infinite for $\alpha \geqslant 2$. 
Using the Dirichlet principle, i.e. the identity 
\begin{equation}
\label{Dirichlet:P}
E_\varepsilon=\max_{U\in H^1(\Omega_\varepsilon)}\left\{\int_{\Gamma_\varepsilon} n_2 U\,{\rm d}\sigma-\frac{1}{2}\int_{\Omega_\varepsilon}|\nabla U|^2\,{\rm d}\xi\right\},
\end{equation}
available for every $\varepsilon>0$, we first prove the following energy blow up for $\alpha>2$:
\begin{prop}
\label{prop:estim_alpha}
For every $\alpha>2$, there exists a constant $C_\alpha>0$ such that the Dirichlet energy \eqref{eq:Eepsilon} satisfies
\begin{subequations}
\label{eq:esqtim_first}
\begin{equation}
\label{estm:blow_up}E_\varepsilon\geqslant C_\alpha \, \varepsilon^{\frac{3}{\alpha+1}-1}\quad\forall\,\varepsilon>0.
\end{equation}
In the case where the solid has locally a flat bottom (i.e. contact would occur along a segment), there exists $C_\infty>0$ such that
\begin{equation}
\label{estm:blow_up_flat}E_\varepsilon\geqslant C_{\infty} \varepsilon^{-1}\quad\forall\,\varepsilon>0.
\end{equation}
\end{subequations}
\end{prop}
\begin{rem}
Surprisingly enough, we notice by comparing with the results of Theorem~\ref{theo:main_theorem} and Proposition~\ref{prop:flat_case} that the estimates \eqref{eq:esqtim_first} are sharp. Estimates~\eqref{eq:esqtim_first} prevent the rigid body from 
colliding with the cavity's wall with non-zero velocity but do not permit to decide between the two remaining choices: ``smooth landing'' in finite time or ``infinite time touchdown''.
\end{rem}

\begin{proof}
The main idea consists in building a suitable test function in the Dirichlet principle \eqref{Dirichlet:P}. We seek this function as a piecewise polynomial. For the sake of simplicity and unless necessary, we will drop in the notation the dependence on $\varepsilon$ of the quantities introduced in the proof. 

%========================================
\begin{figure}[h]
\centering
\input{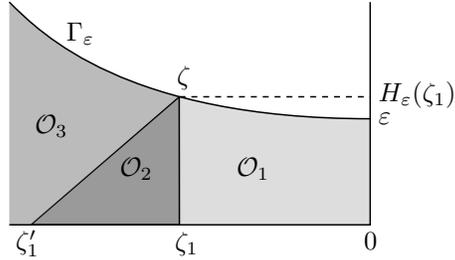}
\caption{\label{fig_proof} The partition of $\Omega_\varepsilon$ into $\mathcal O_1\cup \mathcal O_2\cup \mathcal O_3$.}
\end{figure}
%========================================

Let us begin by introducing the following partition of $\Omega_\varepsilon$. The set $\mathcal O_1$ and 
$\mathcal O_2$ are as pictured on Figure~\ref{fig_proof} and $\mathcal O_3:=\Omega_\varepsilon\setminus(\overline{\mathcal O_1\cup \mathcal O_2})$. The constants
$\zeta_1<0$ and $\zeta'_1<0$ will be specified later on. 

Denoting by $\zeta$ the point $(\zeta_1,H_\varepsilon(\zeta_1))$, we define the following polynomial functions:
$$
W_1(\xi):=-\frac{1}{2\varepsilon}(\xi_1^2-\xi_2^2)\quad\text{ and }\quad
W_2(\xi):=\frac{1}{2\varepsilon}(\xi_2+H_\varepsilon(\zeta_1))(\xi_2-h(\xi_1))+W_1(\zeta),
$$
where 
$$
h(\xi_1):=\frac{H_\varepsilon(\zeta_1)}{\zeta_1-\zeta'_1}(\xi_1-\zeta'_1).
$$
Finally, the test function to be used in \eqref{Dirichlet:P} reads: 
\begin{equation}
W(\xi):=\begin{cases}
W_1(\xi)&\text{in }\mathcal O_1\\
W_2(\xi)&\text{in }\mathcal O_2\\
W_1(\zeta)&\text{in }\mathcal O_3.
\end{cases}
\end{equation}
One can easily check that $W\in H^1(\Omega_\varepsilon)$ and that on the boundary $\Gamma_\varepsilon$ of the solid we have:
$$W(\xi)=
\begin{cases} 
W_1(\xi)&\text{on }\overline{\mathcal O_1}\cap \Gamma_\varepsilon\\
W_1(\zeta)&\text{otherwise on }\Gamma_\varepsilon.
\end{cases}
$$
Based on formula \eqref{Dirichlet:P}, we can obtain a lower bound for $E_\varepsilon$ as follows:
$$E_\varepsilon\geqslant \int_{\Gamma_\varepsilon} n_2 W\,{\rm d}\sigma-\frac{1}{2}\int_{\Omega_\varepsilon}|\nabla W|^2\,{\rm d}\xi.$$
Since $\int_{\Gamma_\varepsilon} n_2\,{\rm d}\sigma=0$ and $W$ is a constant function on $\Gamma_\varepsilon\setminus{\bar{\mathcal O}_1}$, we can rewrite the inequality above as:
\begin{equation}
\label{main_estim}
E_\varepsilon\geqslant \int_{\Gamma_\varepsilon\cap \mathcal  O_1}\hspace{-0.3cm} n_2 W_1\,{\rm d}\sigma-W_1(\zeta)\int_{\Gamma_\varepsilon\cap \mathcal O_1} \hspace{-0.3cm}n_2\,{\rm d}\sigma
-\frac{1}{2}\int_{\mathcal O_1}|\nabla W_1|^2\,{\rm d}\xi-\frac{1}{2}\int_{\mathcal O_2}|\nabla W_2|^2\,{\rm d}\xi.
\end{equation}
We can now compute explicitly every term arising in the right hand side of this estimate. We have:
$$\int_{\Gamma_\varepsilon\cap \mathcal  O_1}n_2 W_1\,{\rm d}\sigma=-\frac{1}{2\varepsilon}\int_0^{|\zeta_1|} \left[\xi_1^2-(H_\varepsilon(\xi_1))^2\right]\,{\rm d}\xi_1,$$
and then, after some elementary algebra, we get:
\begin{equation}
\label{eq:212}
\int_{\Gamma_\varepsilon\cap \mathcal  O_1}n_2 W_1\,{\rm d}\sigma=\frac{1}{\varepsilon}\left[\frac{\kappa^2|\zeta_1|^{3+2\alpha}}{6+4\alpha}-\frac{|\zeta_1|^3}{6}\right]
+\frac{\kappa|\zeta_1|^{2+\alpha}}{2+\alpha}+\varepsilon\frac{|\zeta_1|}{2}.
\end{equation}
Addressing the second term in the right hand side of \eqref{main_estim} and observing that $\int_{\Gamma_\varepsilon\cap \mathcal O_1} n_2\,{\rm d}\sigma=|\zeta_1|$, we get:
\begin{align}
-W_1(\zeta)\int_{\Gamma_\varepsilon\cap \mathcal O_1} n_2\,{\rm d}\sigma&=\frac{|\zeta_1|}{2\varepsilon}\left[\zeta_1^2-(H_\varepsilon(\zeta_1))^2\right]\nonumber\\
&=\frac{1}{2\varepsilon}\left[|\zeta_1|^3-\kappa^2|\zeta_1|^{3+2\alpha}\right]-\kappa|\zeta_1|^{2+\alpha}-\varepsilon\frac{|\zeta_1|}{2}.
\label{eq:213}
\end{align}
The third term is computed as follows:
$$\int_{\mathcal O_1}|\nabla W_1|^2\,{\rm d}\xi=\frac{1}{\varepsilon^2}\int_0^{|\zeta_1|}\int_0^{H_\varepsilon(\xi_1)}(\xi_1^2+\xi_2^2)\,{\rm d}\xi,$$
and this expression leads to:
\begin{equation}
\label{eq:214}
\int_{\mathcal O_1}|\nabla W_1|^2\,{\rm d}\xi=\frac{1}{\varepsilon^2}\left[\frac{\kappa|\zeta_1|^{4+\alpha}}{4+\alpha}+\frac{\kappa^3|\zeta_1|^{4+3\alpha}}{12+9\alpha}\right]
+\frac{1}{\varepsilon}\left[\frac{\kappa^2|\zeta_1|^{3+2\alpha}}{3+2\alpha}+\frac{|\zeta_1|^3}{3}\right]
+\frac{\kappa|\zeta_1|^{2+\alpha}}{2+\alpha}+\varepsilon\left[\frac{|\zeta_1|}{3}\right].
\end{equation}
For  the last term of \eqref{main_estim}, we have:
$$\int_{\mathcal O_2}|\nabla W_2|^2\,{\rm d}\xi=\int_{\zeta_1'}^{\zeta_1}\int_0^{H_0(\xi_1)}\left|\partial_{\xi_1}W_2(\xi)\right|^2+
\left|\partial_{\xi_2}W_2(\xi)\right|^2\,{\rm d}\xi$$
where
%\begin{align*}
%\partial_{\xi_1}W_2(\xi)&=\frac{1}{2\varepsilon}(H_\varepsilon(\zeta_1)+\xi_2)\frac{H_\varepsilon(\zeta_1)}{\zeta'_1-\zeta_1}\\
%\partial_{\xi_2}W_2(\xi)&=\frac{1}{2\varepsilon}\frac{H_\varepsilon(\zeta_1)}{\zeta'_1-\zeta_1}(\xi_1-\zeta_1)+\frac{1}{\varepsilon}\xi_2,
%\end{align*}
$$
\partial_{\xi_1}W_2(\xi)=-\frac{1}{2\varepsilon}(H_\varepsilon(\zeta_1)+\xi_2)H_0'(\xi_1)\quad\text{ and }\quad
\partial_{\xi_2}W_2(\xi)=\frac{1}{2\varepsilon}(H_\varepsilon(\zeta_1)-H_0(\xi_1))+\frac{1}{\varepsilon}\xi_2.$$
%where $$H_0(\xi_1):=\frac{H_0(\zeta)+\varepsilon}{\zeta'-\zeta}(\zeta'-\xi_1).$$
After a tedious but straightforward computation, we obtain that:
\begin{equation}
\label{eq:last}
\int_{\mathcal O_2}\left|\partial_{\xi_1}W_2(\xi)\right|\,{\rm d}\xi=\frac{11}{48}\frac{1}{\varepsilon^2}\frac{(H_\varepsilon(\zeta_1))^5}{|\zeta'_1-\zeta_1|}\quad\text{ and }\quad
\int_{\mathcal O_2}\left|\partial_{\xi_2}W_2(\xi)\right|\,{\rm d}\xi=\frac{7}{48}\frac{1}{\varepsilon^2}(H_\varepsilon(\zeta_1))^3|\zeta'_1-\zeta_1|.
\end{equation}
Now, we choose $\zeta_1=-(\varepsilon/\kappa)^{\frac{1}{\alpha+1}}$ (so that $H_\varepsilon(\zeta_1)=2\varepsilon$) and $\zeta'_1=\zeta_1-\varepsilon$. Substituting \eqref{eq:212}, \eqref{eq:213}, \eqref{eq:214}, and \eqref{eq:last} into \eqref{main_estim}, we obtain the following asymptotic expansion: 
\begin{multline*}
\int_{\Gamma_\varepsilon\cap \mathcal  O_1} n_2 W_1\,{\rm d}\sigma-W_1(\zeta)\int_{\Gamma_\varepsilon\cap \mathcal O_1} n_2\,{\rm d}\sigma
-\frac{1}{2}\int_{\mathcal O_1}|\nabla W_1|^2\,{\rm d}\xi-\frac{1}{2}\int_{\mathcal O_2}|\nabla W_2|^2\,{\rm d}\xi\\
=\left(\frac{\alpha+1}{6\alpha+24}\right)\kappa^{-\frac{3}{\alpha+1}}\varepsilon^{\frac{3}{\alpha+1}-1}+
o\left(\varepsilon^{\frac{3}{\alpha+1}-1}\right).\hspace{3cm}
\end{multline*}
For $\varepsilon$ small enough, we get \eqref{estm:blow_up}.
For a solid with a flat bottom, it suffices to replace $H_0$ by $0$ in all the estimates and $\zeta_1$ by a small constant (such that $H_0=0$ on $]\zeta_1,0[$), to get the claimed result \eqref{estm:blow_up_flat}. The proof is now complete.
\end{proof}
The two following propositions allow comparing the Dirichlet energy after simple changes in the geometry.
\begin{prop}
\label{prop:comp}
Let us denote by $E^{[{\mathsf{C}},S_0]}_\varepsilon$ the Dirichlet energy corresponding to a solid of shape $S_0$ in a cavity ${\mathsf{C}}$. If ${\mathsf{C}}^1\subset {\mathsf{C}}^2$ then 
\begin{equation}
\label{estm:includ}
E^{[{\mathsf{C}}^1,S_0]}_\varepsilon\geqslant E^{[{\mathsf{C}}^2,S_0]}_\varepsilon\quad\forall\,\varepsilon>0.
\end{equation}
\end{prop}
In other words, this Proposition asserts that the bigger the cavity is, the lower is the Dirichlet energy. 

\begin{proof}
For every $\varepsilon>0$, we set $\Omega_\varepsilon^k:={\mathsf{C}}^k\setminus S_\varepsilon$ ($k=1,2$).
For every function $U\in H^1(\Omega^2_\varepsilon)$, its restriction to $\Omega^1_\varepsilon$ belongs to $H^1(\Omega^1_\varepsilon)$ and we have the obvious inequality:
$$\int_{\Gamma_\varepsilon} n_2 U\,{\rm d}\sigma-\frac{1}{2}\int_{\Omega^2_\varepsilon}|\nabla U|^2\,{\rm d}\xi
\leqslant\int_{\Gamma_\varepsilon} n_2 U\,{\rm d}\sigma-\frac{1}{2}\int_{\Omega^1_\varepsilon}|\nabla U|^2\,{\rm d}\xi.$$
The conclusion follows then from the Dirichlet principle \eqref{Dirichlet:P}.
\end{proof}
The next Proposition tells us that the Dirichlet energy can be compared for configurations that are images one from the other by a global $C^1$ diffeormorphism. 
\begin{prop}
\label{prop:diffeo}
Let $B$ be a large ball containing  a cavity ${\mathsf{C}}$. For every $0\leqslant \varepsilon\leqslant \varepsilon^\ast$, consider the usual configuration involving a solid $S_0$, its domain $S_\varepsilon$, its boundary $\Gamma_\varepsilon:=\partial S_\varepsilon$ and the fluid domain $\Omega_\varepsilon:={\mathsf{C}}\setminus S_\varepsilon$. 

Let $\tilde H_0:\mathbb R\to\mathbb R$ be a given $C^1$ function. Let $\mathcal U$ be an open set containing $\Gamma_\varepsilon$ for every $\varepsilon\geqslant 0$ small enough (see Figure~\ref{fig_diffeo}). 

For every $C^1$ diffeomorphism $\Phi:B\to B$ such that
\begin{equation}
\label{assump}
\Phi(\xi_1,\xi_2)=(\xi_1,\xi_2+\tilde H_0(\xi_1)),\quad\forall\,\xi=(\xi_1,\xi_2)\in\mathcal U,
\end{equation}
the following estimate holds true (for every $\varepsilon> 0$ small enough):
$$c_1 E^{[{\mathsf{C}},S_0]}_\varepsilon\leqslant E^{[\tilde{\mathsf{C}},\tilde S_0]}_\varepsilon\leqslant c_2 E^{[{\mathsf{C}},S_0]}_\varepsilon,$$
where $\tilde{\mathsf{C}}:=\Phi({\mathsf{C}})$, $\tilde S_0:=\Phi(S_0)$ and $c_1>0$ and $c_2>0$ are two constants depending only on ${\mathsf{C}}$, $S_0$ and $\Phi$.
\end{prop}
Notice that the assumption \eqref{assump} entails that
$$\Phi(S_\varepsilon)=\tilde S_\varepsilon\quad\text{ and }\quad\Phi(\Omega_\varepsilon)=\tilde\Omega_\varepsilon,$$
for every $\varepsilon\geqslant0$ small enough, where $\Omega_\varepsilon:={\mathsf{C}}\setminus S_\varepsilon$ and $\tilde \Omega_\varepsilon:=\tilde{\mathsf{C}}\setminus \tilde S_\varepsilon$ and that the local parameterization of the fluid 
domain  $\tilde\Omega_{\varepsilon}$ near the origin is now given by
$$
\{\xi\in\mathbb R^2\,:\,|\xi_1|<\delta^\ast,\,\tilde H_0(\xi_1)<\xi_{2}<H_{\varepsilon}(\xi_{1})+ \tilde H_{0}(\xi_{1})\}.
$$
%========================================
\begin{figure}[h]
\centering
\input{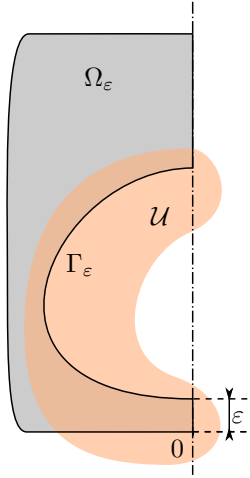}
\caption{\label{fig_diffeo} The open set $\mathcal U$ containing the solid's boundary for every $\varepsilon\geqslant 0$.}
\end{figure}
%========================================

\begin{proof}
Let $\phi:[0,1]\to \Gamma_0$ be a parameterization of the boundary of the solid when $\varepsilon=0$. Then $\phi_\varepsilon=\phi+\varepsilon e_2$ is a parameterization of $\Gamma_\varepsilon$ and 
$\Phi\circ\phi_\varepsilon$ a parameterization of $\tilde\Gamma_\varepsilon:=\Phi(\Gamma_\varepsilon)=\partial \tilde S_\varepsilon$ for every $\varepsilon\geqslant 0$. 

For every $\varepsilon>0$ and every $v\in H^1(\tilde\Omega_\varepsilon)$, we have:
\begin{align*}
\int_{\tilde\Gamma_\varepsilon} n_2\, v\,{\rm d}\sigma&=\int_0^1 [(D\Phi(\phi(s))\phi'(s))^\perp\cdot e_2] v(\Phi(\phi(s))\,{\rm d}s\\
&=-\int_0^1 [\phi'(s)\cdot D\Phi(\phi(s))^{\mathbf{T}} e_1] v(\Phi(\phi(s))\,{\rm d}s.
\end{align*}
From assumption \eqref{assump}, we infer that $D\Phi(\phi(s))^{\mathbf{T}} e_1=e_1$ and therefore:
\begin{equation}
\label{eq:vtilde:1}
\int_{\tilde\Gamma_\varepsilon} n_2\, v\,{\rm d}\sigma=\int_{\Gamma_\varepsilon} n_2\, \tilde v\,{\rm d}\sigma,
\end{equation}
where $\tilde v:=v\circ\Phi$. 

On the other hand, we have upon the change of variables $x=\Phi(\xi)$ the identity:
\begin{equation}
\label{eq:vtilde:2}
\int_{\tilde\Omega_\varepsilon}|\nabla v|^2\,{\rm d}\xi=\int_{\Omega_\varepsilon}\mathbb A\nabla\tilde v\cdot\nabla\tilde v\,{\rm d}x,
\end{equation}
where $\mathbb A:=(D\Phi)^{-1}(D\Phi)^{-{\mathbf{T}}}|\det D\Phi|$. The matrix $\mathbb A$ is positive-definite and hence there exist two positive constants $\lambda_1$ and $\lambda_2$  such that $$\lambda_1|X|^2\leqslant \mathbb A(\xi)X\cdot X\leqslant \lambda_2|X|^2,$$
for every $X\in\mathbb R^2$ and every $\xi\in \overline{\mathsf C}$. Without loss of generality, we can assume that $\lambda_{1}\leqslant 1 \leqslant \lambda_{2}$. Gathering \eqref{eq:vtilde:1} and \eqref{eq:vtilde:2}, we get:
$$
\int_{\tilde\Gamma_\varepsilon} n_2\, v\,{\rm d}\sigma-\frac{1}{2}\int_{\tilde\Omega_\varepsilon}|\nabla v|^2\,{\rm d}\xi=
\int_{\Gamma_\varepsilon} n_2\, \tilde v\,{\rm d}\sigma-\frac{1}{2}\int_{\Omega_\varepsilon}\mathbb A\nabla\tilde v\cdot\nabla\tilde v\,{\rm d}x
$$
and then, since $\lambda_{1}\leqslant 1$, according to \eqref{Dirichlet:P} we get
\begin{equation}
\label{eq:conclud}
E_\varepsilon^{[{\mathsf C},S_0]}\leqslant \lambda_1 E_\varepsilon^{[\tilde{\mathsf C},\tilde S_0]}.
\end{equation}
Remarking that $\Phi^{-1}$  enjoys the properties required for $\Phi$ to get \eqref{eq:conclud}, we deduce that we also have:
$$E_\varepsilon^{[\tilde{\mathsf C},\tilde S_0]}\leqslant \lambda_2^{-1}E_\varepsilon^{[{\mathsf C},S_0]}.$$
The claim of the Proposition follows.
\end{proof}

Typical illustrations of the above result are given in Figures~\ref{fig3} and \ref{fig_pic}.
%========================================
\begin{figure}[h]
\centering
\input{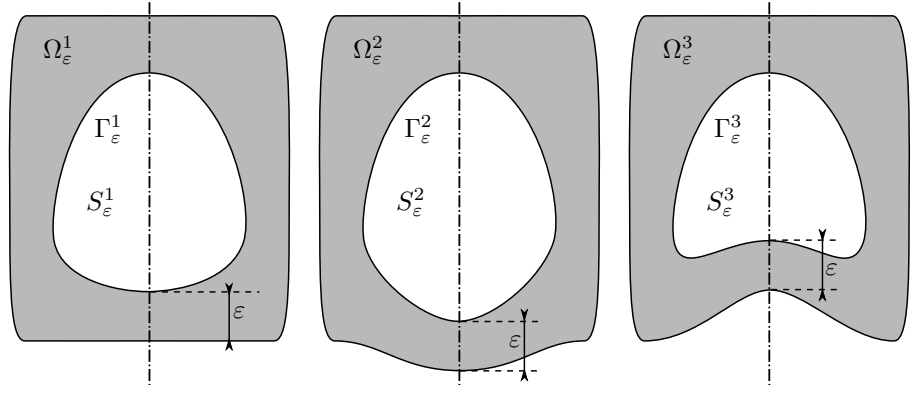}
\caption{\label{fig3} According to Proposition~\ref{prop:diffeo}, the Dirichlet energy behaves similarly as $\varepsilon\to 0^+$ for all of these cases.}
\end{figure}
%========================================
%========================================
\begin{figure}[h]
\centering
\input{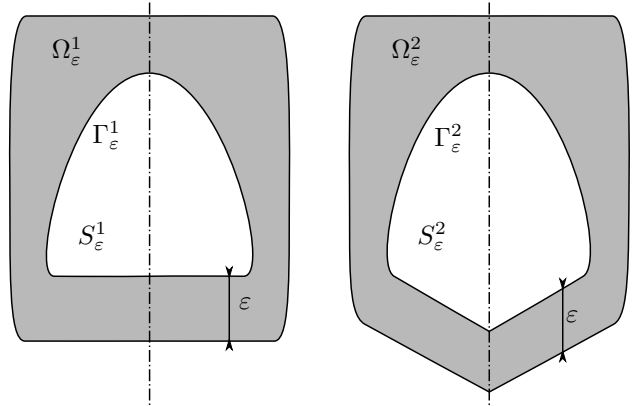}
\caption{\label{fig_pic} Another example of two configurations where, according to Proposition \ref{prop:diffeo}, the Dirichlet energy can be compared as $\varepsilon\to 0^+$ (notice on this example how
we take advantage of working with a half configuration and then recover a full configuration by symmetry).}
\end{figure}
%========================================

\medskip
\noindent{\bf Application. }Combining Propositions~\ref{prop:estim_alpha}, \ref{prop:comp}  and \ref{prop:diffeo}, we can deduce an estimate for the case where the bottom of the solid
is concave (see Figure~\ref{fig2}, on the right) and where there are two contact points for $\varepsilon=0$.  Indeed, with the notation of Figure~\ref{fig2}, according to Proposition~\ref{prop:comp}, for every $\varepsilon>0$ we have 
$$E^{[{\mathsf C}^1,S^1_0]}_\varepsilon\geqslant E^{[{\mathsf C}^2,S^2_0]}_\varepsilon$$ 
and according to Proposition~\ref{prop:diffeo}, there exists a constant $c>0$ such that
$$c\,E^{[{\mathsf C}^2,S^2_0]}_\varepsilon\geqslant E^{[{\mathsf C}^3,S^3_0]}_\varepsilon.$$
Using now Proposition~\ref{prop:estim_alpha}, we infer the existence of a constant $C>0$ such that:
$$C\varepsilon^{-1}\geqslant E^{[{\mathsf C}^3,S^3_0]}_\varepsilon.$$
Consequently, the energy blow up is no greater in case 3 than in case 1.
%========================================
\begin{figure}[h]
\centering
\input{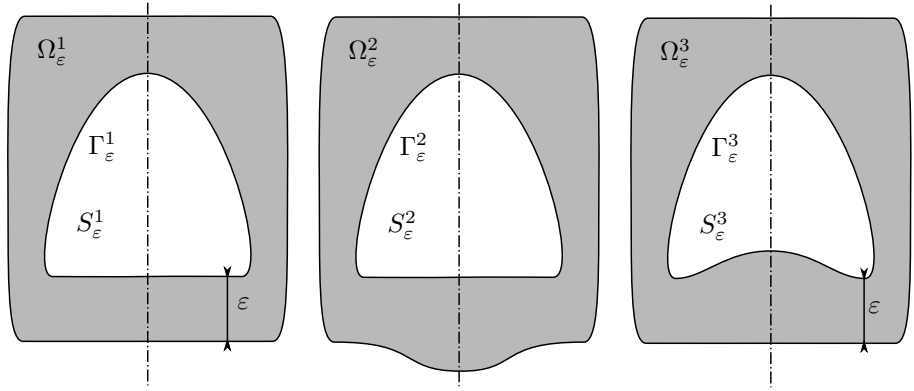}
\caption{\label{fig2} The energy blow-up in these 3 cases can be compared thanks to Proposition~\ref{prop:comp} and \ref{prop:diffeo}.}
\end{figure}
%========================================
%=========================================
\section{From the physical domain to the semi-infinite strip}
%=========================================
\label{sect:Psi}
%-------------------------------------------------------------------
%\subsection{The change of variables}
%-------------------------------------------------------------------
In this section, we describe the change of coordinates $x=\Psi_\varepsilon(\xi)$, $0\leqslant\varepsilon\leqslant \varepsilon^\ast$, used in the sequel to transform the Laplace Neumann problem \eqref{eq:uepsilon} set on $\Omega_{\varepsilon}$ into an elliptic Neumann problem set on  
$\omega_{\varepsilon}:= \Psi_\varepsilon(\Omega_{\varepsilon})=D\cup R_{\varepsilon}$ (see Figure~\ref{fig:chg_of_vari}), where $D$ is a fixed domain and $R_{\varepsilon}=]0,\ell_{\varepsilon}[\times ]0,1[$. Our change of variables is a generalization to the case $\varepsilon>0$ of the one introduced by Ibuki \cite{Ibu74} and used later by Grisvard in \cite{Gri95} and Acosta {\it et al.} in \cite{AcoArmDur05} to study the well-posedness and the regularity of Laplace problems in domains with cusps (in other words this corresponds in our problem to the limit case $\varepsilon=0$). The full description of the diffeomorphism $\Psi_\varepsilon\in C^1(\overline{\Omega}_{\varepsilon},\overline{\omega}_{\varepsilon})$ is given in Appendix~\ref{App:chg_of_vari}. For the sake of clarity, we only give here its definition on some neighborhood of the contact region. More precisely, for every $\varepsilon\geqslant0$ and given $\delta<0$ small enough, let 
\begin{equation}
\label{def:mathcalV}
\mathcal V_\varepsilon:=\{\xi\in\mathbb R^2\,:\delta<\xi_1<0,\,0<\xi_2<H_\varepsilon(\xi_1)\}\subset \Omega_\varepsilon.
\end{equation}
Then, we set in $ \mathcal V_\varepsilon$:
\begin{equation}
\label{def:Psi_on_V}
\Psi_\varepsilon(\xi)=\begin{pmatrix} 
\rho_\varepsilon(\xi_1)\\
\dfrac{\xi_2}{H_\varepsilon(\xi_1)}
\end{pmatrix},
\qquad \forall\xi\in \mathcal V_\varepsilon,
\end{equation}
where the function $\rho_\varepsilon:[\delta,0[\to \mathbb R^+$ is given by
\begin{equation}
\label{def_rho}
\rho_\varepsilon(\xi_1):=\int_\delta^{\xi_1}\frac{\,{\rm d}s}{H_\varepsilon(s)}.
\end{equation}
Introducing
$$\ell_\varepsilon:=\lim_{\xi_1\to 0^+}\rho_\varepsilon(\xi_1),$$ 
We note that $\ell_\varepsilon<+\infty$ for $\varepsilon>0$ and $\ell_0:=+\infty$. More precisely, based on the identity:
$$\int_0^{+\infty}\frac{{\rm d}s}{s^{1+\alpha}+1}=\frac{\pi/(\alpha+1)}{\sin(\pi/(\alpha+1))},$$ 
we can easily verify that:
\begin{equation}
\label{comp_l_eps}
\ell_\varepsilon\underset{\varepsilon=0}{\sim}\varepsilon^{-\frac{\alpha}{\alpha+1}}\kappa^{-\frac{1}{\alpha+1}}\left(\frac{\pi/(\alpha+1)}{\sin(\pi/(\alpha+1))}\right).
\end{equation}
Setting $R_\varepsilon:=\Psi_\varepsilon(\mathcal V_\varepsilon)$, we get:
\begin{equation}
\label{def_R}
R_\varepsilon=]0,\ell_\varepsilon[\times ]0,1[.
\end{equation}
We define the reciprocal function to $\rho_\varepsilon$ as being:
\begin{equation}
\label{def_mu}
\mu_\varepsilon:=\rho_\varepsilon^{-1}\,:\,[0,\ell_\varepsilon[\to [\delta,0[.
\end{equation}
When $\varepsilon=0$, the function $\mu_\varepsilon$ can be made explicit and we have:
\begin{equation}
\label{def_mu0}
\mu_0(x_1)=-(\alpha\kappa)^{-\frac{1}{\alpha}}(x_1+{\hat x}_1)^{-\frac{1}{\alpha}},
\end{equation}
where ${\hat x}_1=(\alpha\kappa)^{-1}|\delta|^{-\alpha}$. For all $\varepsilon\geqslant 0$, the function $\Psi_\varepsilon^{-1}$ admits the following expression in $R_\varepsilon$:
\begin{equation}
\label{def:Psi-1}
\Psi_\varepsilon^{-1}(x)=
\begin{pmatrix}
\mu_\varepsilon(x_1)\\
x_2H_\varepsilon(\mu_\varepsilon(x_1))
\end{pmatrix},
\qquad \forall x\in R_\varepsilon.
\end{equation}
We collect, in the following Lemma, some properties of the function $\mu_\varepsilon$ ($\varepsilon\geqslant 0$) that will be useful in the sequel:
%-------------------------------------------------------------------
\begin{lemma}
\label{LEM:proper_mu}
\begin{enumerate}
\item The following uniform convergence result holds true:
\begin{equation}
\label{unif_conv}
\|\mu_\varepsilon-\mu_0\|_{C^0([0,\ell_\varepsilon])}\to 0\text{ as }\varepsilon\to 0^+.
\end{equation}
\item There exist three positive constants $C_1$, $C_2$ and $C_3$, depending on $\alpha$, $\kappa$ and $\delta$ only, such that, for every $\varepsilon\geqslant 0$ and every $x_1\in[0,\ell_\varepsilon[$:
\begin{subequations}
\label{estim_phi_esp}
\begin{align}
|\mu_\varepsilon(x_1)|&\leqslant C_1(1+x_1)^{-\frac{1}{\alpha}}\label{estim_phi_esp:0}\\
|H_\varepsilon(\mu_\varepsilon)|&\leqslant C_2(1+x_1)^{-1-\frac{1}{\alpha}}\label{estim_phi_esp:1}\\
|{H'_0}(\mu_\varepsilon)|&\leqslant C_2(1+x_1)^{-1}.\label{estim_phi_esp:2}
\end{align}
\end{subequations}
\end{enumerate}
\end{lemma}
%-------------------------------------------------------------------
The proof in postponed to Appendix~\ref{SEC:technic}.

In Appendix~\ref{App:chg_of_vari}, we explain how to define $\Psi_\varepsilon$ in $\Omega_\varepsilon\setminus{\mathcal V}_\varepsilon$ in such 
a way that  $\Psi_\varepsilon( \Omega_\varepsilon\setminus{\mathcal V}_\varepsilon)$ defines a domain $D$ which is independent of $\varepsilon$ (see Figure~\ref{fig:chg_of_vari}).

\begin{figure}[h]
\centering
\input{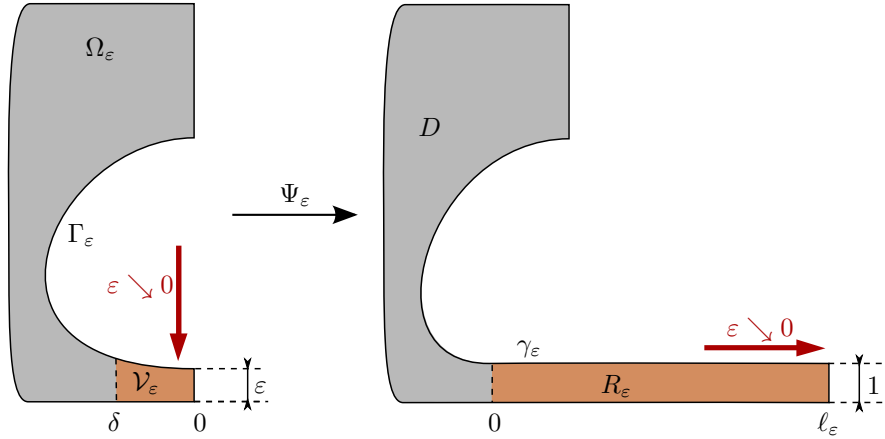}
\caption{\label{fig:chg_of_vari}The domain $\Omega_\varepsilon$  and its image $\omega_\varepsilon$ by $\Psi_\varepsilon$. In particular, $D:=\Psi_\varepsilon( \Omega_\varepsilon\setminus{\mathcal V}_\varepsilon)$ 
does not depend on $\varepsilon\geqslant 0$.}
\end{figure}

Notice that $\ell_\varepsilon\nearrow+\infty$ ($= \ell_{0}$) as $\varepsilon\searrow 0^+$ and therefore that we have the nice inclusion properties:
$$\varepsilon'>\varepsilon\,\Rightarrow\, \omega_{\varepsilon'}\subset \omega_{\varepsilon}\qquad\,\forall\,\varepsilon,\,\varepsilon'\geqslant 0.$$

%The following Lemma, whose proof is postponed to Appendix~\ref{App:chg_of_vari}, will be useful in the sequel:
%\begin{lemma}
%\label{lem:first_estim}
%For every $0<\alpha<+\infty$, we have:
%\begin{subequations}
%\begin{align}
%\ell_\varepsilon&\underset{\varepsilon=0}{\sim}\varepsilon^{-\frac{\alpha}{\alpha+1}}\kappa^{-\frac{1}{\alpha+1}}\left(\frac{\pi/(\alpha+1)}{\sin(\pi/(\alpha+1))}\right)\\
%%\item
%%$\displaystyle\rho_0(\xi_1)\underset{\xi_1=0}{\sim}\frac{|\xi_1|^{-(\alpha-1)}}{\kappa(\alpha-1)}$
%\displaystyle\mu_0(x_1)&\underset{x_1=+\infty}{\sim}-\left(\alpha^{-\frac{1}{\alpha}}\kappa^{-\frac{1}{\alpha}}\right)x_1^{-\frac{1}{\alpha}}\label{estim_mu_0}\\
%\displaystyle\mu_0(\ell_\varepsilon)&\underset{\varepsilon=0}{\sim}-\varepsilon^{\frac{1}{\alpha+1}}
%\kappa^{-\frac{1}{\alpha+1}}\left[\alpha\frac{\pi/(\alpha+1)}{\sin(\pi/(\alpha+1))}\right]^{-\frac{1}{\alpha}}.\label{estim_mu_lesp}
%\end{align}
%\end{subequations}
%\end{lemma}

\begin{prop}
\label{prop:chgt_variables}
The following convergence property holds:
\begin{subequations}
\begin{equation}
\|\Psi_\varepsilon^{-1}-\Psi_{0}^{-1}\|_{C^1(\overline{\omega}_\varepsilon)}\to 0\text{ as }\varepsilon\to 0^+.
\end{equation}
Moreover, if $\alpha>1$, we also have:
\begin{equation}
\|\Psi_\varepsilon^{-1}-\Psi_{0}^{-1}\|_{C^2(\overline{R}_\varepsilon)}\to 0\text{ as }\varepsilon\to 0^+.
\end{equation}
\end{subequations}
\end{prop}
\begin{proof}
We focus on the convergence on the rectangle $R_\varepsilon$, the rest the proof being given in Appendix~\ref{App:chg_of_vari}.

For all $\varepsilon\geqslant 0$, recall that the expression of the function $\Psi_\varepsilon^{-1}$ in $R_\varepsilon$ is given in \eqref{def:Psi-1}.

We have $\mu_\varepsilon'=H_\varepsilon(\mu_\varepsilon)$,  
$\mu''_\varepsilon=H_0'(\mu_\varepsilon)H_\varepsilon(\mu_\varepsilon)$ and $\mu_\varepsilon'''=H_0''(\mu_\varepsilon)H_\varepsilon(\mu_\varepsilon)^2+H_0'(\mu_\varepsilon)^2H_\varepsilon(\mu_\varepsilon)$ for all $\varepsilon\geqslant 0$. 
Since the functions $H_0$ and $H'_0$ are bounded and uniformly continuous on the compact $[\delta,0]$ (because $\alpha>0$), and the same holds true for $H''_0$ if $\alpha>1$, the conclusion follows from \eqref{unif_conv}.
\end{proof}
%-------------------------------------------------------------------
%\subsection{The Neumann problems in the strip}
%-------------------------------------------------------------------
We can now make explicit the matrix $\mathbb A_\varepsilon$, arising in the statement of the Neumann problem~\eqref{eq:uepsilon} (at least in the rectangle $R_\varepsilon$), based on formulas \eqref{eq:defAepsilon} and \eqref{def:Psi-1}:
\begin{equation}
\mathbb A_\varepsilon(x)=
{\rm Id}+x_2 H_0'(\mu_\varepsilon(x_1))
\begin{pmatrix}
0&-1\\
-1&x_2H_0'(\mu_\varepsilon(x_1)
\end{pmatrix}\qquad (\varepsilon\geqslant 0,\,x\in R_\varepsilon).
\end{equation} 
We claim:
%Similarly, for $\varepsilon=0$, the definition \eqref{eq:defA} yields:
%\begin{equation}
%\mathbb A(x)=
%{\rm Id}+x_2 H_0'(\mu_0(x_1))
%\begin{pmatrix}
%0&-1\\
%-1&x_2H_0'(\mu_0(x_1)
%\end{pmatrix},\quad\forall\,x\in R.
%\end{equation} 
\begin{lemma}
\label{lem:ellipticiteuniform}
\begin{enumerate}
\item The following convergence result holds true: 
\begin{subequations}
\label{conv:Aeps_gene}
\begin{equation}
\label{conv:Aeps}
\|\mathbb A_\varepsilon-\mathbb A_0\|_{C^0(\overline{\omega}_\varepsilon)}\to 0\text{ as }\varepsilon\to 0^+.
\end{equation}
If $\alpha>1$, we also have:
\begin{equation}
\label{conv:Aeps1}
\|\mathbb A_\varepsilon-\mathbb A_0\|_{C^1(\overline{R}_\varepsilon)}\to 0\text{ as }\varepsilon\to 0^+.
\end{equation}
\end{subequations}
\item
There exist two constants 
$0<\lambda_1<\lambda_2$, independent of $\varepsilon\geqslant 0$, such that:
\begin{equation}
\label{estim_A}
\lambda_1 |X|^2\leqslant \mathbb A_\varepsilon(x) X\cdot  X\leqslant\lambda_2 |X|^2, \qquad \forall\,   X\in \mathbb R^2, \ \forall\, x\in \omega_\varepsilon.
\end{equation}
\end{enumerate}
\end{lemma}
\begin{proof}
%Recall that the general expression of the matrix $\mathbb A_\varepsilon$ is:
%$$\mathbb A_\varepsilon:=\Big[(D\Psi_\varepsilon)\circ\Psi_\varepsilon^{-1}\Big]\Big[(D\Psi_\varepsilon)\circ\Psi_\varepsilon^{-1}\Big]^{\mathbf{T}}|\det(D\Psi_\varepsilon^{-1})|,$$
%while for $\varepsilon=0$:
%$$
%\mathbb A:=\Big[(D\Psi)\circ\Psi^{-1}\Big]\Big[(D\Psi)\circ\Psi^{-1}\Big]^{\mathbf{T}}|\det(D\Psi^{-1})|.
%$$
The convergences \eqref{conv:Aeps_gene} are a straightforward consequence of Proposition~\ref{prop:chgt_variables}.

The definition \eqref{eq:defAepsilon} of $\mathbb A_\varepsilon$ ($\varepsilon\geqslant 0$) entails that $\mathbb A_\varepsilon(x)$ is positive-definite for every $\varepsilon\geqslant 0$ and every $x\in\omega_\varepsilon$.
Since the eigenvalues depend continuously on the matrix, it suffices to  prove \eqref{estim_A} for $\varepsilon=0$ to get the conclusion of the Lemma. We would be done if $\overline{\omega}_{0}$ were a compact. The way out consists in  
computing the expression of the eigenvalues of $\mathbb A_{0}$ in $R_{0}$. Indeed, we get:
$$
\lambda_j(x)=F_j(|x_2H_0'(\mu_\varepsilon(x_1))|),\quad j=1,2
$$
where
$$
F_1(X):=1+\frac{1}{2}X\left[X-\sqrt{X^2+4}\right],\quad F_2(X):=1+\frac{1}{2}X\left[X+\sqrt{X^2+4}\right],$$
and $F_1(X)\geqslant \lambda_1:=F_1((\alpha+1)\kappa|\delta|^\alpha)>0$ and $F_2(X)\leqslant \lambda_2:=F_2((\alpha+1)\kappa|\delta|^\alpha)$ for every $X\in [0,(\alpha+1)\kappa|\delta|^\alpha]$.
The proof is now complete.
\end{proof}

%===================================
\section{Functional framework}
\label{SEC:function}
%==================================

The domain $\omega_{\varepsilon}$ being bounded for $\varepsilon>0$ and $-\dv(\mathbb A_\varepsilon\nabla\cdot)$ being an elliptic operator, problem \eqref{eq:uepsilon} is  a well-posed Neumann problem for every $\varepsilon>0$, the solution being uniquely defined in $H^{1}(\omega_{\varepsilon})$, up to an additive constant (and provided the compatibility condition is satisfied). In order to study the well-posedness of this system when $\varepsilon=0$, we need to introduce a suitable functional framework since the domain $\omega_{0}$ is infinite in the $x_{1}$ direction. More precisely, for every $\varepsilon\geqslant 0$ and every $\beta\in\mathbb R$, we introduce on $\omega_{\varepsilon}$ and $\partial\omega_\varepsilon$ respectively the measures
$$
{\rm d}\nu_\beta(x):=\begin{cases}(1+x_1)^{\beta}\,{\rm d}x&\text{if }x\in R_\varepsilon\\
{\rm d}x&\text{if }x\in D.
\end{cases}\quad\text{and}\quad
{\rm d}\nu^S_\beta(x):=\begin{cases}(1+x_1)^{\beta}\,{\rm d}x_1&\text{if }x\in \gamma^R_\varepsilon\\
{\rm d}s&\text{if }x\in \gamma^D_\varepsilon,
\end{cases}
$$
where $\gamma^R_\varepsilon:=\{(x_1,1)\,:\, 0<x_1<\ell_\varepsilon\}$ and $\gamma_\varepsilon^D:=\gamma_\varepsilon\setminus \gamma^R_\varepsilon$.

Since the cases $\beta=-2$ and $\beta=2$ will play a particular role, we set
$${\rm d}m={\rm d}\nu_{-2},\quad{\rm d}m^{-1}={\rm d}\nu_{2},\quad{\rm d}\sigma={\rm d}\nu^S_{-2}\quad\text{ and }\quad{\rm d}\sigma^{-1}={\rm d}\nu^S_{2}.$$
Throughout, $L^1(\omega_{\varepsilon},\,{\rm d}\nu_\beta)$, $L^2(\omega_{\varepsilon},\,{\rm d}\nu_\beta)$, $L^1(\gamma_{\varepsilon},\,{\rm d}\nu_\beta^S)$ and 
$L^2(\gamma_{\varepsilon},\,{\rm d}\nu_\beta^S)$ stand for the Lebesgue spaces of integrable and square-integrable functions respectively for the measures ${\rm d}\nu_\beta$ and ${\rm d}\nu_\beta^S$.

For every $\varepsilon\geqslant 0$ and every $\beta\in\mathbb R$, we define the weighted Sobolev spaces:
$$
H^1(\omega_{\varepsilon},\,{\rm d}\nu_\beta):=\left\{u\in L^2(\omega_{\varepsilon},{\rm d}\nu_\beta)\,:\,\partial_{x_i}u\in L^2(\omega_{\varepsilon}),\,i=1,2\right\}.$$
In the particular case $\beta=2$, we also set:
\begin{equation}
\label{eq:defH1N}
H^1_N(\omega_{\varepsilon},{\rm d}m)=\left\{u\in H^1(\omega_{\varepsilon},{\rm d}m)\,:\,\int_{\omega_{\varepsilon}} u\,{\rm d}m=0\right\}.
\end{equation}
This space is well defined (regarding  the $L^1$ condition for $\varepsilon=0$), as it will be verified in Lemma~\ref{lem:conv_l1} below.

Since $\omega_{\varepsilon}$ is bounded  for $\varepsilon>0$, the space $H^1(\omega_{\varepsilon},\,{\rm d}\nu_\beta)$ is, for every $\beta\in\mathbb R$, isomorphic to the classical Sobolev space $H^1(\omega_{\varepsilon})$. However, the use of of the weight 
is more convenient as it will allow us to obtain estimates (in the trace theorems, for the continuity and the coercivity) involving constants which are uniform with respect to $\varepsilon\geqslant 0$. 

The introduction of the space  $H^1_N(\omega_{\varepsilon},{\rm d}m)$ is motivated by the following definition of solutions:
%--------------------------------
\begin{definition}[Finite energy solution]
\label{def:solution}
For $\varepsilon\geqslant 0$, let be $f_{\varepsilon}\in L^2(\omega_\varepsilon,{\rm d}m^{-1})$ an $g_{\varepsilon}\in L^2(\omega_\varepsilon,{\rm d}\sigma^{-1})$ 
satisfying the compatibility condition:
\begin{equation}
\label{comp:cond}
\int_{\omega_\varepsilon} f_\varepsilon\,{\rm d}x+\int_{\gamma_\varepsilon} g_\varepsilon\,{\rm d}x=0.
\end{equation}
Then, a function $u_\varepsilon\in H^1_N(\omega_{\varepsilon},{\rm d}m)$ is called  a  finite energy solution to System~\eqref{eq:uepsilon} if:
\begin{equation}
\label{eq:FVepsilon}
\int_{\omega_\varepsilon}\mathbb A_\varepsilon\nabla u_\varepsilon\cdot \nabla v\,{\rm d}x=\int_{\Omega_\varepsilon} f_\varepsilon v\,{\rm d}x+\int_{\Gamma_\varepsilon} g_\varepsilon v\,{\rm d}s,\qquad\forall\,v\in H^1_N(\omega_{\varepsilon},{\rm d}m).
\end{equation}
The corresponding Dirichlet energy is defined by:
\begin{equation}
\label{eq:dirichlet_energy}
E_\varepsilon:=\int_{\omega_\varepsilon}\mathbb A_\varepsilon\nabla u_\varepsilon\cdot\nabla u_\varepsilon\,{\rm d}x.
\end{equation}
\end{definition}
%------------------------------------

%\item As alredy mentioned at the beginning of this Section, any PDE textbook would provide existence and uniqueness result of solutions for the variational problem \eqref{eq:FVepsilon} when $\varepsilon>0$, based on Riesz representation Theorem and Poincar\'e-Wirtinger inequality.
\begin{rem}
When $\varepsilon=0$, 
we have $f_0/m\in L^2(\omega_0,{\rm d}m)$ and $g_0/m\in L^2(\gamma_0,{\rm d}\sigma)$. We can rewrite \eqref{comp:cond} as:
$$\int_{\omega_0} \frac{f_0}{m}\,{\rm d}m+\int_{\gamma_0} \frac{g_0}{m}\,{\rm d}\sigma=0,$$
which indeed makes sense according to Lemma~\ref{lem:conv_l1} stated below.
\end{rem}
\begin{rem}\label{rem:H1NH1compatibilite}
Since $f_\varepsilon$ and $g_\varepsilon$ satisfy the compatibility condition \eqref{comp:cond}, we can equivalently replace $H^1_N(\omega_\varepsilon,{\rm d}m)$ by $H^1(\omega_\varepsilon,{\rm d}m)$ in the statement of Problem~\ref{eq:FVepsilon}, for every $\varepsilon\geqslant 0$. 
\end{rem}
\begin{rem}
It can be easily checked that for data satisfying in the physical domain the (classical) conditions $F_{0}\in L^{2}(\Omega_{0})$  and $G_{0}\in L^{2}(\Gamma_{0})$, we have necessarily $f_{0}\in L^{2} (\omega_{0}, \,{\rm d}\nu_{2+2/\alpha})$ and $g_{0}\in L^{2} (\gamma_{0}, \,{\rm d}\nu_{1+1/\alpha})$. It is also worth noticing that $L^{2} (\omega_{0}, \,{\rm d}\nu_{2+2/\alpha})\subset L^{2} (\omega_{0}, \,{\rm d}m^{-1})$ for all $\alpha>0$,  while $L^{2} (\gamma_{0}, \,{\rm d}\nu_{1+1/\alpha})\subset L^{2} (\omega_{0}, \,{\rm d}\sigma^{-1})$ for $\alpha\leqslant 1$. In particular, (non zero) constant functions are in $L^{2} (\omega_{0}, \,{\rm d}\sigma^{-1})$ only for $\alpha<2$.
\end{rem}
In the rest of this section, we collect some useful results about the functional space  $H^1_N(\omega_{\varepsilon},{\rm d}m)$ (trace theorems, Poincar\'e inequality, extension operator from $\omega_{\varepsilon}$ to $\omega_0$), paying a very careful attention to ensure that the constants appearing in these continuity estimates are independent of $\varepsilon\geqslant 0$. These results will be used in Section \ref{sect:abstract} to study the well-posedness of the variational problem \eqref{eq:FVepsilon} for $\varepsilon=0$.

Note that $H^1(\omega_{0},{\rm d}m)$ contains functions like $x\mapsto \ln (1+|x|)$, which tends to $+\infty$ as $|x|\to +\infty$. However, we have the following density result:
\begin{lemma}
\label{lem:density}
The space
\begin{equation}
\label{eq:Eomega0}
\mathcal E(\omega_{0})=\{u|_{\omega_{0}}\,:\,u\in C^\infty_0(\mathbb R^2)\}
\end{equation}
is dense in $H^1(\omega_{0},{\rm d}m)$.
\end{lemma}
\begin{proof}
For every integer $n\geqslant 1$, define the cut-off function $\chi_n$ on $\omega_0$ by setting $\chi_n(x)=1$ in $D$ and, for every $x=(x_1,x_2)\in R_0$:
$$\chi_n(x)=\begin{cases}
1&\text{if }x_1\leqslant n\\
\displaystyle 1-\ln\left(\frac{1+x_1}{1+n}\right)&\text{if }n<x_1\leqslant N:=(n+1)e-1\\
0&\text{if }x_{1}>N.
\end{cases}$$
Let $u$ be in $H^1(\omega_{0},{\rm d}m)$ and set $u_n=u\chi_n$. We have:
\begin{align*}
\int_{\omega_{0}}|\nabla(u-u_n)|^2\,{\rm d}x&=\int_{\omega_{0}}|\nabla u - \chi_n\nabla u-u\nabla\chi_n)|^2\,{\rm d}x\\
&\leqslant 2\left(\int_{\omega_{0}} (1-\chi_n)|\nabla u|^2\,{\rm d}x+\int_{\omega_{0}} u^2|\nabla \chi_n|^2\,{\rm d}x\right).
\end{align*}
Let $\mathcal O_n^+=\{x_1>n\}\cap R_{0}$ and $\mathcal O_n^-=\omega_{0}\setminus\overline{\mathcal O_n^+}$.  We deduce from the last estimate that:
$$\int_{\omega_{0}}|\nabla(u-u_n)|^2\,{\rm d}x\leqslant 2\left(\int_{\mathcal O_n^+}|\nabla u|^2\,{\rm d}x+\int_{\mathcal O^+_n} u^2\,{\rm d}m\right),$$
and hence $\|\nabla(u-u_n)\|_{L^2(\omega_{0})}$ goes to 0 as $n$ goes to $+\infty$. Since $\|u-u_n\|_{L^2(\omega_{0},{\rm d}m)}$ obviously goes to 0 as well, we get that $\|u-u_n\|_{H^1(\omega_{0},{\rm d}m)} $ tends to 0.

Now, given $\eta>0$, fix $n$ large enough such that  
\begin{equation}
\label{eq:density:ineg1}
\|u-u_n\|_{H^1(\omega_{0},{\rm d}m)}\leqslant  \eta,
\end{equation}
and let $M$ be an integer larger that $N$. Classical density results for the standard Sobolev space $H^{1}(\mathcal O_M^-,{\rm d}m)$ on the bounded domain $\mathcal O_M^-$ ensure the existence of $v\in  C^\infty_0(\mathbb R^2)$ such that 
\begin{equation}
\label{eq:density:ineg2}
\|u_n-v\|_{H^1(\mathcal O_M^-,{\rm d}m)}\leqslant  \eta.
\end{equation}
In particular, this implies that on the rectangle $\mathfrak R:=]N,M[\times ]0,1[$ we have
\begin{equation}
\label{eq:density:ineg3}
\|v\|_{H^1(\mathfrak R,{\rm d}m)}\leqslant \eta.
\end{equation}
Set then $u^\ast =\theta v\in C^\infty_0(\mathbb R^2)$, where $\theta\in  C^\infty_0(\mathbb R)$ is a (one dimensional) cutt-off function satisfying $0\leqslant \theta(x_{1})\leqslant 1$ for all $x_{1}\in \mathbb R$, $\theta(x_{1})= 1$ for $x_1<N$, $\theta(x_{1})= 0$ for $x_1>M$. Then, we have
$$
\begin{array}{rcl}
\|u_{n}-u^\ast\|_{H^1(\omega_{0},{\rm d}m)}^2 &= &\|u_{n}-u^\ast\|_{H^1(\mathcal O_N^-,{\rm d}m)}^2
+\|u^\ast\|_{H^1(\mathcal O_N^+,{\rm d}m)}^2\\
& = &\|u_{n}- v\|_{H^1(\mathcal O_N^-,{\rm d}m)}^2
+\|\theta v\|_{H^1(\mathfrak R,{\rm d}m)}^2\\
& \leqslant &\|u_{n}- v\|_{H^1(\mathcal O_M^-,{\rm d}m)}^2
+\|\theta v\|_{H^1(\mathfrak R,{\rm d}m)}^2.
\end{array}
$$
Using \eqref{eq:density:ineg2} and \eqref{eq:density:ineg3}, the last inequality shows that
$$\|u_{n}-u^\ast\|_{H^1(\omega_{0},{\rm d}m)}\leqslant C\eta$$
for some  constant $C>0$ (depending only on $\theta$). Combining this estimate with \eqref{eq:density:ineg1} yields 
$$\|u-u^\ast\|_{H^1(\omega_{0},{\rm d}m)}\leqslant (C+1)\eta$$
which concludes the proof, since $\eta$ is arbitrary.
\end{proof}
The following Lemma explains why the case $\beta=2$ plays a particular role in the analysis:
\begin{lemma}
\label{lem:same_sobo}
For any real number  $\beta$, we have the following continuous embedding:
$$H^1(\omega_{0},\,{\rm d}\nu_{\beta})\hookrightarrow H^1(\omega_{0},{\rm d}m).$$
\end{lemma}
\begin{proof}
The result would be obvious if $\omega_{0}$ were bounded. For the sake of simplicity and without loss of generality, we only show that
$$H^1(R_{0},\,{\rm d}\nu_{\beta})\hookrightarrow H^1(R_{0},{\rm d}m),$$
which is equivalent to prove that the continuous embedding 
$$H^1(R_{0},\,{\rm d}\nu_{\beta})\hookrightarrow L^2(R_{0},{\rm d}m),$$
holds true. Let $n$ be an integer greater than $3$ and let $u$ be in $C^1(\overline{\mathfrak R_n})$ where, for every $k>0$, $\mathfrak R_k$ denotes the rectangle $\{0<x_1<k\}\cap R_{0}$. Define the 
cutt-off  function $\chi$ in $R_0$ by:
$$\chi(x)=\begin{cases}
1&\text{if }0\leqslant x_1< 1,\\
2-x_1&\text{if }1\leqslant x_1<2,\\
0&\text{if }2\leqslant x_1,
\end{cases}$$
and set $v=u\chi$ and $w=u(1-\chi)$. We have:
\begin{equation}
\label{eq:1}
\int_{\mathfrak R_n} u^2\,{\rm d}m\leqslant 2\left(\int_{\mathfrak R_2} v^2\,{\rm d}m + \int_{\mathfrak R_n} w^2\,{\rm d}m\right).
\end{equation}
On the one hand:
$$\int_{\mathfrak R_2} v^2\,{\rm d}m\leqslant C \int_{\mathfrak R_2} v^2\,{\rm d}\nu_\beta\leqslant C \int_{\mathfrak R_2} u^2\,{\rm d}\nu_\beta,$$
where the constant $C$ depends only on $\beta$.
On the other hand, for every $x\in\mathfrak R_n$:
$$|w(x)|^2=2\int_0^{x_1}w(s,x_2)\partial_{x_1}w(s,x_2)\,{\rm d}s,$$
and hence:
$$\int_{\mathfrak R_n}|w|^2\,{\rm d}m=2\int_0^1\!\!\int_0^n\!\!\int_0^{x_1}w(s,x_2)\partial_{x_1}w(s,x_2)(1+x_1)^{-2}\,{\rm d}s\,{\rm d}x_1\,{\rm d}x_2.$$
Invoking Fubini's theorem, we get:
$$
\int_{\mathfrak R_n}|w|^2\,{\rm d}m=2\int_0^1\!\!\int_0^nw(s,x_2)\partial_{x_1}w(s,x_2)\left(\int_s^n(1+x_1)^{-2}\,{\rm d}x_1\right)\,{\rm d}s\,{\rm d}x_2,
$$
and Cauchy-Schwarz inequality leads to:
$$\int_{\mathfrak R_n}|w|^2\,{\rm d}m\leqslant 2\left(\int_{\mathfrak R_n}|\nabla w|^2\,{\rm dx}\right)^{1/2}\left(\int_{\mathfrak R_n} u^2\left(\int_s^n(1+x_1)^{-2}\,{\rm d}x_1\right)^2\,{\rm d}s\,{\rm d}x_2\right)^{1/2}.$$
Noticing that for every $n\geqslant 3$:
$$\left(\int_s^n(1+x_1)^{-2}\,{\rm d}x_1\right)^2\leqslant (1+s)^{-2},$$
and hence:
\begin{align*}
\left(\int_{\mathfrak R_n}|w|^2\,{\rm d}m\right)^{1/2}&\leqslant 2\left(\int_{\mathfrak R_n}|\nabla w|^2\,{\rm d}x\right)^{1/2}\\
&=2 \left(\int_{\mathfrak R_n}|(1-\chi)\nabla u-u\nabla\chi|^2\,{\rm d}x\right)^{1/2}\\
&\leqslant 2^{3/2}\left(\int_{\mathfrak R_n}|\nabla u|^2\,{\rm d}x\right)^{1/2}+2^{3/2}\left(\int_{\mathfrak R_2} u^2 \,{\rm d}x\right)^{1/2}\\
&\leqslant 2^{3/2}\left(\int_{\mathfrak R_n}|\nabla u|^2\,{\rm d}x\right)^{1/2}+C\left(\int_{\mathfrak R_2} u^2 \,{\rm d}\nu_\beta\right)^{1/2},
\end{align*}
where the constant $C$ depends on $\beta$ only.
In \eqref{eq:1}, we get:
$$\int_{\mathfrak R_n} u^2\,{\rm d}m\leqslant C\left(\int_{\mathfrak R_n} |\nabla u|^2\,{\rm d}x+\int_{\mathfrak R_2} u^2\,{\rm d}\nu_\beta\right)^{1/2},$$
where $C=C(\beta)$. Since $C^1(\overline{\mathfrak R_n})$ is dense in $H^1(\mathfrak R_n)$, we deduce that this estimate still holds true
for every $u\in H^1(\mathfrak R_n)$.

Let now $u$ be any function in $H^1(\omega_{0},\,{\rm d}\nu_{\beta})$ and denote by $u_n=u|_{\mathfrak R_n}\in H^1(\mathfrak R_n)$. Applying the last estimate to $u_n$ and letting $n$ go to $+\infty$, we obtain the claimed result. 
%$$\int_{\mathfrak R_n} u_n^2\,{\rm d}m\leqslant C\left(\int_{\mathcal O_n} |\nabla u_n|^2\,{\rm d}x+\int_{\mathfrak R_2} u_n^2\,{\rm d}\nu_{\beta}\right)^{1/2},$$
\end{proof}
\begin{rem}
\label{rem:1}
Let $\beta_1$ and $\beta_2$ be two real numbers such that $\beta_1\geqslant \beta_2$. It is obvious to check that:
$$H^1(\omega_{0},\,{\rm d}\nu_{\beta_1})\hookrightarrow H^1(\omega_{0},\,{\rm d}\nu_{\beta_2}).$$
Then, it follows from Lemma \ref{lem:same_sobo} that for every $\beta\leqslant -2$, we have $H^1(\omega_{0},\,{\rm d}\nu_{\beta})=H^1(\omega_{0},{\rm d}m)$. 
\end{rem}
The next result shows in particular that the average of functions of $H^1(\omega_{0},dm)$ can be considered, and therefore the space $H_N^1(\omega_{\varepsilon},dm)$ introduced in \eqref{eq:defH1N} is well defined for $\varepsilon=0$.
\begin{lemma}
\label{lem:conv_l1}
For every $\beta<-3/2$, we have the following continuous embedding:
\begin{equation}
\label{embedd}
L^2(\omega_{0},{\rm d}m)\hookrightarrow L^1(\omega_{0},\,{\rm d}\nu_{\beta})\quad\text{and}\quad L^2(\gamma_{0},\,{\rm d}\sigma)\hookrightarrow L^1(\gamma_{0},\,{\rm d}\nu^S_{\beta}).
\end{equation}
\end{lemma}
\begin{proof}
One can simply observe that  for every $u$ in $L^2(\omega_{0},{\rm d}m)$, we have
\begin{align*}
\int_{R_{0}} |u(x)|(1+x_1)^\beta\,{\rm d}x&=\int_{R_{0}}|u(x)|(1+x_1)^{-1}(1+x_1)^{\beta+1}\,{\rm d}x\\
&\leqslant \|u\|_{L^2(R_{0},{\rm d}m)}\left(\int_{R_{0}} (1+x_1)^{2\beta+2}\,{\rm d}x\right)^{1/2}.
\end{align*}
The conclusion follows for the first embedding in \eqref{embedd}. The second embedding is proved exactly the same way. 
\end{proof}

In order to establish a convergence result as $\varepsilon$ tends to $0$, we need to be able to extend functions defined on $\omega_\varepsilon$ to $\omega_{0}$ in such a way that the extension operator be uniformly bounded with respect to $\varepsilon$. The following result provides the existence of such an operator.
\begin{lemma}[Extension operator]
\label{LEM:extension}
For every $\varepsilon>0$ small enough, there exists an extension operator:
$$T_\varepsilon:H^1(\omega_\varepsilon,{\rm d}m)\to H^1(\omega_{0},{\rm d}m),$$
such that $$\|T_\varepsilon\,u\|_{H^1(\omega_{0},{\rm d}m)}\leqslant \sqrt{3}\|u\|_{H^1(\omega_\varepsilon,{\rm d}m)}.$$
\end{lemma}
\begin{proof}
Since $D\subset \omega_\varepsilon$ for every $\varepsilon\geqslant 0$, it is sufficient to define  $R_\varepsilon$ from $H^1(R_\varepsilon,{\rm d}m)$ to $H^1(R_{0},{\rm d}m)$.

For every $\varepsilon>0$, set
$$
\ell_\varepsilon^\ast:=\ell_\varepsilon(2-e^{-1})+(1-e^{-1}),
$$
and let $\chi_\varepsilon$ be the cutt-off function defined in $R_0$ as follows:
$$\chi_\varepsilon(x)=
\begin{cases}
1&\text{if }x_1<\ell_\varepsilon\\
\displaystyle  1+\ln\left(1+\frac{\ell_\varepsilon-x_1}{\ell_\varepsilon+1}\right)&\text{if }\ell_\varepsilon\leqslant x_1<\ell_\varepsilon^\ast\\
0&\text{if }\ell_\varepsilon^\ast\leqslant x_1.
\end{cases}$$
Note that $0\leqslant \chi_\varepsilon(x)\leqslant 1$ for  every $x\in R_0$. Moreover, for $\varepsilon>0$ small enough, the quantity $\ell_\varepsilon^-:=2\ell_\varepsilon-\ell_\varepsilon^\ast$ is positive and for every 
$x\in R_0$ such that $x_1>\ell_\varepsilon^-$ we have:
\begin{equation}
\label{eq:propchi}
|\chi_\varepsilon'(2\ell_\varepsilon-x_{1})|^{2} =m(x_{1})=(1+x_{1})^{-2}.
 \end{equation}
For every $u\in H^1(R_\varepsilon,{\rm d}m)$, we define $T_\varepsilon u$ in $R_0$ as follows:
$$(T_\varepsilon u)(x)=
\begin{cases}
u(x_1,x_2)&\text{if }x_1\leqslant \ell_\varepsilon\\
\chi_\varepsilon(x_1)u(2\ell_\varepsilon-x_1,x_2)&\text{if }\ell_\varepsilon\leqslant x_1<\ell_\varepsilon^\ast\\
0&\text{if }\ell_\varepsilon^\ast\leqslant x_1.
\end{cases}$$
It can be easily verified that $T_\varepsilon u\in H^1_{\rm loc}(R_0)$. Moreover, we have:
\begin{multline*}
\int_{R_{0}}|\nabla (T_\varepsilon u)|^2\,{\rm d}x=\int_{R_\varepsilon}|\nabla u_\varepsilon|^2\,{\rm d}x\\\
\hspace{3cm}+\int_{R_{0}\cap\{\ell_\varepsilon<x_1<\ell_\varepsilon^\ast\}}
\left|\partial_{x_1}\chi_\varepsilon(x)u(2\ell_\varepsilon-x_1,x_2)-\chi_\varepsilon(x)\partial_{x_1}
u(2\ell_\varepsilon-x_1,x_2)\right|^2\,{\rm d}x\\
+\int_{R_{0}\cap\{\ell_\varepsilon<x_1<\ell_\varepsilon^\ast\}}
|\partial_{x_2}u(2\ell_\varepsilon-x_1,x_2)|^2\,{\rm d}x.\hspace{3cm}
\end{multline*}
Applying the change of variables $x_1'=2\ell_\varepsilon-x_1$ in the last two integrals of the right hand side and using Cauchy-Schwarz inequality, we get by using \eqref{eq:propchi} that:
$$\int_{R_{0}}|\nabla (T_\varepsilon u)|^2\,{\rm d}x\leqslant \int_{R_\varepsilon}|\nabla u_\varepsilon|^2\,{\rm d}x
+2\int_{R_{0}\cap\{\ell_\varepsilon^-<x_1<\ell_\varepsilon\}}u^2\,{\rm d}m+2\int_{R_{0}\cap\{\ell_\varepsilon^-<x_1<\ell_\varepsilon\}}|\nabla u|^2\,{\rm d}x.
$$
Consequently
\begin{equation}
\label{eq:inegnablaT}
\int_{R_{0}}|\nabla (T_\varepsilon u)|^2\,{\rm d}x\leqslant 3\|u\|_{H^1(R_\varepsilon,{\rm d}m)}^2.
\end{equation}
On the other hand, we also have:
\begin{align*}
\int_{R_{0}}|T_\varepsilon u|^2\,{\rm d}m&\leqslant \int_{R_\varepsilon}|u_\varepsilon|^2\,{\rm d}m+
\int_{R_{0}\cap\{\ell_\varepsilon^-<x_1<\ell_\varepsilon\}}|\chi_\varepsilon(2\ell_\varepsilon-x_1,x_2)|^2|u(x)|^2\,{\rm d}m\\
&\leqslant 2\|u\|_{L^2(R_\varepsilon,{\rm d}m)}^2.
\end{align*}
The announced estimate follows then immediately by combining \eqref{eq:inegnablaT} and the last inequality.
\end{proof}

\begin{lemma}[Uniform trace mapping]
\label{lem:trace}
Let $\mathcal E(\omega_0)$ be the functional space defined by \eqref{eq:Eomega0}. Then, the mapping
$$u\in\mathcal E(\omega_0)\mapsto u|_{\gamma_{0}}\in L^2(\gamma_{0},\,{\rm d}\sigma),$$
can be uniquely extended as a linear continuous operator:
$$\Lambda_0:u\in H^1(\omega_{0},{\rm d}m)\to L^2(\gamma_{0},\,{\rm d}\sigma).$$
Moreover, denoting by $\Lambda_\varepsilon$ the usual trace mapping from $H^1(\omega_\varepsilon)$ into $L^2(\gamma_\varepsilon)$ for $\varepsilon>0$, there exists a constant $C>0$ (independent of $\varepsilon\geqslant 0$) such that, 
for every $\varepsilon\geqslant 0$:
\begin{equation}
\label{unif_trace}
\|\Lambda_\varepsilon(u)\|_{L^2(\gamma_\varepsilon,{\rm d}\sigma)}\leqslant C\|u\|_{H^1(\omega_\varepsilon,{\rm d}m)}.
\end{equation}
\end{lemma}
\begin{proof}
Classically, it is sufficient to prove the existence of $\Lambda_0$ defined as an application from $H^1(R_0,{\rm d}m)$ into $L^2(\gamma^R_0,{\rm d}\sigma)$. For every $u\in\mathcal E(R_0)$, we have:
\begin{align*}|u(x_1,1)|^2&=\int_0^1\frac{\,{\rm d}}{\,{\rm d}s}(|u(x_1,s)|^2s){\rm
d}s\\
&=2\int_0^1\partial_{x_2}u(x_1,s)u(x_1,s)s{\rm
d}s+\int_0^1|u(x_1,s)|^2\,{\rm d}s.
\end{align*}
Multiplying both sides of this equality by $(1+x_1)^{-2}$, integrating from $0$ to $+\infty$ with respect to $x_1$ and using Cauchy-Schwarz inequality, we get:
$$\int_{\gamma^R_{0}}u^2(x_1,1)\,{\rm d}\sigma\leqslant 2\left(\int_{R_{0}} |\partial_{x_2}u|^2\,{\rm d}x\right)^{1/2}\left(\int_{R_{0}} u^2\,{\rm d}\nu_\beta\right)^{1/2}+\int_{R_{0}} u^2\,{\rm d}m,$$
where $\beta=-4$. According to Remark~\ref{rem:1}, we deduce that there exists a constant $C>0$ such that:
$$\int_{\gamma^R_{0}}u^2(x_1,1)\,{\rm d}\sigma\leqslant C\|u\|^2_{H^1(R_{0},{\rm d}m)}.$$ 
We conclude to the existence of $\Lambda_0$ by recalling the density of $\mathcal E(\omega_{0})$ into $H^1(\omega_{0},{\rm d}m)$ proved in Lemma \ref{lem:density}.

To get the uniform estimate \eqref{unif_trace}, we write that, for every $u\in H^1(\omega_\varepsilon)$ and using Lemma~\ref{LEM:extension}:
$$\|\Lambda_\varepsilon(u)\|_{L^2(\gamma_\varepsilon,{\rm d}\sigma)}\leqslant \|\Lambda_0(T_\varepsilon u)\|_{L^2(\gamma_0,{\rm d}\sigma)}\leqslant \tilde C\|T_\varepsilon u\|_{H^1(\omega_0,{\rm d}m)}
\leqslant C\|u_\varepsilon\|_{H^1(\omega_\varepsilon,{\rm d}m)},$$
where $\tilde C$ and $C$ are positive constants independent of $\varepsilon\geqslant 0$. The proof is now complete.
\end{proof}

\begin{lemma}[Poincar\'e-Wirtinger inequality]
\label{lem:poincare}
There exists a constant $C>0$ (independent of $\varepsilon$) such that for every $\varepsilon\geqslant 0$: 
\begin{equation}
\label{eq:poincare}
\|u\|_{L^2(\omega_\varepsilon,{\rm d}m)}\leqslant C\|\nabla u\|_{L^2(\omega_\varepsilon)},\quad\forall\,u\in H^1_N(\omega_\varepsilon,{\rm d}m).
\end{equation}
\end{lemma}
\begin{proof}
The result is proved in two steps. Using a direct calculation, we first show that this inequality holds on the (finite or semi-infinite) strip  $R_\varepsilon:=]0,\ell_\varepsilon[\times ]0,1[$ with $C=\sqrt{2} $:   
\begin{equation}
\label{eq:poincarerectangle}
\|u\|_{L^2(R_\varepsilon,{\rm d}m)}\leqslant \sqrt{2}\|\nabla u\|_{L^2(R_\varepsilon)},\quad\forall\,u\in H^1_N(R_\varepsilon,{\rm d}m).
\end{equation}
Next, we prove by contradiction that inequality \eqref{eq:poincarerectangle} implies \eqref{eq:poincare}. \\

Let $u$ be a function in $C^1(\overline{R_\varepsilon})$ or $\mathcal E(R_0)$ if $\varepsilon=0$. For every $x=(x_1,x_2)$ and $x'=(x_1',x_2')$ in $R_\varepsilon$, we have:
\begin{align*}
u(x_1,x_2)-u(x_1',x_2')&=u(x_1,x_2)-u(x_1',x_2)+u(x_1',x_2)-u(x_1',x_2')\\
&=\int_{x_1'}^{x_1}\partial_{x_1}u(s,x_2)\,{\rm d}s+\int_{x_2'}^{x_2}\partial_{x_2}u(x_1',s)\,{\rm d}s.
\end{align*}
Multiplying by $\left[(1+x_1)^{2}m(R_\varepsilon)\right]^{-1}$ and integrating the last equation with respect to $x_1$ from $0$ to $\ell_\varepsilon$, we get:
\begin{multline}
\label{eq:Poinc_wirt_main}
\frac{1}{m(R_\varepsilon)}\int_0^{\ell_\varepsilon}(1+x_1)^{-2}u(x_1,x_2)\,{\rm d}x_1-u(x_1',x_2')=\\
\frac{1}{m(R_\varepsilon)}\int_0^{\ell_\varepsilon}\left(\int_{x_1'}^{x_1}\partial_{x_1}u(s,x_2)\,{\rm d}s\right)(1+x_1)^{-2}\,{\rm d}x_1
+\int_{x_2'}^{x_2}\partial_{x_2}u(x_1',s)\,{\rm d}s.
\end{multline}
Applying Fubini's theorem to the first term of the right hand side, we get:
$$\int_0^{\ell_\varepsilon}\left(\int_{x_1'}^{x_1}\partial_{x_1}u(s,x_2)\,{\rm d}s\right)(1+x_1)^{-2}\,{\rm d}x_1=
\int_{x_1'}^{\ell_\varepsilon}\partial_{x_1}u(s,x_2)\left(\int_s^{\ell_\varepsilon}(1+x_1)^{-2}\,{\rm d}x_1\right)\,{\rm d}s.$$
Integrating now \eqref{eq:Poinc_wirt_main} with respect to $x_2$ from 0 to 1, we deduce that:
$$
\left|u(x')-\frac{1}{m(R_\varepsilon)}\int_{R_\varepsilon} u(x)\,{\rm d}m\right|\leqslant \frac{1}{m(R_\varepsilon)}\int_{R_\varepsilon}|\partial_{x_1}u(s,x_{2})|(1+s)^{-1}\,{\rm d}s\,{\rm d}x_2
+
\int_{0}^{1}|\partial_{x_2}u(x_1',s)|\,{\rm d}s.
$$
According to Cauchy-Schwarz inequality, we get:
$$
\left|u(x')-\frac{1}{m(\omega_\varepsilon)}\int_{\omega_\varepsilon} u(x)\,{\rm d}m\right|^2\leqslant 2\left[\frac{1}{m(\omega_\varepsilon)}\|\partial_{x_1} u\|_{L^2(\omega_\varepsilon)}^2
+
\int_{0}^{1}|\partial_{x_2}u(x_1',s)|^2\,{\rm d}s\right],$$
and then, multiplying by $(x_1'+1)^{-2}$ and integrating with respect to $x'$ on $\omega_\varepsilon$, we obtain:
$$
\int_{\omega_\varepsilon}\left|u(x')-\frac{1}{m(\omega_\varepsilon)}\int_{\omega_\varepsilon} u(x)\,{\rm d}m\right|^2\!\!\!\,{\rm d}m(x')\\
\leqslant 2\left[\|\partial_{x_1} u\|_{L^2(\omega_\varepsilon)}^2+\|\partial_{x_2} u\|_{L^2(\omega_\varepsilon)}^2\right],
$$
which shows that \eqref{eq:poincarerectangle} holds true.

Now, we show by contradiction that \eqref{eq:poincare} also holds. If not, there would exist two sequences $(\varepsilon_n)_{n\geqslant 1}\searrow 0$ and $(u_n)_{n\geqslant 1}$, with $u_n\in H_N^1(\omega_n, \,{\rm d}m)$ (for the sake of clarity, we set $\omega_n:= \omega_{\varepsilon_n}$ throughout the proof), such that
\begin{subequations}
\begin{align}
\|u_n\|_{L^2(\omega_n,{\rm d}m)}&=1 \label{eq:unnorme1} \\
\|\nabla u_n\|_{L^2(\omega_n)}&\to 0\quad\text{as }n\to +\infty.\label{eq:gradun0}
\end{align}
\end{subequations}
On the one hand, setting $\overline{u_{n}}:=\frac{1}{m(R_{n})}\int_{R_{n}} u_{n} \,{\rm d}m$,
the function defined on $R_{n}:= R_{\varepsilon_n}$ by $  u_n - \overline{u_{n}}$ obviously satisfies $v_{n}\in H_N^1(R_n, \,{\rm d}m)$ and, thanks to \eqref{eq:poincarerectangle} and \eqref{eq:gradun0},
$$
\left\|u_n - \overline{u_{n}}\right\|_{L^2(R_n, \,{\rm d}m)}\leqslant \sqrt{2}\|\nabla u_n \|_{L^2(R_n)}\to 0\quad\text{as }n\to +\infty,
$$
and thus 
\begin{equation}
\label{eq:ununbarre}
\|u_{n}-\overline{u_{n}}\|_{H^1(R_n, \,{\rm d}m)}\to 0\quad\text{as }n\to +\infty.
\end{equation}
Moreover, using Cauchy-Schwarz inequality and \eqref{eq:unnorme1}, we have
$$
|\overline{u_{n}} |= \frac{1}{m(R_{n})}\int_{R_{n}} |u_{n}| \,{\rm d}m \leqslant \frac{1}{\sqrt{m(R_{n})}}
\leqslant \frac{1}{\sqrt{m(R_{1})}}<\infty.
$$
Consequently, there exists a constant $U_R\in \mathbb R$ such that the sequence of real numbers $(\overline{u_{n}})_{n \geqslant 1}$ converge (up to a subsequence) to $U_R$. According to \eqref{eq:ununbarre}, this shows that 
\begin{equation}
\label{eq:unlimR}
\|u_{n}-U_R\|_{H^1(R_n, \,{\rm d}m)}\to 0\quad\text{as }n\to +\infty.
\end{equation}
On the other hand, on the domain $D=\omega_{n}\setminus{R_{n}}$, we immediately get  from \eqref{eq:unnorme1} and \eqref{eq:gradun0}, using the compactness of the injection from $H^1(D)$ into $L^2(D)$, that (up to 
a subsequence) there exists a constant $U_D\in \mathbb R$ such that 
\begin{equation}
\label{eq:unlimD}
\|u_{n}-U_D\|_{H^1(D)}\to 0\quad\text{as }n\to +\infty.
\end{equation}
The continuity of the trace of $u_n\in H_N^1(\omega_n, \,{\rm d}m)$ through the interface $\partial D\cap \partial R_{n}$ implies that $U_R=U_D:=U$. Since $u_n\in H_N^1(\omega_n, \,{\rm d}m)$, this common value $U$ is necessarily zero, as
$$
0  = \int_{\omega_n}u_{n}\,{\rm d}m =  \int_{D}u_{n}\,{\rm d}m + \int_{R_n}u_{n}\,{\rm d}m  \to (m(D)+m(R_0))U\quad\text{as }n\to +\infty. 
$$
But this fact is on contradiction with \eqref{eq:unnorme1}, \eqref{eq:unlimR} and \eqref{eq:unlimD}.
\end{proof}

%====================================================
\section{Some abstract well-posedness and convergence results}
\label{sect:abstract}
%====================================================

We are now in position to prove the well-posedness of the Neumann problem in the unbounded domain $\omega_0$ and a convergence result as $\varepsilon$ tends to $0$. Applying Riesz  representation Theorem, we immediately get by Poincar\'e-Wirtinger inequality (see Lemma \ref{lem:poincare}) the following well-posedness and uniqueness result:
 \begin{theorem}
 \label{theo:1}
 For every $f_0\in L^2(\omega_0,{\rm d}m^{-1})$ and every $g_0\in L^2(\gamma_0,\,{\rm d}\sigma^{-1})$ satisfying the compatibility condition
 $$\int_{\omega_0}f_0\,{\rm d}x+\int_{\gamma_0}g_0\,{\rm d}s=0,$$
 there exists a unique finite energy solution $u_0\in H^1_N(\omega_0,{\rm d}m)$ (in the sense of Defintion~\ref{def:solution}) to Problem~\ref{eq:FVepsilon} when $\varepsilon=0$.
\end{theorem}
\begin{rem}[Regularity of the solutions]
Investigating the maximal regularity for the solution $u_0$ in term of weighted Sobolev spaces is out of the range of our study. However, we can mention 
the following very basic result: In case $g_0\in H^{1/2}_{\rm loc}$, then it is classical to verify that $u_0|_{\mathfrak R_n}\in H^2(\mathfrak R_n)$ for every 
$n\geqslant 0$ (recall that $\mathfrak R_n:=]0,n[\times ]0,1[$). Moreover, still for every $n\geqslant 0$, the function $u_0^n:=u_0|_{\mathfrak R_n}$ satisfies 
$$-\dv(\mathbb A_0\nabla u_0^n)=f_0\text{ in }L^2(\mathfrak R_n)\quad\text{and}\quad \mathbb A_0\nabla u_0^n\cdot n=g_0\text{ in }H^{1/2}(\gamma^n_0),$$
where $\gamma_0^n$ is the upper boundary of $\mathfrak R_n$.
\end{rem}

As already mentioned in the beginning of the previous Section, the existence and uniqueness of a solution $u_\varepsilon\in H^1_N(\omega_\varepsilon,{\rm d}m)$ for Problem~\ref{eq:FVepsilon} when $\varepsilon>0$ is classical.
So, let us now investigate the convergence of $u_\varepsilon$ as $\varepsilon\to 0^+$. 

Notice that 
every functions $f_\varepsilon\in L^2(\omega_\varepsilon,{\rm d}m^{-1})$ and $g_\varepsilon\in L^2(\gamma_\varepsilon,{\rm d}\sigma^{-1})$ can be seen as functions of $L^2(\omega_0,{\rm d}m^{-1})$ and $L^2(\gamma_0,\,{\rm d}\sigma^{-1})$ respectively by setting 
 $f_\varepsilon:=0$ in $\omega_0\setminus\omega_\varepsilon$ and $g_\varepsilon:=0$ on $\gamma_0\setminus\gamma_\varepsilon$.

 \begin{theorem}
 \label{THEO:convergence}
For every $\varepsilon\geqslant 0$, let $f_\varepsilon\in L^2(\omega_\varepsilon,{\rm d}m^{-1})$ and $g_\varepsilon\in L^2(\gamma_\varepsilon,{\rm d}\sigma^{-1})$ be given such that the compatibility condition \eqref{comp:cond} is satisfied and denote by $u_\varepsilon\in H^1_N(\omega_\varepsilon,{\rm d}m)$ the unique solution to Problem~\ref{eq:FVepsilon}.

Assume that
\begin{equation}
\label{assume}
f_\varepsilon\to f_0\text{ in }L^2(\omega_0,{\rm d}m^{-1})\quad\text{and}\quad g_\varepsilon\to g_0\text{ in }L^2(\gamma_0,\,{\rm d}\sigma^{-1}).
\end{equation}
Then, under the convergence result \eqref{conv:Aeps},  we have:
\begin{subequations}
\begin{align}
\|\nabla(u_0-u_{\varepsilon})\|_{L^2(\omega_{\varepsilon})}&\to 0\quad\text{as }\varepsilon\to 0^+,\label{eq:result:1}\\
%\int_{\omega_\varepsilon}\mathbb A(\nabla u-\nabla u_\varepsilon)\cdot (\nabla u-\nabla u_\varepsilon)\,{\rm d}x&\to 0,\label{eq:result:1}\\
\left|\int_{\omega_0}\mathbb A_0\nabla u_0\cdot\nabla u_0\,{\rm d}x-\int_{\omega_\varepsilon}\mathbb A_\varepsilon \nabla u_\varepsilon\cdot\nabla u_\varepsilon\,{\rm d}x\right|&\to 0\quad\text{as }\varepsilon\to 0^+.
\label{eq:result:2}
\end{align}
\end{subequations}
 \end{theorem}
 \begin{proof}
Throughout this proof, $C$ will denote a constant that may change from line to line, but that is independent of $\varepsilon$. 

Taking $v=u_\varepsilon\in H^1(\omega_\varepsilon,{\rm d}m)$ in \eqref{eq:FVepsilon} (see Remark \ref{rem:H1NH1compatibilite}), we get that:
 $$\int_{\omega_\varepsilon}\mathbb A_\varepsilon\nabla u_\varepsilon\cdot\nabla u_\varepsilon\,{\rm d}x\leqslant \|f_\varepsilon\|_{L^2(\omega_0,{\rm d}m^{-1})}\|u_\varepsilon\|_{L^2(\omega_\varepsilon,{\rm d}m)}+
 \|g_\varepsilon\|_{L^2(\gamma_0,\,{\rm d}\sigma^{-1})}\|u_\varepsilon\|_{L^2(\gamma_\varepsilon,\,{\rm d}\sigma)}.$$
Since $\mathbb A_\varepsilon$ is positive definite uniformly (with respect to $\varepsilon$) according to Lemma~\ref{lem:ellipticiteuniform}, and since the continuity of the trace operator and the Poincar\'e-Wirtinger constants are uniform with respect to $\varepsilon\geqslant 0$ as well (as asserted in Lemma~\ref{lem:trace} and Lemma~\ref{lem:poincare}), we obtain that:
% $$\lambda_1\|\nabla u_\varepsilon\|^2_{L^2(\omega_\varepsilon)}\leqslant C
% (\|f_\varepsilon\|_{L^2(\omega,{\rm d}m)}+ \|g_\varepsilon\|_{L^2(\gamma_0,\,{\rm d}\sigma)})\|\nabla u_\varepsilon\|_{L^2(\omega_\varepsilon)},$$
$$\|\nabla u_\varepsilon\|^2_{L^2(\omega_\varepsilon)}\leqslant C
 (\|f_\varepsilon\|_{L^2(\omega_0,{\rm d}m^{-1})}+ \|g_\varepsilon\|_{L^2(\gamma_0,\,{\rm d}\sigma^{-1})})\|\nabla u_\varepsilon\|_{L^2(\omega_\varepsilon)},$$
and therefore, using again Lemma~\ref{lem:poincare}:
\begin{equation}
\label{eq:lokji}
\|u_\varepsilon\|_{H^1(\omega_\varepsilon,{\rm d}m)}\leqslant C,\quad\forall\,\varepsilon\geqslant 0.
\end{equation}
Specifying now $v=u_0|_{\omega_\varepsilon}-u_\varepsilon\in H^1(\omega_\varepsilon,{\rm d}m)$ in \eqref{eq:FVepsilon}, we get:
\begin{equation}
\label{eq:est_1}
\int_{\omega_\varepsilon}\mathbb A_\varepsilon\nabla u_\varepsilon\cdot(\nabla u_0-\nabla u_\varepsilon)\,{\rm d}x=\int_{\omega_\varepsilon}f_\varepsilon (u_0-u_\varepsilon)\,{\rm d}x
+\int_{\gamma_\varepsilon}g_\varepsilon(u_0-u_\varepsilon)\,{\rm d}s.
\end{equation}
On the other hand, taking $v=u_0-T_\varepsilon u_\varepsilon\in H^1(\omega_0,{\rm d}m)$ in \eqref{eq:FVepsilon} when $\varepsilon=0$, where $T_\varepsilon$ is the extension operator introduced in Lemma~\ref{LEM:extension}, we get:
\begin{equation}
\label{eq:est_2}
\int_{\omega_0}\mathbb A_0\nabla u_0\cdot(\nabla u_0-\nabla T_\varepsilon u_\varepsilon)\,{\rm d}x=\int_{\omega_0}f_0 (u_0-T_\varepsilon u_\varepsilon)\,{\rm d}x
+\int_{\gamma_0}g_0(u_0-T_\varepsilon u_\varepsilon)\,{\rm d}s.
\end{equation}
Setting $\omega_\varepsilon^c:=\omega_0\setminus{\omega_\varepsilon}=[\ell_\varepsilon,+\infty[\times]0,1[$ and $\gamma_\varepsilon^c=\gamma_0\setminus\gamma_\varepsilon$ and subtracting \eqref{eq:est_1} from  \eqref{eq:est_2}, we obtain:
\begin{multline}
\label{grand_form}
\int_{\omega_\varepsilon}\mathbb A_0(\nabla u_0-\nabla u_\varepsilon)\cdot (\nabla u_0-\nabla u_\varepsilon)\,{\rm d}x=-\int_{\omega_\varepsilon}(\mathbb A_0-\mathbb A_\varepsilon)\nabla u_\varepsilon\cdot(\nabla u_0-\nabla u_\varepsilon)\,{\rm d}x\\
-\int_{\omega_\varepsilon^c}\mathbb A_0\nabla u_0\cdot(\nabla u_0-\nabla T_\varepsilon u_\varepsilon)\,{\rm d}x
+\int_{\omega_\varepsilon}(f_0-f_\varepsilon)(u_0-u_\varepsilon)\,{\rm d}x+
\int_{\gamma_\varepsilon}(g_0-g_\varepsilon)(u_0-u_\varepsilon)\,{\rm d}s\\
+\int_{\omega_\varepsilon^c}f_0(u_0-T_\varepsilon u_\varepsilon)\,{\rm d}x
+\int_{\gamma_\varepsilon^c} g_0(u_0-T_\varepsilon u_\varepsilon )\,{\rm d}s.
\end{multline}
Lemma~\ref{LEM:extension} together with the estimate \eqref{eq:lokji} ensure that $\|T_\varepsilon u_\varepsilon\|_{H^1(\omega_\varepsilon,{\rm d}m)}\leqslant C$ for every $\varepsilon\geqslant 0$. On the other hand, taking into account 
the convergence result \eqref{conv:Aeps} and the hypothesis \eqref{assume} in \eqref{grand_form}, we get 
$$
\int_{\omega_\varepsilon}\mathbb A_0(\nabla u_0-\nabla u_\varepsilon)\cdot (\nabla u_0-\nabla u_\varepsilon)\,{\rm d}x\longrightarrow 0\quad\text{as }\varepsilon\to 0^+,
$$
and \eqref{eq:result:1} follows with \eqref{estim_A}.

To prove now \eqref{eq:result:2}, we write that:
\begin{multline*}
\int_{\omega_0}\mathbb A_0\nabla u_0\cdot\nabla u_0\,{\rm d}x-\int_{\omega_\varepsilon}\mathbb A_\varepsilon \nabla u_\varepsilon\cdot\nabla u_\varepsilon\,{\rm d}x
=\int_{\omega_\varepsilon}\mathbb A_0(\nabla u_0-\nabla u_\varepsilon)\cdot(\nabla u_0+\nabla u_\varepsilon)\,{\rm d}x\\
+\int_{\omega_\varepsilon}(\mathbb A_0-\mathbb A_\varepsilon)\nabla u_\varepsilon\cdot \nabla u_\varepsilon\,{\rm d}x,
\end{multline*}
and the conclusion follows, invoking again the same aforementioned boundedness and convergence arguments. 
 \end{proof}

%===================================
\section{Application to the collision problem}
\label{sect:application}
%===================================
The weak formulation of System~\ref{eq:uepsilon_2} reads, for every $\varepsilon\geqslant 0$:
\begin{equation}
\label{weak_form}
\int_{\omega_\varepsilon}\mathbb A_\varepsilon\nabla u_\varepsilon\cdot \nabla v\,{\rm d}x=\int_{\Gamma_\varepsilon} g_\varepsilon v\,{\rm d}s,\qquad\forall\,v\in H^1_N(\omega_{\varepsilon},{\rm d}m),
\end{equation}
where $g_\varepsilon=D\Psi_\varepsilon^{-1}\tau\cdot e_1$. This quantity can be made explicit on $\gamma^R_\varepsilon=\{(x_1,1)\,:\, 0<x_1<\ell_\varepsilon\}$ using the expression \eqref{def:Psi-1} of $\Psi_\varepsilon^{-1}$. Thus we get 
$$g_\varepsilon=H_\varepsilon(\mu_\varepsilon)\quad\text{on }\gamma^R_\varepsilon.$$
Notice that although the expression of the matrix $\mathbb A_\varepsilon$ depends on $\alpha$ (i.e. on the nature of the cusp) and on $\varepsilon\geqslant 0$, this dependence is somehow irrelevant regarding the well-posedness 
of Problem~\ref{weak_form} because, as asserted by Lemma~\ref{lem:ellipticiteuniform}, the matrix is always uniformly elliptic and therefore the left hand side 
of \eqref{weak_form} always defines a symmetric, elliptic bilinear form on $H^1_N(\omega_\varepsilon,{\rm d}m)$, according to Lemma~\ref{lem:poincare}.

All of the relevant information regarding the well-posedness of Problem~\ref{weak_form} is carried by the boundary data $g_\varepsilon$. Considering Theorem~\ref{theo:1}, a sufficient condition for Problem~\ref{weak_form}
to be well-posed when $\varepsilon=0$ is $g_0\in L^2(\gamma_0,{\rm d}\sigma^{-1})$ while, according to Theorem~\ref{THEO:convergence}, the convergence of the solution $u_\varepsilon$ toward $u_0$ and of the Dirichlet energy $E_\varepsilon$ toward $E_0$ is ensured if 
$g_\varepsilon\to g_0$ in $L^2(\gamma_0,{\rm d}\sigma^{-1})$. These conditions are easy to check and lead to distinguish two cases, a sub-critical case $\alpha<2$ and a super-critical case $\alpha\geqslant 2$. Let us emphasize that the critical value $2$ is nothing but the dimension, and this is in agreement with the results of \cite{CarNazSok09,NazSokTas09}.

%==================================
\subsection{The sub-critical case $\alpha<2$ }
\label{subsect:sub}
%===================================
\begin{theorem}
\label{THEO:nice_case}
When $\alpha<2$, the following assertions hold true:
\begin{enumerate}
\item Well posedness of the limit problem: The Neumann boundary value problem \eqref{weak_form} is well posed for $\varepsilon=0$. In particular, the corresponding Dirichlet energy $E_0$ is finite.
\item Convergence of solutions: $\|\nabla u_\varepsilon-\nabla u_0\|_{L^2(\omega_\varepsilon)}\to 0$ as $\varepsilon\to 0^+$, where $u_\varepsilon$ and $u_0$ are the solutions to Problem \eqref{weak_form} for $\varepsilon>0$ and $\varepsilon=0$ respectively.
\item Convergence of the Dirichlet energy: The Dirichlet energy $E_\varepsilon$  corresponding to Problem \eqref{weak_form} with $\varepsilon>0$ tends to $E_0$, the finite Dirichlet energy of the problem when $\varepsilon=0$.
\end{enumerate}
  \end{theorem}
  %--------------------------------------------
  Considering the implication of this result for the physical problem of collision, we deduce that the added mass \eqref{added_mass} is bounded uniformly in $\varepsilon\geqslant 0$. Using this estimate in \eqref{cauchy:pb}, we get that the velocity of the solid is 
  bounded from below and therefore:
\begin{cor}
\label{cor_shock}
When $\alpha<2$, the solid meets the cavity's wall in finite time with non-zero velocity (real shock case).
 \end{cor}
  \begin{proof}[of Theorem~\ref{THEO:nice_case}]
As already mentioned, the first point of the Theorem is a straightforward consequence of Theorem~\ref{theo:1}.  Indeed, applying Theorem~\ref{theo:1} with $f_0=0$ and $g_0=H_0(\mu_0)$, we get 
existence and uniqueness of a solution if $g_0\in L^2(\gamma_0,{\rm d}\sigma^{-1})$. 
Considering \eqref{def_mu0}, we deduce that:
 $$g_0\underset{x_1=+\infty}{\sim}\kappa^{-\frac{1}{\alpha}}\alpha^{-1-\frac{1}{\alpha}}x_1^{-1-\frac{1}{\alpha}},$$
and requiring $g_0$ to be in $L^2(\gamma_0,\,{\rm d}\sigma^{-1})$ leads to $\alpha<2$.
 
 The two remaining points result from Theorem~\ref{THEO:convergence}. It suffices to prove that $g_\varepsilon:=H_\varepsilon(\mu_\varepsilon)$ (extended by $0$ on $]\ell_\varepsilon,+\infty[$) converges to $g_0:=H_0(\mu_0)$ in $L^2(\gamma_0,{\rm d}\sigma^{-1})$.
 According to Lemma~\ref{LEM:proper_mu}, we get that $g_\varepsilon$ tends to $g_0$ a.e. on $\mathbb R_+$ and that:
 $$|g_\varepsilon-g_0|^{2}\leqslant C(1+x_1)^{-2-\frac{2}{\alpha}},$$
 for some constant $C>0$ independent of $\varepsilon\geqslant 0$.
 The conclusion follows from the dominated convergence theorem.
 \end{proof}
%===================================
\subsection{The super-critical case $\alpha\geqslant 2$ }
\label{subsect:super}
%===================================
When $\alpha\geqslant 2$, the boundary term $g_0:=H_0(\mu_0)$ in Problem~\ref{weak_form} is not anymore in $L^2(\gamma_0,\,{\rm d}\sigma^{-1})$, preventing from reasoning as in the previous section.

Actually, we already know, from Proposition~\ref{naza}, that Problem~\ref{weak_form} (when $\varepsilon=0$) does not admit finite energy solution in this case and from Proposition~\ref{prop:estim_alpha} that $E_\varepsilon\to +\infty$ as $\varepsilon\to 0^+$. 

In order to derive the first term in the asymptotic expansion of $E_\varepsilon$ when $\varepsilon$ goes to 0 in this case, we proceed as follows. For every $\varepsilon\geqslant 0$ we seek an ansatz $u_\varepsilon^s$ to $u_\varepsilon$, that contains all the information about the asymptotic behavior at infinity of $u_\varepsilon$  when $\varepsilon\to 0$, responsible of the blow up of the Dirichlet energy as $\varepsilon\to 0$. Equivalently, in the physical domain, this ansatz contains all the information about the appearance of the cusp singularity at the contact point. In particular, $u_\varepsilon^s$ will be shown to satisfy:
\begin{equation}
\label{blow_up}
\int_{\omega_\varepsilon}\mathbb A_\varepsilon\nabla u_\varepsilon^s\cdot \nabla u_\varepsilon^s{\rm d}x\to +\infty\quad\text{ as }\quad \varepsilon\to 0^+.
\end{equation}

This is why we call this ansatz as the singular part of the solution (which is a slight abuse of language since $u_\varepsilon$ is smooth and has finite Dirichlet energy for every $\varepsilon>0$), and we will refer to $u_\varepsilon^r:=u_\varepsilon-u_\varepsilon^s$ as the regular part of $u_\varepsilon$. 

The ansatz $u_\varepsilon^s$ will be derived by adapting to our semi-infinite strip the multiscale expansion method used in \cite{CarNazSok09,NazSokTas09} to obtain the singular behavior near the contact point (in the physical domain). More precisely, for every $\varepsilon\geqslant 0$, the ansatz $u_\varepsilon^s$ will be constructed such that the following properties hold true:
\begin{itemize}
\item The function $u_\varepsilon^s$ is smooth, supported in $R_\varepsilon$ and extended by 0 in $D$, and the quantity $\mathbb A_\varepsilon\nabla u_\varepsilon^s\cdot n$ vanishes on the boundary $\partial\omega_\varepsilon\setminus\gamma_\varepsilon$. This implies in particular that for every $\varepsilon>0$, the function $u_\varepsilon^s$ satisfies the weak formulation \eqref{eq:FVepsilon} with source terms $(f^s_\varepsilon, g_\varepsilon^s)\in L^2(\omega_\varepsilon,{\rm d}m^{-1})\times 
L^2(\gamma_\varepsilon,{\rm d}\sigma^{-1})$, where $f_\varepsilon^s:=-\dv(\mathbb A_\varepsilon\nabla u^s_\varepsilon)$ in $\omega_\varepsilon$ and $g^s_\varepsilon:=\mathbb A_\varepsilon\nabla u^s_\varepsilon\cdot n$ on $\gamma_\varepsilon$.
\item If we set
\begin{equation}\label{eq:fepsrgepsr}
f_\varepsilon^r:=-\dv(\mathbb A_\varepsilon\nabla u^r_\varepsilon)=-f_\varepsilon^s\text{ in }\omega_\varepsilon\qquad\text{and}\qquad g_\varepsilon^r:=\mathbb A_\varepsilon\nabla u^r_\varepsilon\cdot n=g_\varepsilon-g_\varepsilon^s\text{ on }\gamma_\varepsilon,
\end{equation}
then for $\varepsilon=0$, $f_0^r\in L^2(\omega_0,{\rm d}m^{-1})$, $g_0^r\in L^2(\gamma_0,{\rm d}\sigma^{-1})$ and the following convergences hold true:
$$\|f_\varepsilon^r-f_0^r\|_{L^2(\omega_\varepsilon,{\rm d}m^{-1})}\to 0\quad\text{ and }\quad\|g_\varepsilon^r-g_0^r\|_{L^2(\gamma_\varepsilon,{\rm d}\sigma^{-1})}\to 0\text{ as }\varepsilon\to 0^+.$$
\end{itemize}
Then, according to Lemma~\ref{lem:conv_l1} and since $f^r_\varepsilon$ and $g^r_\varepsilon$ satisfy the compatibility condition for every $\varepsilon>0$ (this follows from the fact that $u_\varepsilon^s$ is smooth and $\int_{\omega_\varepsilon}g_\varepsilon\,{\rm d}s=0$), we can pass to 
the limit in \eqref{comp:cond} to get:
$$\int_{\omega_0}f_0^r\,{\rm d}x+\int_{\gamma_0}g_0^r\,{\rm d}s=0.$$
Consequently, for every $\varepsilon\geqslant0$, $u_\varepsilon^r$ is a solution (in the sense of Definition \ref{def:solution}) of Problem \eqref{eq:FVepsilon} with the source terms $(f^r_\varepsilon,g_\varepsilon^r)$ given by \eqref{eq:fepsrgepsr}. On the other hand, the properties of $f_\varepsilon^r$ and $g^\varepsilon_r$ ensure, thanks to Theorem~\ref{THEO:convergence}, that $u^r_0\in H^1_N(\omega_0,{\rm d}m)$ 
and that 
$$\|\nabla u_\varepsilon^r - \nabla u_0^r\|_{L^2(\omega_\varepsilon)}=\|\nabla u_\varepsilon - \nabla(u_\varepsilon^s+ u_0^r)\|_{L^2(\omega_\varepsilon)}
\to 0\quad\text{as}\quad\varepsilon\to 0^+.$$ 
%With these setting, the identity \eqref{deomp} turns out to be a decomposition of $u_\varepsilon$ into a ``singular'' part $u_\varepsilon^s$ and a regular part 
%$u_\varepsilon^r$. The vocable ``singular'' is a slight abuse of language because $u_\varepsilon^s$ lies in $H^1_N(\omega_\varepsilon,{\rm d}m^{-1})$ for every $\varepsilon>0$ (and hence has finite Dirichlet energy). 
Considering the limit problem ($\varepsilon=0$), this construction will provide a natural solution 
$$u_0=u_0^s+u_0^r,$$
which does not belong to $H^1(\omega_0,{\rm d}m)$, leading  to supplement Definition~\ref{def:solution} with:
\begin{definition}[Infinite energy solution]
\label{def:solution_2}
Let $u_0^r\in H^1_N(\omega_0,{\rm d}m)$ be a solution to System~\eqref{eq:uepsilon} (for $\varepsilon=0$) with volume source term $f^r_0\in L^2(\omega_0,{\rm d}m^{-1})$ an boundary data $g^r_0\in L^2(\omega_0,{\rm d}\sigma^{-1})$ satisfying the compatibility condition 
\begin{equation}
\label{comp:cond_3}
\int_{\omega_0} f_0^r{\rm d}x+\int_{\gamma_0} g^r_0{\rm d}s=0.
\end{equation}
Let $u_0^s$ be a smooth function, supported in $R_0$ (and extended by $0$ in $D$) with infinite Dirichlet energy and such that $\mathbb A_0\nabla u_0^s\cdot n=0$ on 
$\partial\omega_0\setminus\gamma_0$ and denote 
$$f_0^s:=-\dv(\mathbb A_0\nabla u_0^s)\text{ in }\omega_0\quad\text{and}\quad g_0^s:=\mathbb A_0\nabla u_0^s\cdot n\text{ on }\gamma_0.$$
Then the function:
$$u_0:=u_0^s+u_0^r,$$
set in $\omega_0$ is called  an  infinite energy solution to System~\eqref{eq:uepsilon} (for $\varepsilon=0$) with volume source term $f_0:=f^r_0+f^s_0$ and boundary 
data $g_0:=g_0^r+g_0^s$.
\end{definition} 
\begin{rem}
The (smooth) function $u_0^s$ is required to be supported in $R_0$ in order to ensure that the expression of the volume source term $f_0^s$ makes sense. Indeed, the entries of the matrix $\mathbb A_0$ are only supposed to be continuous in $D$ whereas they are $C^1$ in $R_0$. 
\end{rem}
%This following comments are worth being taking into account:
%\begin{itemize}
%\item The smooth function $u_0^s$ requires to be supported in $R_0$ for the expression of the volume source term $f_0^s$ to make sense. Indeed, the entries of the matrix 
%$\mathbb A_0$ are only supposed to be continuous in $D$ whereas they are $C^1$ in $R_0$. 
%\item The functions $f_0$ and $g_0$ are no reason whatsoever to be in $L^1(\omega_0)$ and $L^1(\gamma_0)$ respectively. The problem of knowing whether 
% they satisfy the compatibility condition is therefore meaningless.
%\end{itemize}

We can now give the expression of the ansatz $u_\varepsilon^s$: Let $\chi$ be a smooth cut-off function defined in $\omega_0$ such that $\chi=1$ in $]1,+\infty[\times]0,1[$, $\chi$ is independent of $x_2$ in $]0,1[\times]0,1[$ and $\chi=0$ in $D$. 
For every $\varepsilon\geqslant 0$, we define the function $u_\varepsilon^s$ in $\omega_\varepsilon$ by setting:
\begin{equation}
\label{ansatz_exp}
u_\varepsilon^s(x)=\chi(x)\left[-\int_0^{x_1} \mu_\varepsilon(s)\,{\rm d}s+
\frac{1}{2}x_2^2\left[{H_\varepsilon(\mu_\varepsilon(x_1))-\mu_\varepsilon(x_1){H'_0}(\mu_\varepsilon(x_1))}\right]\right],
\end{equation}
for every $x\in\omega_\varepsilon$. The following Lemma, the proof of which is given in Appendix~\ref{SEC:technic}, asserts that the Dirichlet energy indeed blows up as $\varepsilon$ goes to $0^+$:
\begin{lemma}
\label{LEM:w_var}
The Dirichlet energy of $u_\varepsilon^s$ behaves as follows when $\varepsilon\to 0^+$:
\begin{equation}
\label{asymp_behav_energ}
\int_{\omega_\varepsilon}\mathbb A_\varepsilon\nabla u_\varepsilon^s\cdot \nabla u_\varepsilon^s{\rm d}x\underset{\varepsilon=0}{\sim}
\begin{cases} 
\frac{1}{3}\kappa^{-1}|\ln(\varepsilon)|&\mbox{if } \alpha=2 \\
\displaystyle\frac{1}{3}\varepsilon^{\frac{3}{1+\alpha}-1}\kappa^{-\frac{3}{1+\alpha}}\frac{3\pi/(1+\alpha)}{\sin(3\pi/(1+\alpha))} & \mbox{if } \alpha>2. 
\end{cases}  
\end{equation}
\end{lemma}

\begin{theorem}
\label{THEO:bad_case}
When $\alpha\geqslant 2$, the following assertions hold true:
\begin{enumerate}
\item There exists a function $u^r_0\in H^1(\omega_0,{\rm d}m)$ such that
$$\|\nabla u_\varepsilon-\nabla(u_\varepsilon^s+u^r_0)\|_{L^2(\omega_\varepsilon)}\to 0\qquad\text{as }\varepsilon\to 0^+.$$
%Considering \eqref{asymp_behav_energ}, it means that $u_\varepsilon^s+u^r_0$ are the two first terms in the asymptotic expansion of $u_\varepsilon$ as $\varepsilon\to 0^+$.
\item The function $u_0:=u_0^s+u_0^r$ is an infinite energy solution to System~\eqref{eq:uepsilon_2} (with $\varepsilon=0$), in the sense of Definition~\ref{def:solution_2}.
\item The Dirichlet energy of $u_\varepsilon$ behaves as the Dirichlet energy of $u_\varepsilon^s$ as $\varepsilon\to 0^+$.
\end{enumerate}
\end{theorem}
\begin{proof}
For the sake of clarity, we provide a constructive proof to explain how to obtain the ansatz \eqref{ansatz_exp}. Our method can be seen as an adaptation for every $\varepsilon\geqslant 0$ of the multiscale expansion method used in \cite{NazSokTas09} in the case $\varepsilon=0$.

First, recall that the system \eqref{eq:uepsilon} under consideration reads:
\begin{subequations}
\label{sys:u_varepsilon}
\begin{alignat}{3}
-\dv(\mathbb A_\varepsilon\nabla u_\varepsilon)&=0&\quad&\text{in }\omega_\varepsilon\\
\mathbb A_\varepsilon\nabla u_\varepsilon\cdot n&=g_\varepsilon&&\text{on }\gamma_\varepsilon\\
\mathbb A_\varepsilon\nabla u_\varepsilon\cdot n&=0&&\text{on }\partial\omega_\varepsilon\setminus\gamma_\varepsilon,
\end{alignat}
\end{subequations}
where we have set  
$$g_\varepsilon(x)=D\Psi^{-1}_\varepsilon\tau\cdot e_1=H_\varepsilon(X_\varepsilon^1), \qquad\qquad X_\varepsilon^1:=\mu_\varepsilon(x_1).$$ 
On the rectangle $R_\varepsilon$, we have:
\begin{multline}
\label{eq:nablaA}
-\dv (\mathbb A_\varepsilon\nabla u_\varepsilon)(x)=-\partial^2_{x_1^2}u_\varepsilon(x)+2x_2{H'_0}(X_1^\varepsilon)\partial^2_{x_1x_2}u_\varepsilon(x)
-\left[1+\big(x_2{H'_0}(X_1^\varepsilon)\big)^2\right]\partial^2_{x_2^2}u_\varepsilon(x)\\
+{H'_0}(\mu_\varepsilon)\partial_{x_1}u_\varepsilon
-
x_2\big[2{H'_0}(X_1^\varepsilon)^2+{H_0''}(X_1^\varepsilon)H_\varepsilon(X_1^\varepsilon)\big]\partial_{x_2}u_\varepsilon(x).
\end{multline}
We seek an approximate solution to System~\eqref{sys:u_varepsilon} in $R_\varepsilon$ in the form:
\begin{equation}
\label{express_ansatz}
\hat u_\varepsilon^s(x)=v_\varepsilon(X_1^\varepsilon)+H_\varepsilon(X_1^\varepsilon) V_\varepsilon(X_1^\varepsilon,x_2),\quad \forall\,x=(x_1,x_2)\in R_\varepsilon,
\end{equation}
the functions $v_\varepsilon$ and $V_\varepsilon$ being to be determined.
With \eqref{eq:nablaA} we get:
\begin{subequations}
\label{express_w_eps}
\begin{equation}
\label{def:A_var_hatu}
-\dv(\mathbb A_\varepsilon\nabla \hat u_\varepsilon^s)(x)=-{H_\varepsilon}(X_1^\varepsilon)\left[\partial^2_{x^2_2} V_\varepsilon({X_1^\varepsilon},x_2)+{H_\varepsilon}(X_1^\varepsilon) v''_\varepsilon({X_1^\varepsilon})\right]+\hat f^s_\varepsilon(x)
\end{equation}
where 
\begin{multline}
\label{def:hatf}
\hat f^s_\varepsilon(x):=-{H_\varepsilon}(X_1^\varepsilon) \Big\{{H_0''}(X_1^\varepsilon){H_\varepsilon}(X_1^\varepsilon) V_\varepsilon({X_1^\varepsilon},x_2)
+x_2^2{H'_0}(X_1^\varepsilon)^2\partial^2_{x_2^2} V_\varepsilon({X_1^\varepsilon},x_2)\\
+2{H'_0}(X_1^\varepsilon){H_\varepsilon}(X_1^\varepsilon)[\partial_{x_1} V_\varepsilon({X_1^\varepsilon},x_2)-x_2\partial^2_{x_1x_2} V_\varepsilon({X_1^\varepsilon},x_2)]\\
-x_2{H_0''}(X_1^\varepsilon){H_\varepsilon}(X_1^\varepsilon)\partial_{x_2} V_\varepsilon({X_1^\varepsilon},x_2)+H_0({X_1^\varepsilon})^2\partial^2_{x_1^2} V_\varepsilon({X_1^\varepsilon},x_2)\Big\}.
\end{multline}
On $\gamma_\varepsilon^R$, i.e. for $x=(x_1,1)$, $0< x_1<\ell_\varepsilon$ or equivalently $\delta<X_1^\varepsilon<0$, we have:
\begin{equation}
\label{def:bound_A}
\mathbb A_\varepsilon\nabla \hat u_\varepsilon^s\cdot n(x_1,1)={H_\varepsilon}(X_1^\varepsilon)\Big[-v'_\varepsilon({X_1^\varepsilon}){H'_0}(X_1^\varepsilon)+\partial_{x_2} V_\varepsilon({X_1^\varepsilon},1)\Big]+\hat r^s_\varepsilon(x),
\end{equation}
where:
\begin{equation}
\label{def:hatr}
\hat r^s_\varepsilon(x):={H_\varepsilon}(X_1^\varepsilon)\Big\{{H'_0}(X_1^\varepsilon)^2[\partial_{x_2}V_\varepsilon({X_1^\varepsilon},1)-V_\varepsilon({X_1^\varepsilon},1)]
-{H_\varepsilon}(X_1^\varepsilon){H'_0}(X_1^\varepsilon)\partial_{x_1}V_\varepsilon({X_1^\varepsilon},1)\Big\}.
\end{equation}
On the lower boundary $\{(x_1,0)\,:\,0< x_1<\ell_\varepsilon\}$, we have:
\begin{equation}
\label{cond_lower}
\mathbb A_\varepsilon\nabla \hat u_\varepsilon^s\cdot n(x_1,0)=H_\varepsilon(X_1^\varepsilon)\partial_{x_2}V_\varepsilon(X_1^\varepsilon,0).
\end{equation}
Finally, on the vertical right boundary $\{(\ell_\varepsilon,x_2)\,:\,0<x_2<1\}$ we get:
\begin{equation}
\label{cond_right}
\mathbb A_\varepsilon\nabla \hat u_\varepsilon^s\cdot n(\ell_\varepsilon,x_2)=\varepsilon[v'_\varepsilon(0)+\varepsilon\partial_{x_1}V_\varepsilon(0,x_2)].
\end{equation}
\end{subequations}
The function $\hat u_\varepsilon^s$ is supposed to be an ansatz for $u_\varepsilon$, so in view of the expressions \eqref{express_w_eps}, we seek the functions $v_\varepsilon$ and $V_\varepsilon$ in order to cancel the ``leading'' (i.e. less decreasing) terms in \eqref{def:A_var_hatu}. We also want the Neumann boundary conditions for 
$\hat u_\varepsilon$ to approximate ``at best'' the boundary conditions of $u_\varepsilon$. This leads to 
the following one dimensional Neumann system that must be satisfied for every fixed $\delta<X_1^\varepsilon<0$:
\begin{subequations}
\label{sys_for_V_varepsilon}
\begin{alignat}{3}
-\partial^2_{x^2_2} V_\varepsilon(X_1^\varepsilon,\cdot)&=H_\varepsilon(X_1^\varepsilon)v''_\varepsilon(X_1^\varepsilon)&\quad&\text{on }]0,1[\\
\partial_{x_2} V_\varepsilon(X_1^\varepsilon,1)&=1+v'_\varepsilon(X_1^\varepsilon)H'_0(X_1^\varepsilon)\\
\partial_{x_2} V_\varepsilon(X_1^\varepsilon,0)&=0.
\end{alignat}
\end{subequations}
The compatibility condition, necessary for this System to admit solutions, reads:
$$H_\varepsilon(X_1^\varepsilon)v''_\varepsilon(X_1^\varepsilon)+1+v'_\varepsilon(X_1^\varepsilon)H'_0(X_1^\varepsilon)=0,\quad \delta<X_1^\varepsilon<0.$$
It can been rewritten as:
$$\frac{d^2}{dx_1^2}(v_\varepsilon(\mu_\varepsilon(x_1)))=-\mu_\varepsilon'(x_1).$$
We choose as a solution to this EDO, the one which vanishes at $x_1=0$:
\begin{equation}
\label{def:v}
v_\varepsilon(\mu_\varepsilon(x_1))=-\int_0^{x_1} \mu_\varepsilon(s)\,{\rm d}s.
\end{equation}
We deduce that:
$$-H_\varepsilon(X_1^\varepsilon)v''_\varepsilon(X_1^\varepsilon)=1+v_\varepsilon'(X_1^\varepsilon){H'_0}(X_1^\varepsilon)=\frac{H_\varepsilon(X_1^\varepsilon)
-X_1^\varepsilon{H'_0}(X_1^\varepsilon)}
{H_\varepsilon(X_1^\varepsilon)},$$
whence we infer the expression of $V_\varepsilon$:
\begin{equation}
\label{def:V}V_\varepsilon(x_1,x_2)=\frac{1}{2}x_2^2\left[\frac{H_\varepsilon(X_1^\varepsilon)-X_1^\varepsilon{H'_0}(X_1^\varepsilon)}
{H_\varepsilon(X_1^\varepsilon)}\right],
\end{equation}
and then the expression of $\hat u_\varepsilon^s$ in $R_\varepsilon$  thanks to \eqref{express_ansatz}:
\begin{equation}
\label{def:hatw}
\hat u_\varepsilon^s(x_1,x_2)=-\int_0^{x_1} \mu_\varepsilon(s)\,{\rm d}s+
\frac{1}{2}x_2^2\left[{H_\varepsilon(X_1^\varepsilon)-X_1^\varepsilon{H'_0}(X_1^\varepsilon)}\right].
\end{equation}
Reconsidering now the expressions \eqref{def:A_var_hatu} and \eqref{def:bound_A}, we have by construction $-\dv(\mathbb A_\varepsilon\nabla \hat u_\varepsilon^s)=\hat f_\varepsilon^s$ in $\omega_\varepsilon$,  
$\mathbb A_\varepsilon\nabla \hat u_\varepsilon^s\cdot n=\hat r_\varepsilon^s+g_\varepsilon:=\hat g_\varepsilon^s$ on $\gamma_\varepsilon^R$ and $\mathbb A_\varepsilon\nabla \hat u_\varepsilon^s\cdot n=0$ on the lower and right boundaries of $R_\varepsilon$ (according respectively to \eqref{cond_lower} and \eqref{cond_right}).
Recalling that $\chi$ is the cut-off function introduced above Lemma~\ref{LEM:w_var}, we can define $u_\varepsilon^s$ in the whole domain $\omega_\varepsilon$ by setting:
$$u_\varepsilon^s(x):=\chi(x)\hat u_\varepsilon^s(x),\qquad (x\in\omega_\varepsilon),$$
and we recover the announced expression \eqref{ansatz_exp}.

It remains to verify now that $u^s_\varepsilon$ indeed carries the ``singular'' part of $u_\varepsilon$. Straightforward computation leads to:
\begin{alignat*}{3}
-\dv(\mathbb A_\varepsilon\nabla u_\varepsilon^s)&=f_\varepsilon^s&\quad&\text{in }\omega_\varepsilon\\
\mathbb A_\varepsilon \nabla u_\varepsilon^s\cdot n&= g_\varepsilon^s&&\text{on }\gamma_\varepsilon\\
\mathbb A_\varepsilon \nabla u_\varepsilon^s\cdot n&=0&&\text{on }\partial\omega_\varepsilon\setminus\gamma_\varepsilon
\end{alignat*}
where:
\begin{subequations}
\label{def_f_r1}
\begin{alignat}{3}
f^s_\varepsilon&=\left[-\dv(\mathbb A_\varepsilon\nabla \chi)\hat u_\varepsilon^s-2\mathbb A_\varepsilon:(\nabla\chi\otimes\nabla \hat u_\varepsilon^s)\right]+\chi \hat f^s_\varepsilon&\quad&\text{in }\omega_\varepsilon\\
g^s_\varepsilon&=\left[\hat u_\varepsilon^s(\mathbb A_\varepsilon\nabla\chi)\cdot n\right]+\chi \hat g_\varepsilon^s&&\text{on }\gamma_\varepsilon.
\end{alignat}
\end{subequations}
The function $u_\varepsilon^r:=u_\varepsilon-u_\varepsilon^s$ satisfies:
\begin{subequations}
\label{sys:u_eps_r}
\begin{alignat}{3}
-\dv(\mathbb A_\varepsilon\nabla  u_\varepsilon^r)&=f^r_\varepsilon&\quad&\text{in }\omega_\varepsilon\\
\mathbb A_\varepsilon \nabla u_\varepsilon^r\cdot n&= g^r_\varepsilon&&\text{on }\gamma_\varepsilon\\
\mathbb A_\varepsilon \nabla u_\varepsilon^r\cdot n&=0&&\text{on }\partial\omega_\varepsilon\setminus\gamma_\varepsilon,
\end{alignat}
\end{subequations}
where:
\begin{equation}
\label{def_f_r2}
f^r_\varepsilon=-f_\varepsilon^s\text{ in }\omega_\varepsilon\quad\text{and}\quad
g^r_\varepsilon=g_\varepsilon-g_\varepsilon^s\text{ on }\gamma_\varepsilon.
\end{equation}
For every $\varepsilon>0$ the functions $f_\varepsilon^s$ and $g_\varepsilon^s$ satisfy the compatibility condition (they are defined as being respectively the divergence 
and the flux across the boundary of the smooth vector field $\mathbb A_\varepsilon\nabla u^s_\varepsilon$). Moreover, $\int_{\gamma_\varepsilon} g_\varepsilon\,{\rm d}s=0$ for every $\varepsilon\geqslant 0$. We deduce that $f_\varepsilon^r$ and $g_\varepsilon^r$ satisfy the compatibility condition as well (for every  $\varepsilon>0$).

The following Lemma ensures that the ansatz function $u_\varepsilon^s$ does the job it has be designed for:
%=============================
\begin{lemma}
\label{LEM:conv_f_0}
The function $f^r_0$ belongs to $L^2(\omega_0,{\rm d}m^{-1})$, the function $g^r_0$  belongs to $L^2(\gamma_0,\,{\rm d}\sigma^{-1})$ and they satisfy the compatibility condition \eqref{comp:cond_3}. 
Moreover, we have the following convergence results:
\begin{equation}
\label{eq:converg_f_r}
\|f^r_\varepsilon-f^r_0\|_{L^2(\omega_\varepsilon,{\rm d}m^{-1})}\to 0\quad\text{ and }\quad\|g^r_\varepsilon-g^r_0\|_{L^2(\gamma_\varepsilon,\,{\rm d}\sigma^{-1})}\to 0\text{ as }\varepsilon\to 0^+.
\end{equation}
\end{lemma}
%=============================
The proof is postponed to Appendix~\ref{SEC:technic}.

To complete the proof of the two firsts points of the theorem, it suffices now to apply Theorem~\ref{THEO:convergence}. 

Finally, the last remaining point results from the second triangular inequality:
$$\left|\left(\int_{\omega_\varepsilon}\mathbb A_\varepsilon \nabla u_\varepsilon\cdot\nabla u_\varepsilon\,{\rm d}x\right)^{\frac{1}{2}}-
\left(\int_{\omega_\varepsilon}\mathbb A_\varepsilon \nabla u^s_\varepsilon\cdot\nabla u^s_\varepsilon\,{\rm d}x\right)^{\frac{1}{2}}\right|
\leqslant \left(\int_{\omega_\varepsilon}\mathbb A_\varepsilon \nabla u^r_\varepsilon\cdot\nabla u^r_\varepsilon\,{\rm d}x\right)^{\frac{1}{2}},
$$
together with Lemma~\ref{LEM:w_var} and again Theorem~\ref{THEO:convergence}.
The proof of the theorem is now complete.
\end{proof}
%===================================
\section{Back to miscellaneous cases of Section~\ref{sect:preliminary}}
\label{subsect:flat}
In Section~\ref{sect:preliminary} we state some results allowing comparing the Dirichlet energy for different configurations, some of them being not covered by our general study. We show in this Section that the method used to determine the asymptotic of the Dirichlet energy in the previous Section can be adapted to the case where the bottom 
of the solid is locally flat (referred to as ``the flat case'' in the sequel).  
\subsection{The flat case}
We shall now focus on the cases depicted on the left of Fig.~\ref{fig_pic}, i.e. where the bottom of the solid is locally flat. Being more specific, we consider the case where the function 
$H_\varepsilon$ has the following expression (hereafter, in addition to the already defined constants, $\delta'$ is a negative number such that $\delta<\delta'<0$):
$$H_\varepsilon(\xi_1)=\begin{cases}
\hat H_\varepsilon(\xi_1-\delta')&\text{if }\delta\leqslant \xi_1<\delta'\\
\varepsilon&\text{if }\delta'\leqslant \xi_1<0,
\end{cases}$$
where $\hat H_\varepsilon(\xi_1):=\kappa |\xi_1|^{1+\alpha}+\varepsilon$ with $\alpha>2$. Denoting $\hat\delta:=\delta-\delta'<0$, we can define, associated with $\hat H_\varepsilon$ and $\hat\delta$, the functions $\hat\rho_\varepsilon$, $\hat\mu_\varepsilon$, $\hat\Psi_\varepsilon$ and $\hat\Psi^{-1}_\varepsilon$ based on formula \eqref{def_rho}, \eqref{def_mu}, \eqref{def:Psi_on_V} and \eqref{def:Psi-1}. We will also need the constant $\hat\ell_\varepsilon:=\lim_{\xi_1\to 0^+}\hat\rho_\varepsilon(\xi_1)$ . Observe now that, corresponding to $H_\varepsilon$ and using the very same formula as above, we get:
$$\rho_\varepsilon(\xi_1)=\begin{cases}
\hat\rho_\varepsilon(\xi_1-\delta')&\text{if }\delta\leqslant \xi_1<\delta'\\
\frac{1}{\varepsilon}(\xi_1-\delta')+\hat\ell_\varepsilon&\text{if }\delta'\leqslant \xi_1<0,
\end{cases}$$
with $\ell_\varepsilon:=\rho_\varepsilon(0)=\hat\ell_\varepsilon-\frac{\delta'}{\varepsilon}$ ($\ell_0=+\infty$). Notice that the function $\rho_{\varepsilon=0}$ is only defined for $\delta\leqslant \xi_1<\delta'$. For $\varepsilon\geqslant 0$, the inverse of $\rho_\varepsilon$ defined on $[0,\ell_\varepsilon[$ reads:
$$\mu_\varepsilon(x_1)=\begin{cases}
\hat\mu_\varepsilon(x_1)+\delta'&\text{if }0\leqslant x_1<\hat\ell_\varepsilon\\
\varepsilon(x_1-\hat\ell_\varepsilon)+\delta'&\text{if }\hat\ell_\varepsilon\leqslant x_1<\ell_\varepsilon.
\end{cases}$$
We deduce that, for every $\varepsilon\geqslant 0$:
\begin{equation}
\label{speer}
H_\varepsilon(\mu_\varepsilon(x_1))=\begin{cases}\hat H_\varepsilon(\hat\mu_\varepsilon(x_1))&\text{if }0\leqslant x_1<\hat\ell_\varepsilon\\
\varepsilon&\text{if }\hat\ell_\varepsilon\leqslant x_1<\ell_\varepsilon,
\end{cases}
\quad\text{and}\quad
H_\varepsilon^{(k)}(\mu_\varepsilon(x_1))=\begin{cases}\hat H^{(k)}_\varepsilon(\hat\mu_\varepsilon(x_1))&\text{if }0\leqslant x_1<\hat\ell_\varepsilon\\
0&\text{if }\hat\ell_\varepsilon\leqslant x_1<\ell_\varepsilon,
\end{cases}
\end{equation}
where $k\geqslant 1$ stands for the derivative of order $k$ (if well defined).
On the rectangle $R_\varepsilon$, we get:
$$\Psi_\varepsilon^{-1}=\hat\Psi_\varepsilon^{-1}+
\begin{pmatrix}
\delta'\\
0
\end{pmatrix}\text{ in }\hat R_\varepsilon=]0,\hat\ell_\varepsilon[\times]0,1[
\quad
\text{and}
\quad
\Psi^{-1}_\varepsilon(x)=\varepsilon
\begin{pmatrix}
x_1-\hat\ell_\varepsilon\\
x_2\end{pmatrix}+
\begin{pmatrix}
\delta'\\
0
\end{pmatrix}\text{in }[\hat\ell_\varepsilon,\ell_\varepsilon[\times]0,1[.
$$
According to formula \eqref{eq:defAepsilon}, we deduce that:
\begin{equation}
\label{def:barA}
\mathbb A_\varepsilon=\begin{cases} \hat{\mathbb A}_\varepsilon&\text{in }\hat R_\varepsilon:=]0,\hat\ell_\varepsilon[\times]0,1[\\
\varepsilon{\rm Id}&\text{in }[\hat\ell_\varepsilon,\ell_\varepsilon[\times]0,1[.
\end{cases}
\end{equation}
Notice once more that $\hat R_{\varepsilon=0}=R_{\varepsilon=0}=]0,+\infty[\times ]0,1[$. Considering \eqref{def:barA}, 
we deduce that the conclusions of Lemma~\ref{lem:ellipticiteuniform} still holds true. This convergence results is required in Theorem~\ref{THEO:convergence}, while 
all the results of Section~\ref{SEC:function} are completely independent of the change of variables. We can now jump directly to the computation of the ansatz function. 
The general expression \eqref{ansatz_exp} leads to, in our case:
$$u^s_\varepsilon(x)=
\begin{cases}
\hat u_\varepsilon^s(x)-\delta'\chi(x)\left[x_1+\frac{1}{2}x_2^2\hat H_0'(\hat\mu_\varepsilon(x_1))\right]&\text{if }x\in\hat\omega_\varepsilon\\
-\int_0^{\hat\ell_\varepsilon}\hat\mu_\varepsilon(s)\,{\rm d}s+\frac{\varepsilon}{2}(x_2^2-x_1^2+2x_1\ell_\varepsilon-\hat\ell_\varepsilon^2)
&\text{if }x\in[\hat\ell_\varepsilon,\ell_\varepsilon[\times]0,1[,
\end{cases}$$
where $\hat u_\varepsilon^s$ is the ansatz for the problem corresponding to $\hat H_\varepsilon$ and $\chi$ the cut-off function defined above identity \eqref{ansatz_exp}.  
The mirror image of Lemma~\ref{LEM:w_var} is 
\begin{lemma}
The Dirichlet energy of $u_\varepsilon^s$ behaves as follows when $\varepsilon\to 0^+$:
$$\int_{\omega_\varepsilon}\mathbb A_\varepsilon\nabla u_\varepsilon^s\cdot\nabla u_\varepsilon^s\,{\rm d}x\underset{\varepsilon=0}{\sim}\frac{|\delta'|^3}{3}\varepsilon^{-1}.$$
\end{lemma}
Unlike the other cases, neither $\kappa$ nor $\alpha$ appear in the expression of the leading term in the asymptotic expansion of the energy. They would probably play a role
in lower order terms only. This observation lead us to think that the technical condition $\alpha>2$ have to be understood as a regularity assumption for the solid's boundary and is not 
related with the ``strength'' of the cusp of the fluid domain when $\varepsilon=0$.

We can now claim:
\begin{prop}
\label{prop:flat_case}
Theorem~\ref{THEO:bad_case} is true for the flat case. In particular, the Dirichlet energy behaves as follows when $\varepsilon$ goes to 0:
\begin{equation}
\label{equi_flat}
E_\varepsilon\underset{\varepsilon=0}{\sim}\frac{|\delta'|^3}{3}\varepsilon^{-1}.
\end{equation}
\end{prop}
\begin{proof}
It suffices to verify that 
Lemma~\ref{LEM:conv_f_0} and more precisely that the expressions \eqref{les_trucs} are in the appropriate function spaces. There is a subtlety here because the decay 
properties \eqref{estim_phi_esp} are not true in our case. However, with \eqref{speer} and  \eqref{les_trucs} and since there is at least one derivative of $H_0$ in every product arising in the right hand side 
of the expressions \eqref{les_trucs}, it can be verify that we still get enough decay rate to get the conclusion. 
\end{proof}

Considering the problem of collision, we can apply Lemma~\ref{lem_intro} to get:
\begin{prop}
In the flat case, the solid reaches the cavity's bottom in finite time with null velocity (smooth landing case).
\end{prop}
%=========================================
\subsection{Other cases}
%===================================
Combining \eqref{equi_flat} with the results of Section~\ref{sect:preliminary}, we can easily deduce the following:
\begin{itemize}
\item In the situation depicted in the right of Fig.~\ref{fig_pic}, the solid will collide with the outer boundary in finite time with null velocity (smooth landing case).
\item In both configurations in the right of Fig.\ref{fig2}, the solid will behave the same way when approaching the outer boundary and reach it in finite time. It is not possible (without further computations) to determine 
wether the velocity is null or not at the touching time.
\end{itemize} 
%\section{Miscellaneous collision problems}
%%===================================
%\begin{figure}[h]
%\centering
%%\def\svgscale{0.3}
%\input{miscellanious.tex}
%\caption{\label{fig:chg_of_vari}The domain $\Omega_\varepsilon$  and its image $\omega_\varepsilon$ by $\Psi_\varepsilon$. In particular, $D:=\Psi_\varepsilon( \Omega_\varepsilon\setminus{\mathcal V}_\varepsilon)$ 
%does not depend on $\varepsilon\geqslant 0$.}
%\end{figure}

%%%%%%%%%%%%%%%%%%%%%%%%%%%%%%%%
\appendix
\section{Construction of the change of variables}
\label{App:chg_of_vari}
%===================================
In this Section, we aim to construct the diffeomorphisms $\Psi_\varepsilon$ ($0\leqslant \varepsilon\leqslant \varepsilon^\ast$) earlier introduced in Section~\ref{sect:Psi}. 

Recall the definitions \eqref{def:mathcalV} of $\mathcal V_\varepsilon$ ($0\leqslant \varepsilon\leqslant \varepsilon^\ast$). 
Since we need to emphasize the dependance in $\delta$, we denote it rather $\mathcal V_\varepsilon(\delta)$ in this Section and we 
recall that $\Psi_\varepsilon$ has already been defined in $\mathcal V_\varepsilon(\delta)$ in Section~\ref{sect:Psi}. 

To simplify the construction, there is no loss of generality in assuming that (up to a rescaling) $H_0(\delta)=1$. 

Let us define $\hat{\Psi}_\varepsilon:=\Psi_\varepsilon+\delta e_1$ and notice that, on the left vertical boundary of $\mathcal V_\varepsilon(\delta)$, we have:
$$\hat\Psi((\delta,\xi_2))=\left(\delta,\frac{\xi_2}{1+\varepsilon}\right),\qquad (0<\xi_2<1+\varepsilon).$$
Starting from this observation, our leading idea is to extend $\hat\Psi_\varepsilon$ (rather than $\Psi_\varepsilon$) as a perturbation of the identity in $\Omega_\varepsilon\setminus\overline{\mathcal V}_\varepsilon(\delta)$.
We proceed in several steps.\\

\noindent{\bf First step. }We introduce an open cover $(U_k)_{1\leqslant k\leqslant 3}$ of $\Omega_\varepsilon$ (see Figure~\ref{patch}), such that:
\begin{enumerate}
\item $\Omega_\varepsilon\subset U_1\cup U_2\cup U_3$ for every $\varepsilon\geqslant 0$ small enough;
\item There exists $\delta''<\delta'<\delta$ such that:
$$\mathcal V_\varepsilon(\delta')\subset U_1\subset \mathcal V_\varepsilon(\delta'')\quad\text{and}\quad U_k\cap \mathcal V_\varepsilon(\delta')=\varnothing\quad(k=2,3),$$
for every $\varepsilon\geqslant 0$ small enough;
\item $\Gamma_\varepsilon\cap U_3=\varnothing$ for every $\varepsilon\geqslant 0$ small enough.
\end{enumerate}
Consider $(\chi_k)_{1\leqslant k\leqslant 3}$ a partition of unity subordinated to the open cover $U_k$ ($k=1,2,3$) and let us define $\hat\Psi_\varepsilon^k$ ($k=1,2,3$) three functions respectively defined in $U_1$, $U_2$ 
and $U_3$ and out of which we are going to build $\hat\Psi_\varepsilon$.

\begin{figure}[h]
\centering
\input{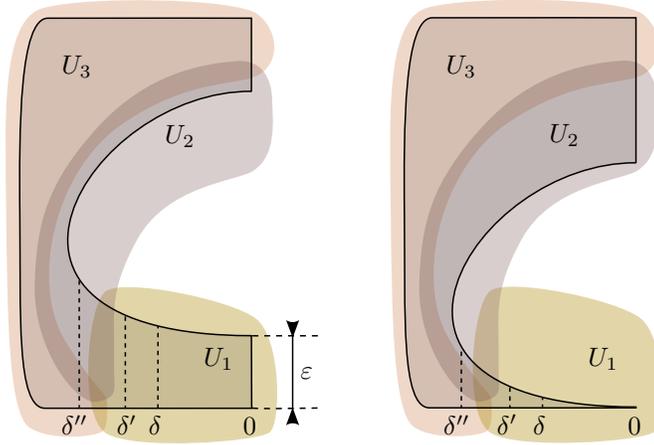}
\caption{\label{patch}The open cover $(U_k)_{1\leqslant k\leqslant 3}$ of $\Omega_\varepsilon$  (for $\varepsilon>0$ and $\varepsilon=0$).}
\end{figure}

\noindent{\bf Step 2. }(Construction of $\hat\Psi_\varepsilon^1$). For every $\varepsilon\geqslant 0$, we define the function
\begin{subequations}
\label{new_def_rho}
\begin{equation}
\hat\rho_\varepsilon(\xi_1)=\delta+\int_\delta^{\xi_1} F_\varepsilon(t)\,{\rm d}t\qquad (\delta''<\xi_1<0),
\end{equation}
where the function $F_\varepsilon\in C^0([\delta'',0[)$ is given by:
\begin{equation}
F_\varepsilon(t)=\begin{cases}
1&\text{if }\xi_1<\delta'\\
\frac{\varepsilon}{1+\varepsilon}\left(a\,t^2+b\,t+c\right)+1&\text{if }\delta'<\xi_1\leqslant \delta\\
H_\varepsilon(t)^{-1}&\text{if }\xi_1>\delta,
\end{cases}
\end{equation}
with 
\begin{equation}
a=-{\frac {3}{ \left( \delta-
\delta' \right) ^{2}}},\quad
b={\frac {2\, \left( \delta+2\,\delta'
 \right) }{ \left( \delta-\delta' \right) ^{2} }},\quad 
c=-\frac{\,\delta'(2\delta+\delta')}{(\delta-\delta')^2}.
\end{equation}
\end{subequations}
We set $\hat\ell_\varepsilon:=\hat\rho_\varepsilon(0)$ ($\varepsilon>0$) and $\hat\ell_0:=+\infty$.

For $\varepsilon\geqslant 0$ small enough, $F_\varepsilon$ is positive and we denote $\hat\mu_\varepsilon:]\delta'',\hat\ell_\varepsilon[\to ]\delta'',0[$ the 
inverse of $\hat\rho_\varepsilon$. 

Define now a $C^1$ function $\hat H$ on $]\delta'',0]$, as on Figure~\ref{pict:phi}, satisfying in particular:
$$\hat H_0(\xi_1)=
\begin{cases}
H_0(\xi_1)&\text{if }\xi_1>\delta\\
1&\text{if }\xi_1<\delta'.
\end{cases}$$

\begin{figure}[h]
\centering
\input{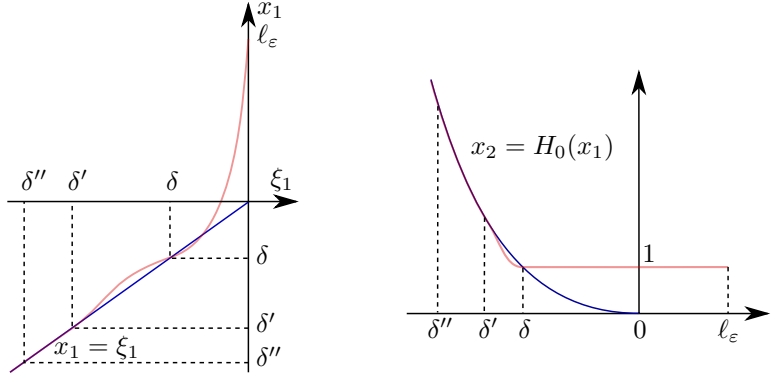}
\caption{\label{pict:phi} Graph of the functions $\hat\rho_\varepsilon$ (left) and $\hat H_0$ (right).}
\end{figure}

The function $\hat\Psi^1_\varepsilon$  is defined on $\mathcal V_\varepsilon(\delta'')$ as follows:
\begin{subequations}
\label{new_def_psi_1}
\begin{equation}
\Psi_\varepsilon^1(\xi)=
\begin{pmatrix}\hat\rho_\varepsilon(\xi_1)\\
\kappa_\varepsilon(\xi)\xi_2-\varepsilon\chi_2(\xi)
\end{pmatrix}
\end{equation}
where 
\begin{equation}
\kappa_\varepsilon(\xi)=\frac{\hat H_0(\hat\rho_\varepsilon(\xi_1))+\varepsilon\chi_2(\xi)}{H_\varepsilon(\xi_1)}.
\end{equation}
\end{subequations}

\noindent{\bf Step 3. }We introduce $\hat\Psi^2_\varepsilon={\rm Id}-\varepsilon e_2$ and $\hat\Psi^3_\varepsilon={\rm Id}$ ($\varepsilon\geqslant0$)
and we claim that the function $\hat\Psi_\varepsilon$ defined by:
$$\hat\Psi_\varepsilon:=\sum_{k=1}^3\chi_k \hat\Psi^k_\varepsilon,$$
fulfilled the requirements. More precisely, setting $\Psi_\varepsilon=\hat\Psi_\varepsilon-\delta e_1$ ($\varepsilon\geqslant0$), we have:

\begin{prop}
\label{prop:suite}
The function $\Psi_\varepsilon$ enjoys the following properties:
\begin{enumerate}
\item For every $\varepsilon\geqslant0$, the set $D:=\Psi_\varepsilon(\Omega_\varepsilon\setminus \overline{\mathcal V}_\varepsilon(\delta))$ is independent of $\varepsilon$. 
\item For every $\varepsilon\geqslant0$, $\Psi_\varepsilon$ is $C^1$, invertible and $\Psi_\varepsilon^{-1}$ is $C^1$.
\item $\|\Psi_\varepsilon^{-1}-\Psi_0^{-1}\|_{C^1(\overline{D})}\to 0$ as $\varepsilon\to 0^+$.
\end{enumerate}
\end{prop}
\begin{proof}
Since $\hat\Psi_\varepsilon$ and $\Psi_\varepsilon$ only differ in a translation, the proof is carried out with $\hat\Psi_\varepsilon$ instead of $\Psi_\varepsilon$.

The first point is easily verified by computing the image of the boundary of $\Omega_\varepsilon\setminus \overline{\mathcal V}_\varepsilon(\delta)$. Indeed, denoting 
$$\hat{\mathcal V}_\varepsilon:=\{(x_1,x_2)\in\Omega_\varepsilon\,:\,\delta''<x_1<\delta,\,0<\xi_2<H_\varepsilon(\xi_1)\},$$
we have:
\begin{align*}
\hat\Psi_\varepsilon(\Gamma_\varepsilon\cap \hat{\mathcal V}_\varepsilon)&=\{x=(x_1,x_2)\in \mathbb R^2\,:\,\delta''<\,x_1<\delta,\,x_2=\hat H_0(x_1)\}\\
\hat\Psi_\varepsilon(\Gamma_\varepsilon\cap U_2)&=\Gamma\cap U_2\\
\hat\Psi_\varepsilon(\Gamma'_\varepsilon \cap (U_3\cup \hat{\mathcal V}_\varepsilon))&=\Gamma'\cap (U_3\cup \hat{\mathcal V}_\varepsilon).
\end{align*}

For the second and third points of the Proposition, we proceed  as follows: 
%================
\begin{enumerate}
\renewcommand{\theenumi}{\alph{enumi}}
\renewcommand{\labelenumi}{\theenumi}
\item We prove that $\hat\Psi_\varepsilon:\mathcal V_\varepsilon(\delta')\mapsto R'_\varepsilon$, where:
$$R_\varepsilon':=\{(x_1,x_2)\in\mathbb R^2\,:\,\delta'<x_1<\ell_\varepsilon,\, 0<x_2<\hat H_0(x_1)\},$$
is a $C^1$ diffeomorphism by studying its inverse, which can be made explicit. With the expression of this inverse, we also prove rather easily that
$\|\hat \Psi_\varepsilon^{-1}-\hat\Psi_0^{-1}\|_{C^1(\overline R'_\varepsilon)}\to 0$ as $\varepsilon\to 0^+$. 
%================
\item By noticing that $\hat\Psi_\varepsilon$ is a $C^1$, $\varepsilon$-perturbation of the identity in $\Omega_\varepsilon\setminus \mathcal V_\varepsilon(\delta')$, we prove that 
$\hat\Psi_\varepsilon$ is also a $C^1$ diffeomorphism from $\Omega_\varepsilon\setminus \mathcal V_\varepsilon(\delta')$ onto its image and that
$\|\hat \Psi_\varepsilon^{-1}-\hat\Psi_0^{-1}\|_{C^1(\overline{\Psi_\varepsilon(\Omega_\varepsilon\setminus \mathcal V_\varepsilon(\delta')}))}\to 0$ as $\varepsilon\to 0^+$. 
\item We get the conclusion of the Proposition by remarking that 
\begin{equation}
\label{local_diffeo}
\det D\hat\Psi_\varepsilon\neq 0\quad\text{in }\Omega_\varepsilon,
\end{equation}
and 
\begin{equation}
\label{non_overlapping}
\hat\Psi_\varepsilon(\mathcal V_\varepsilon(\delta'))\cap \hat\Psi_\varepsilon(\Omega_\varepsilon\setminus \mathcal V_\varepsilon(\delta'))=\varnothing.
\end{equation}
\end{enumerate}
%=================
The inverse of $\hat\Psi_\varepsilon$ in $\mathcal V_\varepsilon(\delta')$ is:
\begin{equation}
\label{express:hatPsi}
\hat\Psi_\varepsilon^{-1}(x)=
\begin{pmatrix}
\hat\mu_\varepsilon(x_1)\\
\frac{H_\varepsilon(\hat\mu_\varepsilon(x_1))}{\hat H_0(x_1)}x_2
\end{pmatrix}.
\end{equation}
Following the lines of the proof of Proposition~\ref{prop:chgt_variables}, we verify that $\hat\Psi_\varepsilon:\mathcal V_\varepsilon(\delta')\mapsto R'_\varepsilon$
is indeed a $C^1$ diffeomorphism and that $\|\hat\Psi_\varepsilon^{-1}-\hat\Psi^{-1}_0\|_{C^1(\overline R'_\varepsilon)}\to 0$ as $\varepsilon\to 0^+$. 

Let us denote now $\mathcal U_\varepsilon:=\Omega_\varepsilon\setminus{\mathcal V}_\varepsilon(\delta')$. With our construction, we get, for every $\xi\in \mathcal U_\varepsilon$:
\begin{equation}
\label{eq:express_psi}
\hat\Psi_\varepsilon(\xi)=\xi+\varepsilon F_\varepsilon(\xi),
\end{equation}
where:
$$
F_\varepsilon(\xi):=\left[\chi_1(\xi)\frac{\chi_2(\xi)-1}{H_\varepsilon(\xi_1)}\xi_2-\chi_2(\xi)-\chi_1(\xi)\chi_2(\xi)\right].
$$
Since $F_\varepsilon$ is lipschitz continuous uniformly in $\varepsilon$ for every $\varepsilon$ small enough, we deduce that $\hat\Psi_\varepsilon$ is one-to-one in $\mathcal U_\varepsilon$ for every $\varepsilon$ small enough.
Finally $\hat\Psi_\varepsilon$ is a bijection from $\mathcal U_\varepsilon$ onto its image. From expression \eqref{eq:express_psi}, according to the local inversion theorem, we get that $\hat\Psi_\varepsilon^{-1}$ is also $C^1$. Using once again \eqref{eq:express_psi}, it is clear that 
$\|\hat\Psi_\varepsilon-\hat\Psi_0\|_{C^1(\overline{\mathcal U}_\varepsilon)}\to 0$ as $\varepsilon\to 0^+$ with $\hat\Psi_0={\rm Id}$.

It remains to address the convergence of $\hat\Psi_\varepsilon^{-1}$ into $\hat\Psi_0^{-1}={\rm Id}$ in $\hat\Psi(\overline{\mathcal U}_\varepsilon)$. On the one hand, we have:
\begin{equation}
\label{eq:Psi-1}
\|\hat\Psi_\varepsilon^{-1}-{\rm Id}\|_{L^\infty(\hat\Psi_\varepsilon(\overline{\mathcal U}_\varepsilon)}=\|\hat\Psi_\varepsilon-{\rm Id}\|_{L^\infty(\overline{\mathcal U}_\varepsilon)}=
\varepsilon\|F_\varepsilon\|_{L^\infty(\overline{\mathcal U}_\varepsilon)}.
\end{equation}
On the other hand, we have:
\begin{equation}
\label{identi_perturb}
D\hat\Psi_\varepsilon(\xi) = {\rm Id} + \varepsilon DF_\varepsilon(\xi)
\end{equation}
and thus, since $DF_\varepsilon$ is clearly uniformly bounded with respect to $\varepsilon$  in $L^\infty(\overline{\mathcal U}_\varepsilon)$ by some constant $C$, its inverse is given via the Neumann series 
$$(D\hat\Psi_\varepsilon(\xi))^{-1}={\rm Id}+\varepsilon G_\varepsilon(\xi),$$
with 
$$G_\varepsilon(\xi)=-\sum_{k=0}^{+\infty}(-\varepsilon)^k DF_\varepsilon(\xi)^{k+1}$$
provided $\varepsilon$ is small enough, or more precisely for $\|\varepsilon DF_\varepsilon \|_{{L^\infty(\overline{\mathcal U}_\varepsilon)}}<1$.
For such $\varepsilon$, we can write that:
\begin{align}
\|D\hat\Psi_\varepsilon^{-1}-{\rm Id}\|_{L^\infty(\hat\Psi_\varepsilon(\overline{\mathcal U}_\varepsilon)}&=\|(D\hat\Psi_\varepsilon\circ\hat\Psi_\varepsilon^{-1})^{-1}-{\rm Id}\|_{L^\infty(\hat\Psi_\varepsilon(\overline{\mathcal U}_\varepsilon)}\nonumber\\
&=\|(D\hat\Psi_\varepsilon)^{-1}-{\rm Id}\|_{L^\infty(\overline{\mathcal U}_\varepsilon)}\nonumber\\
&=\varepsilon\|G_{\varepsilon}\|_{L^\infty(\overline{\mathcal U}_\varepsilon)}\nonumber\\
&\leq\varepsilon\|DF_\varepsilon\|_{L^\infty(\overline{\mathcal U}_\varepsilon)}\nonumber\\
&\leqslant C \varepsilon.\label{eq:dpsi}
\end{align}
Gathering the estimates \eqref{eq:Psi-1} and \eqref{eq:dpsi}, we finally get:
$$\|\hat\Psi_\varepsilon^{-1}-\hat\Psi_0^{-1}\|_{C^1(\hat\Psi_\varepsilon(\overline{\mathcal U}_\varepsilon))}\to 0\text{ as }\varepsilon\to 0^+.$$

We prove \eqref{local_diffeo} by direct computation, using \eqref{express:hatPsi} for $\xi\in\mathcal V_\varepsilon(\delta')$ and \eqref{identi_perturb} for $\xi\in \Omega_\varepsilon\setminus \mathcal V_\varepsilon(\delta')$. Notice in particular 
that \eqref{identi_perturb} remains true ``up to the boundary'' between  $\mathcal V_\varepsilon(\delta')$ and $\Omega_\varepsilon\setminus \mathcal V_\varepsilon(\delta')$. The non-overlapping property \eqref{non_overlapping} is easily verified and 
the proof is now complete.
\end{proof}
%%%%%%%%%%%%%%%%%%%%%%%%%%%%%%%%%%%%%%%%%%%%%%%%%%%%%%%%%%%%%%%%%%%%%%%
%%%%%%%%%%%%%%%%%%%%%%%%%%%%%%%%%%%%%%%%%%%%%%%%%%%%%%%%%%%%%%%%%%%%%%%
\section{Proofs of technical results}
\label{SEC:technic}
\begin{proof}[of Lemma~\ref{LEM:proper_mu}]
One easily checks that the function $F_\varepsilon:=\mu_\varepsilon-\mu_{0}$ is positive on $[0,\ell_\varepsilon]$, which leads to:
$$
|\mu_\varepsilon|\leqslant |\mu_0|,
$$
and \eqref{estim_phi_esp:0} follows from \eqref{def_mu0}.
Moreover, we have $F_\varepsilon'=H_0(\mu_{0})[G_\varepsilon-1]$ where 
$$G_\varepsilon:=\frac{H_\varepsilon(\mu_\varepsilon)}{H_0(\mu_{0})}.$$
Straightforward computations lead to:
$$G'_\varepsilon=G_\varepsilon[H_0'(\mu_\varepsilon)-H_0'(\mu_{0})],$$
which is a positive function on $[0,\ell_\varepsilon]$. Since $G_\varepsilon(0)=1+\varepsilon/H_0(\delta)>1$, we deduce that $F'_\varepsilon>0$ and then that:
$$\|\mu_\varepsilon-\mu_{0}\|_{C^0([0,\ell_\varepsilon])}=-\mu_{0}(\ell_\varepsilon),$$
and \eqref{unif_conv} follows with \eqref{def_mu0} and \eqref{comp_l_eps}.

Still from \eqref{def_mu0} and \eqref{comp_l_eps}, we infer that:
$$G_\varepsilon(\ell_\varepsilon)\underset{\varepsilon=0}{\sim}\left[\alpha\frac{\pi/(\alpha+1)}{\sin(\pi/(\alpha+1))}\right]^{1+\frac{1}{\alpha}},$$
whence we deduce that, on $[0,\ell_\varepsilon]$, we have:
$$|H_\varepsilon(\mu_\varepsilon)|\leqslant CH_0(\mu_0)$$
for some constant $C>0$ dependent on $\alpha$ only. Combining this estimate again with \eqref{def_mu0} and \eqref{comp_l_eps}, we get \eqref{estim_phi_esp:1}.

Finally, since $|\mu_\varepsilon|\leqslant |\mu_0|$, we have
$$|H'_0(\mu_\varepsilon)|\leqslant |H_0'(\mu_0)|,$$
and \eqref{estim_phi_esp:2} follows, using \eqref{def_mu0} and \eqref{comp_l_eps}. The proof is now complete.
\end{proof}
%---------------------------------------------------------------
\begin{proof}[of Lemma~\ref{lem_intro}]
Let us recall that the Cauchy problem \eqref{cauchy:pb} we are dealing with can be rewritten as:
\begin{subequations}
\label{cauchy:pb_2}
\begin{alignat}{3}
\varepsilon'(t)&=\varepsilon_0'F(\varepsilon(t),\varepsilon^\ast),&\quad& t>0\\
\varepsilon(t)|_{t=0}&=\varepsilon^\ast,
\end{alignat}
\end{subequations}
with
$$F(\varepsilon,\varepsilon^\ast):=\sqrt{\frac{m_s+m_f(\varepsilon^\ast)}{m_s+m_f(\varepsilon)}},$$
and $\varepsilon^\ast>0$ and $\varepsilon_0'<0$ are given.
As already mentioned, it is proved in \cite{Chambrion:2012aa} that the function 
$$\varepsilon\in ]0,\varepsilon^*[\mapsto m_f(\varepsilon)\in\mathbb R^+$$ 
is analytic and hence the function $F(\cdot,\varepsilon^\ast):]0,\varepsilon^*[\to \mathbb R^+$
has the same regularity.

The hypothesis $E_\varepsilon\to +\infty$ as $\varepsilon\to 0^+$ entails that $F(\cdot,\varepsilon^\ast)\to 0$ as $\varepsilon\to 0^+$ and therefore that 
the velocity of the solid tends to $0$ when approaching the outer boundary: real shock can not occur in this case. 

The hypothesis $E_\varepsilon=O(\varepsilon^{-\beta})$ means that there exists $0<\varepsilon^\dagger<\varepsilon^\ast$ and $C>0$ such that 
$$E_\varepsilon<C \varepsilon^{-\beta}\quad (0<\varepsilon<\varepsilon^\dagger).$$
This estimate entails that, for a different positive constant still denoted by $C$:
\begin{equation}
\label{eq:estttt}
F(\varepsilon,\varepsilon^\ast)> C\varepsilon^{\frac{\beta}{2}}\quad (0<\varepsilon<\varepsilon^\dagger).
\end{equation}
On the other hand, the function $F(\cdot,\varepsilon^\ast)$ is bounded from below on the compact $[\varepsilon^\dagger,\varepsilon^\ast]$ by some constant $c>0$. We deduce that $\varepsilon'(t)<c\varepsilon'_0$ 
as long as $\varepsilon(t)>\varepsilon^\dagger$. To simplify, let us relabel $t=0$ the time for which $\varepsilon(t)=\varepsilon^\dagger$ (this time being no greater than $\varepsilon_\eta/(c\varepsilon'_0$)).

Using now the estimate \eqref{eq:estttt} in the Cauchy problem, now restated as:
\begin{alignat*}{3}
\varepsilon'(t)&=\varepsilon_0'F(\varepsilon(t),\varepsilon^\ast),&\quad& t>0\\
\varepsilon(t)|_{t=0}&=\varepsilon^\dagger,
\end{alignat*}
we deduce that (changing again the value of the positive constant $C$):
$$\varepsilon(t)\leqslant \left[C\varepsilon_0't+(\varepsilon^\dagger)^{-\frac{\beta}{2}+1}\right]^{\frac{2}{2-\beta}},$$
and therefore, the solid meets the cavity's wall in finite time. 
\end{proof}
%-------------------------------------------------------
\begin{proof}[of Lemma~\ref{LEM:w_var}]
The Dirichlet energy of $u_\varepsilon^s$ is decomposed as follows:
$$\int_{\omega_\varepsilon}|\nabla u_\varepsilon^s(x)|^2\,{\rm d}x=\int_{\omega_\varepsilon\setminus]1,\ell_\varepsilon[\times]0,1[}|\nabla u_\varepsilon^s(x)|^2\,{\rm d}x+\int_1^{\ell_\varepsilon}\int_0^1|\nabla u_\varepsilon^s(x_1,x_2)|^2\,{\rm d}x_2{\rm d}x_1.$$
The first integral in the right hand side is uniformly bounded for $\varepsilon\geqslant 0$, so let us focus on the latter. 

In $]1,\ell_\varepsilon[\times]0,1[$, we have:
\begin{align*}
\partial_{x_1} u_\varepsilon^s(x)&=-\mu_\varepsilon(x_1)\left[1+\frac{1}{2}x_2^2 H_0''(\mu_\varepsilon(x_1))H_\varepsilon(\mu_\varepsilon(x_1))\right]\\
\partial_{x_2} u_\varepsilon^s(x)&=x_2\left[H_\varepsilon(\mu_\varepsilon(x_1))-\mu_\varepsilon(x_1)H_0'(\mu_\varepsilon(x_1))\right].
\end{align*}
Expending $|\partial_{x_1} u_\varepsilon^s(x)|^2+|\partial_{x_2} u_\varepsilon^s(x)|^2$, and using the estimates \eqref{estim_phi_esp}, we obtain that the only remaining term which is not uniformly bounded (with respect to $\varepsilon\geqslant 0$) by a  function in $L^1(\omega_0)$ is:
$$\int_1^{\ell_\varepsilon}\int_0^1|\mu_\varepsilon(x_1)|^2\,{\rm d}x_2{\rm d}x_1=\int_1^{\ell_\varepsilon}|\mu_\varepsilon(x_1)|^2\,{\rm d}x_1.$$
The change of variables $\xi_1=\mu_\varepsilon(x_1)$ leads to:
$$\int_1^{\ell_\varepsilon}|\mu_\varepsilon(x_1)|^2\,{\rm d}x_1=\int_{\mu_\varepsilon(1)}^0\frac{\xi_1^2\,{\rm d}\xi_1}{\kappa |\xi_1|^{1+\alpha}+\varepsilon}.$$
Setting now $\zeta=\kappa^{\frac{1}{\alpha+1}}\xi_1/\varepsilon^{\frac{1}{\alpha+1}}$, we can transform the expression above into:
$$
\int_{\mu_\varepsilon(1)}^0\frac{\xi_1^2\,{\rm d}\xi_1}{\kappa |\xi_1|^{1+\alpha}+\varepsilon}=\kappa^{-\frac{3}{\alpha+1}}\varepsilon^{\frac{3}{1+\alpha}-1}\int^0_{\beta(\varepsilon)}\frac{\zeta^2\,{\rm d}\zeta}{|\zeta|^{1+\alpha}+1},
$$
where $\beta(\varepsilon):=\kappa^{-\frac{1}{1+\alpha}}\mu_\varepsilon(1)\varepsilon^{-\frac{1}{1+\alpha}}$. We have now to distinguish, according to the value of $\alpha$:
\begin{equation}
\label{second}
\int^0_{\beta(\varepsilon)}\frac{\zeta^2\,{\rm d}\zeta}{|\zeta|^{1+\alpha}+1}=
\begin{cases}
\int_{\beta(\varepsilon)}^0\frac{{\rm d}\zeta}{|\zeta|+1}+\int_{\beta(\varepsilon)}^0\frac{|\zeta|- 1}{|\zeta|^3+1}\,{\rm d}\zeta&\text{if }\alpha=2\\
\int^0_{-\infty}\frac{\zeta^2\,{\rm d}\zeta}{|\zeta|^{1+\alpha}+1}-\int^{\beta(\varepsilon)}_{-\infty}\frac{\zeta^2\,{\rm d}\zeta}{|\zeta|^{1+\alpha}+1}&\text{if }\alpha>2.
\end{cases}
\end{equation}
We deduce that, when $\alpha=2$:
$$\int^0_{\beta(\varepsilon)}\frac{\zeta^2\,{\rm d}\zeta}{|\zeta|^{1+\alpha}+1}\underset{\varepsilon=0}{\sim}|\ln|\beta(\varepsilon)||\underset{\varepsilon=0}{\sim}\frac{1}{3}|\ln(\varepsilon)|.$$
When $\alpha>2$, since the second term in \eqref{second} tends to 0 with $\varepsilon$, we are led to compute the value of:
$$\int_0^{+\infty}\frac{x^2\,{\rm d}x}{x^{1+\alpha}+1}.$$
Yet another change of variable, namely $y=x^3$, allows us to do that. We finally get:
$$ \int_0^{+\infty}\frac{x^2\,{\rm d}x}{x^{1+\alpha}+1}=\frac{1}{3}\int_0^{+\infty}\frac{{\rm d}y}{y^{\frac{1+\alpha}{3}}+1}=\frac{1}{3}\frac{3\pi/(1+\alpha)}{\sin(3\pi/(1+\alpha)}.$$
The proof is then complete.
\end{proof}
%========================================
\begin{proof}[of Lemma~\ref{LEM:conv_f_0}]

We observe first that $f^r_\varepsilon=0$ in $D$ and $g_\varepsilon^r=0$ in $\gamma_\varepsilon^D$ for every $\varepsilon\geqslant 0$ because of the cut-off function $\chi$ and hence 
$\omega_\varepsilon$  can be replaced by $R_\varepsilon$ and $\gamma_\varepsilon$ by $\gamma_\varepsilon^R$ in \eqref{eq:converg_f_r}. Recall that (identities \eqref{def_f_r2} and \eqref{def_f_r1}):
\begin{subequations}
\label{def_f_r3}
\begin{alignat}{3}
f^r_\varepsilon&=\left[\dv(\mathbb A_\varepsilon\nabla \chi)\hat u_\varepsilon^s+2\mathbb A_\varepsilon:(\nabla\chi\otimes\nabla \hat u_\varepsilon^s)\right]-\chi \hat f^s_\varepsilon&\quad&\text{in }\omega_\varepsilon\\
g^r_\varepsilon&=\left[-\hat u_\varepsilon^s(\mathbb A_\varepsilon\nabla\chi)\cdot n+(1-\chi)g_\varepsilon\right]-\chi \hat r_\varepsilon^s&&\text{on }\gamma_\varepsilon.
\end{alignat}
\end{subequations}
Using the expressions \eqref{def:v} and \eqref{def:V} of $v_\varepsilon$ and $V_\varepsilon$ in \eqref{def:hatf} and \eqref{def:hatr}, we get, for every $x\in R_\varepsilon$:
\begin{subequations}
\label{les_trucs}
\begin{multline}
\hat f^s_\varepsilon(x_1,x_2)=x_2^2\left[\frac{3}{2}H_\varepsilon(X_1^\varepsilon)^2{H_0''}(X_1^\varepsilon)+
\frac{1}{2}X_1^\varepsilon H_\varepsilon(X_1^\varepsilon)^2H_0'''(X_1^\varepsilon)\right.\\
\left.-3X_1^\varepsilon H_\varepsilon(X_1^\varepsilon){H'_0}(X_1^\varepsilon){H_0''}(X_1^\varepsilon)-3H_\varepsilon(X_1^\varepsilon){H'_0}(X_1^\varepsilon)^2+3X_1^\varepsilon 
{H'_0}(X_1^\varepsilon)^2\right],
\end{multline}
and for every $x\in \gamma^R_\varepsilon$:
\begin{equation}
\hat r^s_\varepsilon(x)=\left[\frac{1}{2}X_1^\varepsilon H_\varepsilon(X_1^\varepsilon){H'_0}(X_1^\varepsilon){H_0''}(X_1^\varepsilon)+H_\varepsilon(X_1^\varepsilon){H'_0}(X_1^\varepsilon)^2
-X_1^\varepsilon{H'_0}(X_1^\varepsilon)^3\right].
\end{equation}
\end{subequations}
Observing that $|{H_0''}(\mu_\varepsilon)|$ and $|H_0'''(\mu_\varepsilon)|$ are uniformly bounded (because $\alpha>2$), we deduce, according to the estimates 
\eqref{estim_phi_esp}, that the functions:
$$x\mapsto |\hat f^s_\varepsilon(x)|^2(1+x_1)^2\quad\text{and}\quad x\mapsto |\hat r^s_\varepsilon(x)|^2(1+x_1)^2,$$ are uniformly (in $\varepsilon\geqslant 0$) bounded by a function belonging to $L^1(R_0)$ and $L^1(\gamma_0^R)$ respectively.
Invoking again Lemma~\ref{LEM:proper_mu} and applying the dominated convergence theorem, we get that: 
$$\|\hat f^s_\varepsilon-\hat f^s_0\|_{L^2(R_\varepsilon,{\rm d}m^{-1})}\to 0\quad\text{ and }\quad\|\hat r^s_\varepsilon-\hat r^s_0\|_{L^2(\gamma^R_\varepsilon,\,{\rm d}\sigma^{-1})}\to 0\text{ as }\varepsilon\to 0^+.$$

Using Proposition~\ref{prop:chgt_variables}, we deduce that $-\dv(\mathbb A_\varepsilon\nabla \chi)$ converges uniformly in $]0,1[\times ]0,1[$ to 
$-\dv(\mathbb A_0\nabla \chi)$ and since, in addition, $\hat u_\varepsilon^s$ converges 
in $C^1([0,1]^2)$ to $\hat u^s_0$,  we get, according to the identities \eqref{def_f_r3} that 
$$\|f^r_\varepsilon-f^r_0\|_{L^2(R_\varepsilon,{\rm d}m)}\to 0\text{ as }\varepsilon\to 0^+.$$
We proceed similarly to show the second convergence result.

We know (see above Lemma\ref{LEM:conv_f_0}) that $f_\varepsilon^r$ and $g_\varepsilon^r$ satisfy the compatibility condition \eqref{comp:cond} for every $\varepsilon>0$. 
Lemma~\ref{lem:conv_l1} and the convergence results above allow us to pass to limit in \eqref{comp:cond}. The proof is now complete.
\end{proof}

%%%%%%%%%%%%%%%%%%%%%%%%%%%%%%%%%%%%%%%%%%%%%%%%%%%%%%%%%%%%%%%%%%%%%%%
\bibliographystyle{siam} 
\bibliography{contact_bib}

\end{document}